\documentclass[a4paper,12pt]{article}

\usepackage{amsmath}
\usepackage{multirow,bigdelim}
\usepackage{amscd}
\usepackage{amssymb,amsfonts,amsmath,amsthm}
\usepackage{verbatim}

\usepackage{amsmath}
\usepackage{amsthm}
\usepackage{amssymb}

\usepackage{latexsym}
\usepackage{array}
\usepackage{enumerate}

\numberwithin{equation}{section}

\usepackage{amsmath,amssymb,mathrsfs,amsthm}
\usepackage[all]{xy}
\usepackage{graphicx}
\usepackage{ulem}
\usepackage{ascmac}
\usepackage{mathtools}

\usepackage[alphabetic,initials]{amsrefs}
    \BibSpec{arXiv}{
    +{}{\PrintAuthors}{author}
    +{,}{ \textit}{title}
    +{,}{ }{date}
    +{,}{ arXiv:}{eprint}}
\usepackage{mleftright}
    \mleftright
\usepackage{tikz}
    \usetikzlibrary{cd,decorations,decorations.pathreplacing,arrows,calc,intersections}

\newcommand{\BDC}{{\mathbf{D}}^{\mathrm{b}}}

\newcommand{\Mod}{\mathrm{Mod}}

\newcommand{\sect}{\Gamma}
\newcommand{\rsect}{{\mathrm{R}}\Gamma}
\newcommand{\CC}{\mathbb{C}}

\newcommand{\RR}{\mathbb{R}}
\newcommand{\QQ}{\mathbb{Q}}
\newcommand{\ZZ}{\mathbb{Z}}

\newcommand{\PP}{{\mathbb P}}

\renewcommand{\(}{\left(}
\renewcommand{\)}{\right)}

\newcommand{\Cal}{\cal}

\newcommand{\Ker}{{\rm Ker}}

\newcommand{\rd}{{\rm rd}}
\newcommand{\an}{{\rm an}}

\renewcommand{\dim}{{\rm dim}}

\newcommand{\id}{{\rm id}}

\newcommand{\Int}{\operatorname{Int}}

\newcommand{\grad}{\operatorname{grad}}

\newcommand{\Sol}{{\rm Sol}}

\newcommand{\tl}[1]{\widetilde{#1}}

\newcommand{\simto}{\overset{\sim}{\longrightarrow}}

\newcommand{\op}{\mbox{\scriptsize op}}
\newcommand{\SD}{\mathcal{D}}

\newcommand{\SO}{\mathcal{O}}

\newcommand{\SM}{\mathcal{M}}
\newcommand{\SN}{\mathcal{N}}
\newcommand{\SL}{\mathcal{L}}

\newcommand{\SE}{\mathcal{E}}
\newcommand{\SF}{\mathcal{F}}
\newcommand{\SG}{\mathcal{G}}
\newcommand{\SC}{\mathcal{C}}
\newcommand{\Modcoh}{\mathrm{Mod}_{\mbox{\scriptsize coh}}}
\newcommand{\Modhol}{\mathrm{Mod}_{\mathrm{hol}}}
\newcommand{\Modrh}{\mathrm{Mod}_{\mbox{\scriptsize rh}}}

\newcommand{\BDCcoh}{{\mathbf{D}}^{\mathrm{b}}_{\mbox{\scriptsize coh}}}

\newcommand{\BDChol}{{\mathbf{D}}^{\mathrm{b}}_{\mathrm{hol}}}

\newcommand{\DD}{\mathbb{D}}
\newcommand{\Lotimes}[1]{\overset{L}{\otimes}_{#1}}
\newcommand{\Dotimes}{\overset{D}{\otimes}}

\newcommand{\Potimes}{\overset{+}{\otimes}}

\newcommand{\rhom}{{\rm R}{\mathcal{H}}om}

\newcommand{\Prihom}{{\rm R}{\mathcal{I}}hom^+}
\newcommand{\Prhom}{\rhom^+}
\newcommand{\I}{{\rm I}}

\newcommand{\var}[1]{\overline{#1}}
\newcommand{\BEC}{{\mathbf{E}}^{\mathrm{b}}}

\newcommand{\EE}{\mathbb{E}}

\newcommand{\inj}{``\varinjlim"}

\newcommand{\bfR}{\mathbf{R}}
\newcommand{\bfL}{\mathbf{L}}
\newcommand{\bfD}{\mathbf{D}}
\newcommand{\rmR}{{\mathrm{R}}}
\newcommand{\rmE}{{\mathrm{E}}}

\newcommand{\bfE}{\mathbf{E}}

\renewcommand{\Re}{\operatorname{Re}}

\newcommand{\SP}{\mathcal{P}} 
\newcommand{\SB}{\mathcal{B}}
\newcommand{\sfE}{\mathsf{E}}

\newcommand{\sfL}{\mathsf{L}}

\newcommand{\rmI}{\mathrm{I}}
\newcommand{\bfQ}{\mathbf{Q}}
\newcommand{\pt}{\mathrm{pt}}
\renewcommand{\Ker}{\operatorname{Ker}}
\newcommand{\ord}{\operatorname{ord}}
\newcommand{\dR}{\mathrm{dR}}
\newcommand{\rmb}{\mathrm{b}}

\newcommand{\reg}{\mathrm{reg}}
\newcommand{\irr}{\mathrm{irr}}
\newcommand{\red}{\mathrm{red}}
\DeclareMathOperator{\rk}{rk}
\DeclareMathOperator{\mult}{mult}

\newcommand{\rcEC}{\mathbf{E}_{\mathbb{R}{\text -}\mathrm c}}
\newcommand{\CCirr}{\mathrm{CC}_{\mathrm{irr}}} 

\newcommand{\bs}{\left.\right\backslash} 
\newcommand{\vbar}{\left.\right|} 
\newcommand{\xyoplus}{\bigoplus^{{\phantom{A}}}} 
\DeclareRobustCommand{\longtwoheadrightarrow}{\relbar\joinrel\twoheadrightarrow}

\DeclarePairedDelimiter{\abs}{\lvert}{\rvert} 
\DeclarePairedDelimiter{\inprod}{\langle}{\rangle} 
\DeclarePairedDelimiterX{\Set}[2]{\lbrace}{\rbrace}{#1\ \delimsize\vert\ #2}

\newcommand\blfootnote[1]{%
  \begingroup
  \addtocounter{footnote}{-1}%
  \renewcommand\thefootnote{}\footnote{#1}%
  \endgroup
}

\newtheorem{theorem}{Theorem}[section]

\newtheorem{corollary}[theorem]{Corollary}
\newtheorem{lemma}[theorem]{Lemma}
\newtheorem{proposition}[theorem]{Proposition}

\theoremstyle{definition}
\newtheorem{definition}[theorem]{Definition}
\theoremstyle{remark}
\newtheorem{remark}[theorem]{\sc Remark}
\newtheorem{example}[theorem]{\sc Example}

\title{A Morse theoretical approach to Fourier transforms
of holonomic {$\SD$}-modules in dimension one
}

\author{Kazuki KUDOMI
\footnote{Mathematical Institute, Tohoku University,
Aramaki Aza-Aoba 6-3, Aobaku, Sendai, 980-8578, Japan.
E-mail: kazuki.kudomi.q3@dc.tohoku.ac.jp}
and Kiyoshi TAKEUCHI
\footnote{Mathematical Institute, Tohoku University,
Aramaki Aza-Aoba 6-3, Aobaku, Sendai, 980-8578, Japan.
E-mail: takemicro@nifty.com} }

\begin{document}

\maketitle

\begin{abstract}
We study Fourier transforms of holonomic D-modules on the complex affine
line and show that their enhanced solution complexes are described by a 
twisted Morse theory. We thus recover and even strengthen the well-known 
formula for their exponential factors i.e. the stationary phase method. 
Moreover, we define a Lagrangian cycle that we call the irregular 
characteristic cycle and describe the enhanced solution complex of the 
Fourier transform by it. In this way, we obtain a new perspective, from
which we can geometrically see how the standard properties of holonomic 
D-modules are transformed via the Fourier transform. In the course of our 
study, a formula for the (classical) characteristic cycles of the Fourier
transforms will be also obtained and natural bases of their holomorphic 
solutions will be constructed via rapid decay homology cycles.
\end{abstract}

\section{Introduction}\label{uni-sec:0}

The study of Fourier transforms of $\SD$-modules is an
important subject in both algebraic analysis and algebraic
geometry. Despite its long history, we do not
know so far much of the global properties of the Fourier
transforms, such as the monodromies and the Stokes
matrices of their holomorphic solutions, even in
the simplest case of dimension one (see
for example, D'Agnolo-Hien-Morando-Sabbah
\cite{DHMS20} and Hohl \cite{Hoh22}).
For the results in higher dimensions, see
Brylinski \cite{Bry86}, Daia \cite{Dai00},
Ito-Takeuchi \cite{IT20a}, \cite{IT20b}, 
Kashiwara-Schapira \cite {KS97} 
and Takeuchi \cite{Tak22}. After some
pioneering works by Malgrange in \cite{Mal88} and
\cite{Mal91}, inspired from the theory of
Fourier transforms of $l$-adic sheaves in positive
characteristic, Bloch-Esnault \cite{BE04b} and
Garc{\'i}a L{\'o}pez \cite{Gar04} introduced
independently the so-called local Fourier transforms of
algebraic holonomic
$\SD$-modules $\SM$ on the complex affine line $X= \CC_z$. 
In what follows, we assume that $\SM$ is an algebraic
meromorphic connection i.e. a localized algebraic
holonomic $\SD$-module on $X= \CC_z$. Let 
$Y= \CC_w$ be the dual of $X= \CC_z$ and 
$\var{X}\simeq\PP^1$ (resp. $\var{Y}\simeq\PP^1$)
the projective compactification of $X$ (resp. $Y$). 
Then the Fourier transform
$\SM^\wedge$ of $\SM$ is an algebraic holonomic 
$\SD$-module on $Y= \CC_w$ and the new
method of \cite{BE04b} and \cite{Gar04} 
enables us to describe the formal structure
i.e. the exponential factors of 
$\SM^\wedge$ in terms of that of $\SM$. 
Such an explicit description was obtained by Fang \cite{Fan09}, 
Graham-Squire \cite{Gra13} 
and Sabbah \cite{Sab08}. 
We thus now know that the exponential factors
of $\SM^\wedge$ are obtained by the Legendre
transform from those of $\SM$ 
(for the definition, see Section \ref{legendre}). We call it 
the stationary phase method. 
Subsequently, based on this 
result, Mochizuki \cite{Mochi10}, \cite{Mochi18} gave also
a description of the Stokes structure of $\SM^\wedge$ at infinity. 
Moreover, by using 
the new theories of the irregular Riemann-Hilbert
correspondence established by D'Agnolo-Kashiwara
\cite{DK16} and the enhanced Fourier-Sato transforms
of Kashiwara-Schapira \cite{KS16a} adapted to it,
in the two papers \cite{DK18} and \cite{DK23}
D'Agnolo and Kashiwara reformulated and reproved
the stationary phase method more elegantly. 
For this purpose, in \cite{DK18} they used
some microlocal notions, such as enhanced micro-supports
and multiplicity test functors, to treat the
exponential factors of $\SM$ other than the
linear ones. As we do not see any information of 
the linear exponential factors by the enhanced micro-support, 
in \cite{DK23} they developed a
theory of nearby and vanishing cycles of
enhanced ind-sheaves to treat them separately. 

In this paper, we apply the Morse theoretical method in
Ito-Takeuchi \cite{IT20a} to
give a unified proof to the results in the two
papers \cite{DK18} and \cite{DK23}. Moreover
we obtain not only the exponential factors but also
the enhanced solution complex 
$Sol_{\var{Y}}^{\rmE}(\tl{\SM^\wedge})$ of $\tl{\SM^\wedge}$, where 
by the inclusion map $i_Y : 
Y = \CC_w \xhookrightarrow{\ \ \ }\var{Y}=\PP^1$
we set $\tl{\SM^\wedge}
:= i_{Y\ast}\SM^\wedge \simeq\bfD i_{Y\ast} \SM^\wedge 
\in \Modhol(\SD_{\var{Y}})$. 
Specifically, we describe it 
on a sufficiently small punctured disk centered at 
a singular point of $\SM^\wedge$ or at 
infinity. Note that in \cite{DK18} the authors 
determined only the multiplicities of the 
exponential factors of $\SM^\wedge$ 
via  the ``microlocal shadow" of $\SM^\wedge$ i.e. the 
enhanced micro-support associated to it. We thus 
have upgraded such an indirect proof of the 
stationary phase method to a more concrete and 
manipulatable one, which enables us to extract 
deeper informations of the Fourier 
transform $\SM^\wedge$. For its 
application to the monodromies at infinity of $\SM^\wedge$ 
see \cite{KT24}. Let us explain the outline of our proof. 
Recall that in 
\cite[Proposition 5.4.5]{DK18} a local description
of the enhanced solution complex of
the original $\SD$-module $\SM$ was obtained.
We refine this result to get also a global explicit
description of it (see Section \ref{uni-sec:T1} for the details)
and apply the Morse theoretical method in \cite{IT20a} to it. 
More precisely, by the inclusion map $i_X : 
X=\CC_z \xhookrightarrow{\ \ \ }\var{X}=\PP^1$
we set $\tl{\SM} := i_{X\ast}\SM\simeq\bfD i_{X\ast}\SM
\in\Modhol(\SD_{\var{X}})$ and construct an 
$\RR$-constructible enhanced sheaf 
$G\in\rcEC^0(\CC_{\var{X}^{\an}})$ on 
the underlying complex manifold $\var{X}^{\an}$ of 
$\var{X} = \PP^1$ such that for 
the enhanced solution complex $Sol_{\var{X}}^{\rmE}(\tl{\SM})$ 
of $\tl{\SM}$ we have an isomorphism 
\begin{align}
Sol_{\var{X}}^{\rmE}(\tl{\SM})
\simeq \CC_{\var{X}^{\an}}^{\rmE}\Potimes G.
\end{align}
This construction is important, because for the 
Fourier-Sato (Fourier-Laplace) transform 
${}^\sfL G\in\BEC(\CC_{\var{Y}^{\an}})$ of $G$ 
(see Definition \ref{FS-trans}) 
there exists an isomorphism
\begin{align}
Sol_{\var{Y}}^{\rmE}(\tl{\SM^\wedge})
\simeq \CC_{\var{Y}^{\an}}^{\rmE}\Potimes{}^\sfL G
\end{align}
and hence all the informations of the Fourier 
transform $\SM^\wedge$ of $\SM$ are 
contained in ${}^\sfL G$. Then our remaining 
task is to describe ${}^\sfL G$ and it turns out 
that the stalk of (an $\RR$-constructible sheaf on 
$\var{Y}^{\an} \times \RR$ representing) 
the enhanced sheaf ${}^\sfL G\in\BEC(\CC_{\var{Y}^{\an}})$ 
at a point $(w,t) \in Y^{\an} \times \RR$ is isomorphic to the complex 
\begin{align}
\rsect_c \(X^{\an}; G(w,t) [1]\)
\end{align}
associated to an $\RR$-constructible sheaf 
$G(w,t)$ on $X^{\an}$ (for the definition, 
see Section \ref{sec:T4}). Next, we study 
very carefully how their cohomology groups 
\begin{align}
H^j_c \(X^{\an}; G(w,t) [1]\)
\simeq H^{j+1}_c \(X^{\an}; G(w,t) \) \qquad (j \in \ZZ )
\end{align}
change as $t \in \RR$ increases. This is what we call here 
``the Morse theoretical method" and it leads us to a 
very precise description of the enhanced sheaf ${}^\sfL G$ 
underlying the enhanced solution complex 
$Sol_{\var{Y}}^{\rmE}(\tl{\SM^\wedge})$ of $\tl{\SM^\wedge}$. 
Then as in \cite{Tak22}, by the exponential 
factors of $\SM$ we define a Lagrangian cycle 
$\CCirr(\SM)$ in the cotangent bundle 
$T^*X^{\an}$ of $X^{\an}$ that we call the 
irregular characteristic cycle of $\SM$ and describe 
the enhanced solution complex of $\tl{\SM^\wedge}$ by it 
(see Figure \ref{fig:CCirr} below). In this way, we 
obtain a new perspective of the Fourier transform, 
from which we can geometrically see how the exponential factors of 
$\SM$ and $\SM^\wedge$ are related and 
the criteria for $\SM^\wedge$ to be 
regular or monodromic follow 
immediately (see Corollaries 
\ref{cor:monod} and \ref{cor:regular}). In particular, 
we see that the generic rank ${\rm rk} \SM^\wedge$
of $\SM^\wedge$ is equal to the covering degree 
at infinity of
the restriction of the projection 
$T^*X^{\an} \simeq X^{\an} \times Y^{\an}
\longrightarrow Y^{\an}$ to $\CCirr(\SM)$ 
(see Figure \ref{fig:CCirr} below). 
To the best of our knowledge, we do not find
such a geometric expression of ${\rm rk} \SM^\wedge$
in the literature. Finally, we also obtain a formula 
for the (classical) characteristic cycle 
${\rm CC} ( \SM^\wedge )$ of $\SM^\wedge$, 
which has never been studied successfully before.

In order to introduce our results more precisely, we 
prepare some notations. For the algebraic meromorphic connection  
$\SM\in\Modhol(\SD_X)$ on the affine line $X=\CC_z$ 
we denote its singular support $\mathrm{sing.supp}(\SM)$ 
by $D\subset X$ and set $U:= X\bs D$. 
Then there exists an isomorphism $\SM \simto \sect_U(\SM)$. 
We denote the generic rank of $\SM$ by ${\rm rk} \SM$. 
Define the analytification 
$\tl{\SM}^{\an}\in\Modhol(\SD_{\var{X}^{\an}})$
of $\tl{\SM}$ by
$\tl{\SM}^{\an} = \SO_{\var{X}^{\an}}
\otimes_{\SO_{\var{X}}}\tl{\SM}$. Then for 
the divisor
$\tl{D} :=D^{\an}\sqcup\{\infty\}\subset\var{X}^{\an}$ in 
$\var{X}^{\an}$ 
we obtain an isomorphism 
$(\tl{\SM})^{\an} \simto (\tl{\SM})^{\an}(\ast\tl{D})$. 
We define the enhanced solution complex $Sol_{\var X}^\rmE(\tl\SM)$ 
of $\tl\SM$ by 
\begin{align}
Sol_{\var X}^\rmE(\tl\SM) := 
Sol_{\var{X}^{\an}}^\rmE(\tl{\SM}^{\an})
\qquad \in\BEC(\I\CC_{\var{X}^{\an}})
\end{align}
(see Section \ref{uni-sec:K1} for the details). 
Similarly, for the Fourier transform 
$\SM^\wedge\in\Modhol(\SD_Y)$ of $\SM$ 
we define the enhanced solution complex 
$Sol_{\var Y}^\rmE(\tl{\SM^\wedge})
\in\BEC(\I\CC_{\var{Y}^{\an}})$ of $\tl{\SM^\wedge}$. 
Now our objective here 
is to describe $Sol_{\var Y}^\rmE(\tl{\SM^\wedge})$ 
in terms of $\SM$. 
Let $D^{\an} = \{ a_1,a_2,\dots,a_l \} 
\subset X^{\an}$
and set $a_\infty\coloneq\infty\in\var{X}^{\an}$
so what we have
\begin{align}
\tl{D} = \mathrm{sing.supp}(\tl{\SM})^{\an}
= \{a_1,a_2,\dots,a_l,a_\infty\}.
\end{align}
For a point $a_i\in D^{\an}$ 
and a Puiseux germ $f(z)\in\SP_{S_{a_i}X^{\an}}^\prime$ 
(see Section \ref{sec:K6}  
for the definition) which is holomorphic on a  sector $S \subset 
X^{\an}$ along $a_i$, we define a complex Lagrangian 
submanifold $ \Lambda_i^f$ of 
$T^\ast X^{\an} \simeq X^{\an}\times Y^{\an}=X^{\an}\times\CC_w$ by 
\begin{align}
\Lambda_i^f \coloneq \Set*{(z,f^\prime(z))}{z\in S}
\ \subset T^\ast X^{\an}.
\end{align}
In this paper, we always assume that sectors 
are connected and open. 
For the point $a_i\in D^{\an}$, let
\begin{align}
N_i =N(a_i) \colon \SP_{S_{a_i}X^{\an}}^\prime
\longrightarrow (\ZZ_{\geq 0})_{S_{a_i}X^{\an}}
\end{align}
be the multiplicity for which 
the enhanced solution complex $Sol_{\var X}^\rmE(\tl\SM)= 
Sol_{\var{X}^{\an}}^\rmE(\tl{\SM}^{\an})
 \in\BEC(\I\CC_{\var{X}^{\an}})$ 
of $(\tl{\SM})^{\an}$ has a normal form at it 
(see Definitions \ref{def-mul} and \ref{nf-ind} and Proposition \ref{prop-K4}) 
and set 
\begin{align}
r_i \coloneq N_i(0) \ \in\ZZ_{\geq0}. 
\end{align}
We call $r_i \geq 0$ the regular rank of $(\tl{\SM})^{\an}$ at 
the point $a_i\in D^{\an}$. 
The sections of 
$N_i^{>0}:=N_i^{-1}( \ZZ_{>0} )\subset\SP_{S_{a_i}X}^\prime$ are called 
the exponential factors of 
$(\tl{\SM})^{\an}$ at $a_i$. Assume that 
any exponential factor $f \in N_i^{>0}$ of $(\tl{\SM})^{\an}$ at $a_i$ 
is a (possibly multi-valued) holomorphic function 
on a sufficiently small punctured disk
\begin{align}
B(a_i)^\circ \coloneq
\Set*{z\in X^{\an}=\CC}{0<\abs*{z-a_i}<\varepsilon} 
\quad (0< \varepsilon \ll 1)
\end{align}
centered at the point $a_i$. 
Then the analytic continuation of an 
exponential factor along a loop in $B(a_i)^\circ$ 
is again an exponential factor, 
as the multiplicity $N_i \colon \SP_{S_{a_i}X^{\an}}^\prime
\longrightarrow (\ZZ_{\geq 0})_{S_{a_i}X^{\an}}$ 
is defined on the whole circle $S_{a_i}X^{\an} \simeq S^1$. 
This implies that we can set 
\begin{align}
\CCirr(\SM)_i \coloneq
\left\{\sum_{f\in N_i^{>0}}N_i(f)\cdot[\Lambda_i^f]\right\}
+ r_i\cdot[T_{\{a_i\}}^\ast X^{\an}]
\end{align}
(see Figure \ref{fig:CCirr} below), where the Lagrangian cycle
$\sum_{f\in N_i^{>0}}N_i(f)\cdot[\Lambda_i^f]$ 
over the punctured disk $B(a_i)^\circ$ is defined 
by gluing the ones defined over some sectors 
along the point $a_i$. 
We call $\CCirr(\SM)_i$ the irregular characteristic cycle of
$\SM$ at the point $a_i\in D^{\an}$.
Also for the point $a_\infty=\infty\in\tl{D}$
we can define a Lagrangian cycle
$\CCirr(\SM)_\infty$ in $T^\ast X^{\an}$ as follows. Let
\begin{align}
N_\infty = N(a_\infty)
\colon \SP_{S_{\infty}\var{X}^{\an}}^\prime
\longrightarrow (\ZZ_{\geq 0})_{S_{\infty}\var{X}^{\an}}
\end{align}
be the multiplicity of the analytic meromorphic connection
$(\tl{\SM}^{\an})$ at $a_\infty=\infty\in\var X^{\an}$.
Then by $N_{\infty}^{>0}:=N_{\infty}^{-1}( \ZZ_{>0} )\subset
\SP_{S_{\infty}\var{X}^{\an}}^\prime$ we set
\begin{align}
\CCirr(\SM)_\infty \coloneq
\sum_{f\in N_\infty^{>0}}N_\infty(f)\cdot[\Lambda_\infty^f],
\end{align}
where the complex Lagrangian submanifolds 
$\Lambda_\infty^f\subset
X^{\an}\times\CC\simeq T^\ast  X^{\an}$ are 
defined similarly on a sufficiently small punctured disk 
$B(a_{\infty})^{\circ}$ in 
$\var{X}^{\an}=(\PP^1)^{\an}$ centered at $a_\infty=\infty$.
We thus obtain a Lagrangian cycle
\begin{align}
\CCirr(\SM) \coloneq
\CCirr(\SM)_\infty + \sum_{i=1}^{l}\CCirr(\SM)_i
\end{align}
in $T^\ast X^{\an}$ (see Figure \ref{fig:CCirr}). 
We call it the irregular characteristic cycle of
the meromorphic connection $\SM\in\Modhol(\SD_X)$.
Now let us consider the symplectic transform
$\chi\colon T^\ast X^{\an}\simto T^\ast Y^{\an}$ in 
\cite{DK18} defined by
\begin{align}
T^\ast X^{\an}=X^{\an}\times Y^{\an}\ni(z,w)
\longmapsto (w,-z)\in Y^{\an}\times X^{\an}=T^\ast Y^{\an}.
\end{align}
Let $\Lambda(\SM)\subset T^\ast X^{\an}$ be
the support of the irregular characteristic cycle $\CCirr(\SM)$. 
Set $b_{\infty}:= \infty\in\var{Y}^{\an}$. 
Then for a sufficiently small punctured disk
\begin{align}
B( b_\infty )^\circ \coloneq
\Set*{w\in Y^{\an}=\CC}{\abs{w}>\frac{1}{\varepsilon}}
\quad(0<\varepsilon\ll1)
\end{align}
in $\var{Y}^{\an}=(\PP^1)^{\an}$ centered at
$b_{\infty}= \infty\in\var{Y}^{\an}$ the restriction of the projection
\begin{align}
T^\ast X^{\an}=X^{\an}\times Y^{\an}
\longtwoheadrightarrow Y^{\an}
\end{align}
to $\Lambda(\SM)\cap(X^{\an}\times B( b_\infty )^\circ )$ is
an unramified finite covering over $B( b_\infty )^\circ \subset Y^{\an}$.
\begin{figure}
    \centering
\begin{tikzpicture}[scale=0.7]
    \coordinate(O)at(0,0);
    \coordinate(Ym)at(0.5,-2);
    \coordinate(YM)at(7.7,-2);
    \coordinate(Yan)at(7.9,-2);

    \coordinate(Xm)at(0,-1.5);
    \coordinate(XM)at(0,2.7);
    \coordinate(Xan)at(0,2.9);

    \coordinate(Yinfty)at(7.75,-2);
    \coordinate(Xinfty)at(0,2.75);
    \coordinate(CCirr)at(7.75,0);

    \fill[lightgray] (7.25,3) rectangle (7.75,-2);

    \draw[dashed] (Xinfty) -- (8,2.75);
    \draw[dashed] (Yinfty)++(-0.5,0) -- (7.25,3);
    \draw[dashed] (Yinfty) -- (7.75,3);
    
    
    \coordinate(S)at(7.5,-2);
    \draw (5.5,-1) node (T) {$B(b_\infty)^\circ$};
    \draw (S)--(T);

    \draw[very thick] (7.75,-2) -- (7.25,-2);
    \filldraw[black]  (7.25,-2) circle[radius=0.05];
    \draw  (7.25,-2) circle[radius=0.05];

    \draw[semithick,->,>=stealth](Ym)--(YM) node[at={(Yan)},above,right]{$Y^\an$};
    \draw[semithick,->,>=stealth](Xm)--(XM) node[at={(Xan)},above]{$X^\an$};
    
    \node[rotate=-90] at(7.5,-1.1) {$\twoheadrightarrow$};
    \node[rotate=-90] at(7.5,1.3) {$\twoheadrightarrow$};

    \draw (-0.25,0) -- (8,0);
    \draw (1,1.3-0.2) .. controls (1.2,0.6-0.2) and (2,0.43-0.2)  .. (7.75,0.05) node[pos=0.6,above]{$\CCirr(\SM)_i$};
    \draw (1,-1.3+0.2) .. controls (1.2,-0.6+0.2) and (2,-0.43+0.2) .. (7.75,-0.05);
    \draw (0.5,2.85-1.3) .. controls (1.2,2.85-0.6) and (2,2.85-0.5) .. (7.75,2.75) node[pos=0.7,below right]{$\CCirr(\SM)_\infty$};

    \filldraw[white]  (Xinfty) circle[radius=0.075];
    \draw  (Xinfty) circle[radius=0.075] node[left]{$\infty$};
    \filldraw[white]  (Yinfty) circle[radius=0.075];
    \draw  (Yinfty) circle[radius=0.075] node[below]{$\infty$};

    \filldraw[white]  (CCirr) circle[radius=0.075];
    \draw  (CCirr) circle[radius=0.075];
    \filldraw[white]  (7.75,2.75) circle[radius=0.075];
    \draw  (7.75,2.75) circle[radius=0.075];
    \filldraw[black]  (0,0) circle[radius=0.075];
    \draw (-0.25,0) node[left]{$a_i$};
\end{tikzpicture} 
\caption{The irregular characteristic cycle of $\SM$.}
\label{fig:CCirr}
\end{figure}
The same is true also for the Lagrangian cycle 
$\chi( \CCirr(\SM) )$ in $T^\ast Y^{\an} \simeq T^\ast X^{\an}$.
We denote by $d(\SM)\in\ZZ_{\geq0}$ its covering degree over 
$B( b_\infty )^\circ \subset Y^{\an}$. 
Let $V\subset B( b_\infty )^\circ$ be
a simply connected (open) sector along the point $\infty
\in\var{Y}^{\an}$. Since $\chi( \CCirr(\SM) )$ is a Lagrangian 
cycle, then there exist (mutually distinct) Puiseux germs
\begin{align}
g_1,g_2,\dots,g_n\in\SP^{\prime}_{S_\infty\var{Y}^{\an}}
\end{align}
which are holomorphic on $V\subset B( b_\infty )^\circ$
and positive integers $d_1, d_2,\dots, d_n>0$ such that we have
\begin{align}\label{eq:d_ig_i}
\chi(\CCirr(\SM))\cap(V \times X^{\an} )
= \sum_{i=1}^{n}d_i\cdot[\Lambda^{g_i}],
\end{align}
where we set
\begin{align}
\Lambda^{g_i} \coloneq \Set*{( w, g_i^\prime(w))}{w\in V}
\subset V \times X^{\an} \simeq T^\ast V
\end{align}
(see Lemmas \ref{lem-T7} and \ref{lem-T2-1}). 
Comparing the covering degrees of the both sides 
over the sector $V\subset B( b_\infty )^\circ$, we find
\begin{align}
d(\SM) = \sum_{i=1}^{n}d_i. 
\end{align} 
Note that 
we can explicitly calculate $g_i\in\SP^{\prime}_{S_\infty \var{Y}^{\an}}$
and $d_i>0$ $(1\leq i\leq n)$ (see 
(the proof of) Lemma \ref{lem-T2-1}). 
Then our first main result in this paper is the following theorem.

\begin{theorem}\label{Thm-T3}
For any simply connected open sector $V\subset B( b_\infty )^\circ$
along the point $b_{\infty}= \infty\in\var{Y}^{\an}$, 
there exists an isomorphism
\begin{align}
\pi^{-1}\CC_{V}\otimes Sol_{\var{Y}}^{\rmE}(\tl{\SM^\wedge})
\simeq \bigoplus_{i=1}^n
\(\EE_{V\vbar\var{Y}^{\an}}^{\Re g_i}\)^{\oplus d_i}.
\end{align}
In particular, the generic rank ${\rm rk} \SM^\wedge$ of
the Fourier transform $\SM^\wedge$ of $\SM$ is
equal to the covering degree $d(\SM)$ of
the irregular characteristic cycle $\CCirr(\SM)$ over
the punctured disk $B( b_\infty )^\circ \subset Y^{\an}$
centered at $b_{\infty}= \infty\in\var{Y}^{\an}$.
\end{theorem}

Although as we see in \cite[Lemma 4.4.1]{DK18} D'Agnolo and 
Kashiwara avoided some sectors in the proof 
of their main theorem in \cite{DK18}, one may 
deduce also a result similar to Theorem \ref{Thm-T3} from 
the results in \cite{DK18} and \cite{DK23} 
by relying on the general properties of the 
enhanced solution complexes to holonomic 
$\SD$-modules. The point is that 
we do not use such an indirect argument in 
our proof. We have also another expression of 
the generic rank ${\rm rk} \SM^\wedge$ 
(see Corollary \ref{Cor-T8}). 
For a generalization of Theorem \ref{Thm-T3} to arbitrary 
holonomic $\SD$-modules on $X= \CC$ see Section \ref{sec:holD}. 
If an exponential factor $f \in N_i^{>0}$ at a point $a_i \in 
\tl{D}$ is holomorphic on a sector $S$ along it, 
by abuse of notations we write $f \in N_i^{>0}(S)$. 
For such $f\in N_i^{>0}(S)$ and $w \in Y^{\an}= \CC$ 
we define a holomorphic function $f^w$ on $S$ by
\begin{align}\label{funcfw} 
f^w(z) \coloneq zw-f(z) \quad \(z\in S \).
\end{align}
Then to prove Theorem \ref{Thm-T3}, for $w \in V$ 
we consider the real-valued functions
$\Re f^w\colon S \longrightarrow\RR\
 (f\in N_i^{>0}(S))$ as Morse functions
and apply a Morse theory associated to them.
In contrast to the usual Morse theory,
that we use here relies on
the sublevel sets of ``several" Morse functions.
See Section \ref{sec:T4} for the details. 

We obtain similar results also for the other singular 
points of $\SM^{\wedge}$ as follows. 
For a point $b\in Y^{\an}=\CC$ we set 
\begin{equation}
N_{\infty,b}^{>0}\coloneq\{f\in N_\infty^{>0}\mid f(z)=
bz+(\textit{lower order terms})\} \quad \subset N_\infty^{>0}.
\end{equation}
Then for the sufficiently small punctured disk $B(a_\infty)^\circ\subset 
X^{\an}$ centered at the point $a_\infty=\infty\in\var{X}^{\an}$, 
we define a Lagrangian cycle $\CCirr(\SM)^b$ in $T^\ast 
B(a_\infty)^\circ\simeq B(a_\infty)^\circ\times Y^{\an}\subset T^\ast X^{\an}$ by
\begin{equation}
\CCirr(\SM)^b\coloneq\sum_{f\in N_{\infty,b}^{>0}}N_\infty(f)\cdot[\Lambda_\infty^f].
\end{equation}
Note that $\CCirr(\SM)^b$ is a part of $\CCirr(\SM)_{\infty}$ 
which does not affect the covering degree $d(\SM)$ in 
Theorem \ref{Thm-T3}. 
Let $\Lambda(\SM)^b\subset T^\ast X^{\an}$ be 
the support of $\CCirr(\SM)^b$.
Then for a sufficiently small punctured disk $B(b)^\circ\in Y^{\an}$ 
centered at the point $b\in Y^{\an}$ the restriction of the projection
\begin{equation}
T^\ast X^{\an}=X^{\an}\times Y^{\an}\longtwoheadrightarrow Y^{\an}
\end{equation} 
to $\Lambda(\SM)^b\cap(X^{\an} \times B(b)^\circ)$ is an unramified finite covering 
over $B(b)^\circ\subset Y^{\an}$. The same is true also 
for the Lagrangian cycle 
$\chi( \CCirr(\SM)^b )$ in $T^\ast Y^{\an} \simeq T^\ast X^{\an}$.
We denote by $d(\SM)^b\in\ZZ_{\geq0}$ its covering degree over 
$B(b)^\circ\subset Y^{\an}$. 
Let $W\subset B(b)^\circ$ be a simply connected (open) sector 
along the point $b\in Y^{\an}$. 
Since $\chi( \CCirr(\SM)^b )$ is a Lagrangian cycle in $T^\ast Y^{\an}$, then there 
exist (mutually distinct) Puiseux germs 
\begin{equation}
h_1,h_2,\dots,h_m\in\SP^{\prime}_{S_bY^{\an}}
\end{equation} 
which are holomorphic on $W\subset B(b)^\circ$ and 
positive integers $e_1,e_2,\dots,e_m >0$ such that 
\begin{equation}\label{eq:e_ih_i}
\chi(\CCirr(\SM)^b)\cap(W \times X^{\an} )=\sum_{i=1}^{m}e_i\cdot[\Lambda^{h_i}]
\end{equation}
(see Lemmas \ref{lem-T7} and \ref{lem-T2-1}). 
Note that we have $d(\SM)^b=\sum_{i=1}^{m}e_i$.
Then, as in the proof of Theorem \ref{Thm-T3}, we obtain the following result.

\begin{theorem}\label{thm-on-bnd}
For any simply connected  
open sector $W\subset B(b)^\circ$ along the point $b\in Y^{\an}$, 
there exists an isomorphism
\begin{equation}
\pi^{-1}\CC_W\otimes Sol_{\var{Y}}^{\rmE}(\tl{\SM^\wedge}) 
\quad \simeq \quad 
\Bigl\{\bigoplus_{i=1}^m\bigl(
\EE_{W\vbar\var{Y}^{\an}}^{\Re h_i}\bigr)^{\oplus e_i}\Bigr\}
\oplus\bigl(\EE_{W\vbar\var{Y}^{\an}}^0\bigr)^{d(\SM)-d(\SM)^b}.
\end{equation}
If we have moreover that $b\neq0$ and $N_{\infty,b}^{>0}=\emptyset$, then 
for the open disk $B(b):= B(b)^{\circ} \sqcup \{ b \}$ 
centered at 
$b\in Y^{\an}$ there exists an isomorphism 
\begin{equation}
\pi^{-1}\CC_{B(b)}\otimes Sol_{\var{Y}}^{\rmE}(\tl{\SM^\wedge})
\simeq\bigl(\EE_{B(b)\vbar\var{Y}^{\an}}^0\bigr)^{d(\SM)},
\end{equation}
which implies that $\SM^\wedge$ is an integrable 
connection on a neighborhood of the point $b\in Y^{\an}$.
\end{theorem}

For a generalization of Theorem \ref{thm-on-bnd} to arbitrary 
holonomic $\SD$-modules on $X= \CC$ see Section \ref{sec:holD}. 
In the situation of Theorem \ref{thm-on-bnd}, 
note that for the irregularity 
$\irr_b(\Gamma_{Y\setminus\{b\}}(\SM^\wedge)) \in \ZZ_{\geq 0}$ of 
the meromorphic connection 
$\Gamma_{Y\setminus\{b\}}(\SM^\wedge)$ 
at the point $b\in Y$ we have 
\begin{equation}
 \irr_b(\Gamma_{Y\setminus\{b\}}(\SM^\wedge))
=\sum_{i=1}^m e_i \cdot \ord_b(h_i).
\end{equation}
Finally, we prove the following formula from 
which we obtain the 
(classical) characteristic cycle ${\rm CC} ( \SM^\wedge )$ of the 
Fourier transform $\SM^\wedge$. 
For the point $b \in Y$ we denote the 
multiplicity of $\SM^\wedge$ along $T_b^\ast Y\subset T^\ast Y$ 
by $\mult_{T_b^\ast Y}(\SM^\wedge) \in \ZZ_{\geq 0}$. 

\begin{theorem}\label{thm-CharCyc}
In the situation of Theorem \ref{thm-on-bnd}, we have 
\begin{align}
\mult_{T_b^\ast Y}(\SM^\wedge) 
&= d(\SM)^b+\irr_b(\Gamma_{Y\setminus\{b\}}(\SM^\wedge)) 
+ N_{\infty}(bz)
\notag \\
 &= d(\SM)^b + \sum_{i=1}^m e_i\cdot \ord_b(h_i) 
+ N_{\infty}(bz) 
\notag \\
 &=  \sum_{i=1}^m e_i \cdot \Bigl\{ 1+ \ord_b(h_i) \Bigr\} 
 + N_{\infty}(bz).  
\end{align}
Moreover if $N_\infty^{>0}$ does not contain the linear factor $bz$, 
 then there exist isomorphisms
\begin{equation}
\Gamma_{\{b\}}(\SM^\wedge)\simeq0, \quad 
H^1_{\{b\}}(\SM^\wedge)
\simeq\{H^1_{\{b\}}(\SO_Y)\}^{d(\SM)-d(\SM)^b}.
\end{equation} 
\end{theorem}
Since the Fourier transform is an exact functor, 
Theorem \ref{thm-CharCyc} holds true also for 
any holonomic $\SD$-module on $X= \CC_z$ (see 
the discussions at the end of Section \ref{sec:T3}).

In Section \ref{uni-sec:T5}, we will 
apply a similar Morse theoretical method also to
obtain a natural basis of the stalk $Sol_Y( \SM^\wedge )_b$ of
the solution complex $Sol_Y( \SM^\wedge )$
to $\SM^\wedge$ at a generic point $b \in Y^{\an}$
of $Y^{\an}$. Specifically, we identify
$Sol_Y( \SM^\wedge )_b$ with some rapid decay homology
group in the sense of Bloch-Esnault \cite{BE04a}
and Hien \cite{Hi07}, \cite{Hi09} and construct a 
basis of the latter. This construction is
applicable to any algebraic meromorphic connection
$\SM$ on $X= \CC$. Then for such $\SM$ it would be
possible to calculate the monodromies and the Stokes
matrices of the holomorphic solutions to
$\SM^\wedge$ explicitly by looking at how
the rapid decay $1$-cycles in the basis deform as the
point $b \in Y^{\an}$ moves.

\section{Preliminaries}\label{uni-sec:K1}

In this section, we recall some basic notions 
and results which will be used in this paper.
We assume here that the reader is familiar with 
the theory of sheaves and functors in the 
framework of derived categories.
For them we follow the terminologies in \cite{KS90} etc.
For a topological space $X$ denote by $\BDC(\CC_X)$ 
the derived category consisting of bounded 
complexes of sheaves of $\CC$-vector spaces on it.

\subsection{Enhanced sheaves}\label{sec:K2}
We refer to \cite{Tam18}, \cite{KS16b}, 
and \cite{DK16} for the details of this subsection.
Let $X$ be a complex manifold and 
we consider the maps 
\begin{align}
X\times\RR^2 \xrightarrow{p_1,p_2,\mu} 
X\times\RR \overset{\pi}{\longrightarrow} X,
\end{align}
where $p_1,p_2,\pi$ are the projections and we set 
$\mu(x,t_1,t_2)\coloneq(x,t_1+t_2)$.
Then we define the bounded derived 
category of enhanced sheaves $\BEC(\CC_X)$ on $X$ by
\begin{align}
\BEC(\CC_X) \coloneq \BDC(\CC_{X\times\RR})/\pi^{-1}\BDC(\CC_X).
\end{align}
The convolution functors $\Potimes$ and $\Prhom$ in 
$\BDC(\CC_{X\times\RR})$ are defined by 
\begin{align}
F\Potimes G &\coloneq \rmR\mu_!(p_1^{-1}F\otimes p_2^{-1}G), \\
\Prhom(F,G) &\coloneq \rmR p_{1\ast}\rhom(p_2^{-1}F,\mu^!G),
\end{align} 
and they induce convolution functors in $\BEC(\CC_X)$, 
which we denote also by $\Potimes$ and $\Prhom$, respectively.
The quotient functor 
\begin{align}
\bfQ \colon \BDC(\CC_{X\times\RR})\longrightarrow\BEC(\CC_X)
\end{align}
has fully faithful left and right adjoints 
$\bfL^\rmE,\bfR^\rmE$ defined by 
\begin{align}
\bfL^\rmE(\bfQ F) &\coloneq 
(\CC_{\{t\geq0\}}\oplus\CC_{\{t\leq0\}})\Potimes F, \\
\bfR^\rmE(\bfQ F) &\coloneq 
\Prhom(\CC_{\{t\geq0\}}\oplus\CC_{\{t\leq0\}},F),  
\end{align}
where $\{t\geq0\}$ stands for 
$\Set*{\(x,t\)\in X\times\RR}{t\geq0}$ 
and $\{t\leq0\}$ is defined similarly.
In this paper, we sometimes identify a sheaf $F$ on $X\times\RR$ with the enhanced sheaf $\bfQ F$ on $X$ associated to it.
The functor 
$\bfL^\rmE\colon\BEC(\CC_X)\longrightarrow
\BDC(\CC_{X\times\RR})$ induces 
a $t$-structure of $\BEC(\CC_X)$ by the standard 
$t$-structure of $\BDC(\CC_{X\times\RR})$.
We denote by $\bfE^0(\CC_X)$ its heart.
For a morphism of complex manifolds $f\colon X\rightarrow Y$,
we define the direct image functor $\bfE f_\ast\colon 
\BEC(\CC_X)\rightarrow\BEC(\CC_Y)$
so that the following diagram commutes:
\begin{equation}
\vcenter{
\xymatrix@M=7pt@C=36pt@R=24pt{
\BDC(\CC_{X\times\RR}) \ar[r]^-{\rmR(f\times\id_{\RR})_\ast} 
\ar[d]^-{\bfQ} & \BDC(\CC_{Y\times\RR}) \ar[d]^-{\bfQ} & \\
\BEC(\CC_X) \ar[r]^-{\bfE f_\ast} & \BEC(\CC_Y). &
}}
\end{equation} 
The proper direct image and (proper) inverse image functors
\begin{align}
\bfE f_! &\colon 
\BEC(\CC_X)\longrightarrow\BEC(\CC_Y), \\
\bfE f^{-1},\bfE f^! &\colon 
\BEC(\CC_Y)\longrightarrow\BEC(\CC_X)
\end{align}
are defined similarly.
We have a natural embedding 
$\epsilon\colon\BDC(\CC_X)\rightarrow\BEC(\CC_X)$ defined by
\begin{align}
\epsilon(L) \coloneq \bfQ(\CC_{\{t\geq0\}}\otimes\pi^{-1}L),
\end{align}
and a bifunctor $\pi^{-1}(\cdot)\otimes(\cdot)\colon
\BDC(\CC_X)\times\BEC(\CC_X)\rightarrow\BEC(\CC_X)$ defined by
\begin{align}
\pi^{-1}L\otimes F \coloneq 
\bfQ(\pi^{-1}L\otimes\bfL^\rmE(F)).
\end{align}

\subsection{Enhanced ind-sheaves}\label{sec:K3}

We briefly recall some notions and results on 
enhanced ind-sheaves without giving detailed definitions.
We refer to \cite{KS01} and \cite{KS06} for ind-sheaves, 
to \cite{DK16} for ind-sheaves on bordered spaces, and 
to \cite{DK16} and \cite{KS16b} for enhanced ind-sheaves.
Let $X$ be a complex manifold and $\RR_\infty$ 
the bordered space $(\RR, \var{\RR}\coloneq\RR\sqcup\{\pm\infty\})$.
Denote by $\BDC(\rmI\CC_X)$ and $\BDC(\rmI\CC_{X\times\RR_\infty})$ 
the bounded derived categories of ind-sheaves on $X$ and 
ind-sheaves on the bordered space $X\times\RR_\infty$, respectively.
We define the the bounded derived 
category of enhanced ind-sheaves $\BEC(\rmI\CC_X)$ on $X$ by 
\begin{equation}
\BEC(\rmI\CC_X) \coloneq 
\BDC(\rmI\CC_{X\times\RR_\infty})/\pi^{-1}\BDC(\rmI\CC_X),
\end{equation}
where $\pi\colon X\times\RR_\infty\rightarrow X$ 
is the projection of bordered spaces.
The quotient functor 
\begin{align}
\bfQ \colon \BDC(\rmI\CC_{X\times\RR_\infty})
\longrightarrow\BEC(\rmI\CC_X)
\end{align}
has fully faithful left and right adjoints
\begin{align}
\bfL^\rmE,\bfR^\rmE \colon \BEC(\rmI\CC_X)\longrightarrow
\BDC(\rmI\CC_{X\times\RR_\infty}).
\end{align}
A $t$-structure of $\BEC(\rmI\CC_X)$ is induced by 
the standard $t$-structure of 
$\BDC(\rmI\CC_{X\times\RR_\infty})$ and the functor 
$\bfL^\rmE$ as in $\BEC(\CC_X)$. 
Denote by $\bfE^0(\rmI\CC_X)$ its heart.
Furthermore, as in the case of enhanced sheaves 
we can define the convolution functors
\begin{align}
\Potimes &\colon \BEC(\rmI\CC_X)\times\BEC(\rmI\CC_X)
\longrightarrow\BEC(\rmI\CC_X), \\
\Prihom &\colon \BEC(\rmI\CC_X)^{\op}\times\BEC(\rmI\CC_X)
\longrightarrow\BEC(\rmI\CC_X),
\end{align}
and the operations of (proper) direct and inverse images
\begin{align}
\bfE f_\ast,\bfE f_{!!} &\colon \BEC(\rmI\CC_X)
\longrightarrow\BEC(\rmI\CC_Y), \\
\bfE f^{-1},\bfE f^! &\colon \BEC(\rmI\CC_Y)
\longrightarrow\BEC(\rmI\CC_X),    
\end{align} 
for a morphism of complex manifolds $f\colon X\rightarrow Y$.
We have a natural embedding $\varepsilon\colon\BDC(\rmI\CC_X)
\rightarrow\BEC(\rmI\CC_X)$ and a bifunctor 
$\pi^{-1}(\cdot)\otimes(\cdot)
\colon\BDC(\rmI\CC_X)\times\BEC(\rmI\CC_X)
\rightarrow\BEC(\rmI\CC_X)$ defined by
\begin{align}
\varepsilon(\SL) \coloneq 
\bfQ(\CC_{\{t\geq0\}}\otimes\pi^{-1}\SL), \\
\pi^{-1}\SL\otimes\SF \coloneq 
\bfQ(\pi^{-1}\SL\otimes\bfL^\rmE(\SF)).
\end{align}
We set $\CC_X^\rmE\coloneq\bfQ\Bigl(\underset{\alpha \to+\infty}
{\inj}\CC_{\{t\geq \alpha \}}\Bigr)\in\BEC(\rmI\CC_X)$, where 
the symbol $\inj$ stands for the inductive limit  
in $\rmI\CC_{X\times\var{\RR}}$.

\subsection{Exponential enhanced (ind-)sheaves}\label{sec:K4}

Let $X$ be a complex manifold.
Denote by $\rcEC^{\rmb}(\CC_X)$ and $\rcEC^{\rmb}(\rmI\CC_X)$ 
the triangulated categories of $\RR$-constructible 
enhanced (ind-)sheaves on $X$
(see \cite{DK16}).
Let $U\subset X$ be an open subset and $\phi,\phi^+,\phi^-\colon U\to\RR$ 
continuous functions such that 
$\phi^+\geq\phi^-$ on it.
For a locally closed subset $Z\subset U$, we define the exponential 
enhanced (ind-)sheaves $\sfE_{Z\vbar X}^\phi,
\sfE_{Z\vbar X}^{\phi^+\vartriangleright\phi^-}\in\bfE^0(\CC_X)$ 
and $\EE_{Z\vbar X}^\phi,\EE_{Z\vbar X}
^{\phi^+\vartriangleright\phi^-}\in\bfE^0(\rmI\CC_X)$ by 
\begin{align}
\sfE_{Z\vbar X}^\phi &\coloneq \bfQ(\CC_{\{t+\phi\geq0\}}), \\
\sfE_{Z\vbar X}^{\phi^+\vartriangleright\phi^-} 
&\coloneq \bfQ(\CC_{\{-\phi^+\leq t<-\phi^-\}}), \\
\EE_{Z\vbar X}^\phi &\coloneq \CC_X^\rmE\Potimes\sfE_{Z\vbar X}^\phi, \\
\EE_{Z\vbar X}^{\phi^+\vartriangleright\phi^-} 
&\coloneq \CC_X^\rmE\Potimes\sfE_{Z\vbar X}
^{\phi^+\vartriangleright\phi^-},
\end{align}
where $\{t+\phi\geq0\}$ and $\{-\phi^+\leq t<-\phi^-\}$ stand for 
$\Set*{(x,t)\in X\times\RR}{x \in Z, t+\phi(x)\geq0}$ and 
$\Set*{(x,t)\in X\times\RR}{x \in Z, -\phi^+(x)\leq t<-\phi^-(x)}$, respectively.
Note that we have exact sequences 
\begin{align}
0 \longrightarrow \sfE_{Z\vbar X}^{\phi^+\vartriangleright\phi^-} 
\longrightarrow
\sfE_{Z\vbar X}^{\phi^+} \longrightarrow \sfE_{Z\vbar X}^{\phi^-} 
\longrightarrow 0, \\
0 \longrightarrow \EE_{Z\vbar X}^{\phi^+\vartriangleright\phi^-} 
\longrightarrow
\EE_{Z\vbar X}^{\phi^+} \longrightarrow \EE_{Z\vbar X}^{\phi^-} 
\longrightarrow 0
\end{align}
in $\bfE^0(\CC_X)$ and $\bfE^0(\rmI\CC_X)$,
respectively.
Note also that if $Z$ is subanalytic and 
$f\colon U\to\CC$ is holomorphic 
then we have $\sfE_{Z\vbar X}^{\Re f}\in\rcEC^0(\CC_X)$ and 
$\EE_{Z\vbar X}^{\Re f}\in\rcEC^0(\rmI\CC_X)$.
For these exponential enhanced ind-sheaves we have 
the following basic property (see \cite[Corollary 3.2.3]{DK18}).

\begin{lemma}\label{lem-K1}
Let $\phi^+,\phi^-\colon U\to\RR$ be 
continuous functions such that $\phi^+\geq\phi^-$ on $U$.
If $\phi^+-\phi^-$ is bounded on $U$, 
then there exists an isomorphism
\begin{align}
\EE_{U\vbar X}^{\phi^+\vartriangleright\phi^-} \simeq0.
\end{align}
\end{lemma}

\subsection{$\SD$-modules}\label{sec:K5}

Let us recall some notions and results on 
$\SD$-modules on a complex manifold $X$
(we refer to \cite{Kas03} and \cite{HTT08} etc.).
Denote by $\SO_X,\Omega_X$ and $\SD_X$ the sheaves 
of holomorphic functions, holomorphic differential forms 
of top degree and holomorphic differential operators on $X$, respectively.
Let $\Mod(\SD_X)$ be the abelian category of left $\SD_X$-modules.
Then we can define $\Modcoh(\SD_X)$ (resp. $\Modhol(\SD_X)$) 
to be the subcategory of $\Mod(\SD_X)$ consisting of 
coherent (resp. holonomic) $\SD_X$-modules.
We write $\BDC(\SD_X)$ for the bounded derived category of 
left $\SD_X$-modules and denote by $\BDCcoh(\SD_X)$ and 
$\BDChol(\SD_X)$ its full triangulated subcategories of 
objects which have coherent and 
holonomic cohomologies, respectively.
The symbols $\Dotimes,\bfD f_\ast,\bfD f^\ast$ stand for 
the standard operations for $\SD$-modules associated to 
a morphism of complex manifolds $f\colon X\to Y$.
The (classical) solution functor is defined by
\begin{align}
Sol_X \colon \BDCcoh(\SD_X)^{\op}\longrightarrow
\BDC(\CC_X), \quad \SM\longmapsto\rhom_{\SD_X}(\SM,\SO_X).
\end{align}
Let $D\subset X$ be a closed hypersurface and 
denote by $\SO(\ast D)$ the sheaf of 
meromorphic functions on $X$ with poles in $D$.
Then for $\SM\in\BDC(\SD_X)$, we set 
\begin{align}
\SM(\ast D)\coloneq\SM\Dotimes\SO_X(\ast D)
\end{align}
and for $f\in\SO_X(\ast D)$ and $U\coloneq X\bs D$, set
\begin{align}
\SD_X e^f &\coloneq \SD_X / 
\Set{P\in\SD_X}{P e^f\vbar_U=0}, \\
\SE_{U\vbar X}^f &\coloneq \SD_X e^f(\ast D).
\end{align}
Note that $\SE_{U\vbar X}^f$ is a holonomic $\SD_X$-module.
In \cite{DK16}, the authors constructed 
the enhanced solution functor on a complex manifold $X$
\begin{align}
Sol_X^\rmE\colon\BDChol(\SD_X)^{\op}
\longrightarrow\rcEC^{\rmb}(\rmI\CC_X)
\end{align}
and proved that it is fully faithful.
Instead of giving its definition, we recall some 
of its properties.

\begin{proposition}\label{prop-K2}
Let $D\subset X$ be a closed hypersurface in 
$X$ and set $U\coloneq X\bs D$.
\begin{enumerate}
\item[{\rm(i)}] If $\SM\in\BDChol(\SD_X)$, 
then there exists an isomorphism in $\BEC(\rmI\CC_X)$
\begin{align}Sol_X^{\rmE}(\SM(\ast D)) \simeq 
\pi^{-1}\CC_U\otimes Sol_X^{\rmE}(\SM).
\end{align}
\item[{\rm(ii)}] Let $f\colon X\to Y$ 
be a morphism of complex manifolds.
If $\SN\in\BDChol(\SD_Y)$, then there 
exists an isomorphism in $\BEC(\rmI\CC_X)$
\begin{equation}
Sol_X^{\rmE}(\bfD f^\ast\SN)\simeq\bfE f^{-1}Sol_Y^{\rmE}(\SN).
\end{equation}
\item[{\rm(iii)}] If $f\in\SO_X(\ast D)$, 
then there exists an isomorphism in $\BEC(\rmI\CC_X)$
\begin{align}
Sol_X^{\rmE}(\SE_{U\vbar X}^f) 
\simeq \EE_{U\vbar X}^{\Re f}.
\end{align}
\item[{\rm(iv)}] Let $\SM$ be 
a regular holonomic $\SD_X$-module and 
set $L\coloneq Sol_X(\SM)$.
Then we have an isomorphism in $\BEC(\rmI\CC_X)$
\begin{align}
Sol_X^{\rmE}(\SM)\simeq\CC_X^{\rmE}\Potimes\epsilon(L).
\end{align}
\end{enumerate}
\end{proposition}

\subsection{Puiseux germs and normal forms of 
enhanced (ind-)sheaves}\label{sec:K6}

In this subsection, we recall some definitions on 
Puiseux germs and 
normal forms of enhanced (ind-)sheaves 
in \cite[Section 5]{DK18} to describe 
the Hukuhara-Levelt-Turrittin theorem 
in terms of enhanced ind-sheaves.
Let $X$ be a complex manifold of dimension one.
For $a\in X$, denote by $\varpi_a\colon \tl{X_a}\to X$ 
the real oriented blow-up of $X$ along $a$ and consider 
the commutative diagram
\begin{equation}
\vcenter{
\xymatrix@M=5pt{
S_aX \ar[r]^-{\tl{\imath}_a} & \tl{X_a} \ar[d]^-{\varpi_a}& \\
X\bs\{a\} \ar[r]^-{j_a} \ar[ur]^-{\tl{\jmath}_a} & X, 
}}\end{equation}
where we set $S_aX\coloneq\varpi_a^{-1}(a)\simeq 
S^1$ and $\tl{\imath}_a,\tl{\jmath}_a,j_a$ 
are the natural embeddings.
We denote by $\SP_{\tl{X_a}}$ the subsheaf of 
$\tl{\jmath}_{a\ast} j_a^{-1}\SO_X$ whose sections are defined by 
\begin{multline*}
\Gamma(\Omega;\SP_{\tl{X_a}}) \coloneq 
\{f\in\Gamma(\Omega;\tl{\jmath}_{a\ast} 
j_a^{-1}\SO_X)\mid \textit{For any $\theta\in\Omega\cap S_aX$,} \\ 
\textit{$f$ admits a Puiseux expansion at $\theta$.}\}
\end{multline*}
for open subsets $\Omega\subset\tl{X_a}$.
Then we define the sheaf of 
Puiseux germs $\SP_{S_aX}$ on $S_aX$ to be  
\begin{align}
\SP_{S_aX}\coloneq\tl{\imath}_a^{\,-1}\SP_{\tl{X_a}}.
\end{align}
For a rational number $\lambda\in\QQ$, 
denote by 
$\SP_{S_aX}^{\lambda},
\SP_{S_aX}^{\leq\lambda}\subset\SP_{S_aX}$ 
the subsheaf of $\SP_{S_aX}$ consisting of sections 
$f$ whose pole order $\ord_a(f)$ at $a$ is 
$\lambda$ and $\leq\lambda$, respectively. 
For an interval $I\subset\RR$, we set 
$\SP_{S_aX}^I\coloneq
\bigsqcup_{\lambda\in I\cap\QQ}\SP_{S_aX}^\lambda$.
Let $z_a$ be a local coordinate centered at $a$ and  
denote by $\SP_{S_aX}^\prime$ the subsheaf of 
$\SP_{S_aX}$ consisting of sections locally contained in
\begin{align}
\bigcup_{p\in\ZZ_{\geq1}}z_a^
{-\frac{1}{p}}\CC[z_a^{-\frac{1}{p}}]
\end{align}
for some (hence, any) branch of $z_a^{1/p}$.
Note that there exists a canonical isomorphism 
$\SP_{S_aX}^\prime\simto\SP_{S_aX}/\SP_{S_aX}^{\leq0}$.

Let $\SM$ be a holonomic $\SD_X$-module. For $a\in X$
the enhanced solution complex $Sol_X^{\rmE}(\SM)$ has 
the following decomposition by some 
exponential enhanced ind-sheaves on 
sufficiently small open sectors along $a$.

\begin{lemma}[see e.g. {\cite[Corollary 3.7]{IT20a}}]\label{lem-ITa}
Let $\SM$ be a holonomic $\SD_X$-module and let $a\in X$.
Then for any $\theta\in S_aX$, 
there exist its sectorial neighborhood 
$V_\theta \subset X \setminus \{ a \} \subset \tl{X_a}$ 
and holomorphic functions 
$f_1,\dots,f_m\in\Gamma(V_\theta;\SP_{\tl{X_a}})$ such that
\begin{align}
\pi^{-1}\CC_{V_\theta}\otimes Sol_X^{\rmE}(\SM) \simeq 
\bigoplus_{i=1}^m\EE_{V_\theta\vbar X}^{\Re f_i}.
\end{align}
\end{lemma}

In \cite{DK18}, D'Agnolo and Kashiwara introduced some 
notions on normal forms of enhanced sheaves to refine the above decomposition.
We recall their definitions in a slightly modified form. 
In particular, as we see in Definition \ref{def-mul} below, 
our multiplicities are defined on $\SP_{S_aX}^\prime$ 
and not on $\SP_{S_aX}$ as in \cite[Section 5.3 and 5.4]{DK18}. 
In this way, we can eliminate the ambiguities of 
exponential factors. 

\begin{definition}\label{def-mul} 
Let $a\in X$ and let 
$N\colon\SP_{S_aX}^\prime\to(\ZZ_{\geq0})_{S_aX}$ be 
a morphism of sheaves of sets on $S_aX \simeq S^1$.
Then the morphism $N$ is said to be a multiplicity at $a$ if 
$N_\theta^{>0}\coloneq N_\theta^{-1}(\ZZ_{>0})
\subset\SP_{S_aX,\theta}^\prime$ is a finite set  
for any $\theta\in S_aX$.
\end{definition}

\begin{definition}[{\cite[Definition 5.3.1]{DK18}}]\label{nf-enhsh} 
Let $F\in\rcEC^{\rmb}(\CC_X)$ be 
an $\RR$-constructible enhanced sheaf. 
Then we say that $F$ has a normal form at $a\in X$ 
if there exists a multiplicity 
$N\colon\SP_{S_aX}^\prime\to(\ZZ_{\geq0})_{S_aX}$ at it 
and any $\theta\in S_aX$ has 
a sectorial open neighborhood $V_\theta \subset X \setminus \{ a \} \subset \tl{X_a}$ 
for which we have an isomorphism
\begin{align}
\pi^{-1}\CC_{V_\theta}\otimes F\simeq 
\bigoplus_{f\in N_\theta^{>0}}
(\sfE_{V_\theta\vbar X}^{\Re f})^{N_\theta(f)}.
\end{align}
\end{definition}

\begin{definition}[{\cite[Definition 5.4.2]{DK18}}]\label{nf-ind} 
Let $\SF\in\rcEC^{\rmb}(\rmI\CC_X)$ be an $\RR$-constructible 
enhanced ind-sheaf. Then we say that $\SF$ has 
a normal form at $a\in X$ if there exists a multiplicity 
$N\colon\SP_{S_aX}^\prime\to(\ZZ_{\geq0})_{S_aX}$ at it 
and any $\theta\in S_aX$ has 
a sectorial open neighborhood $V_\theta \subset 
X \setminus \{ a \} \subset \tl{X_a}$ for which we have 
an isomorphism
\begin{align}
\pi^{-1}\CC_{V_\theta}\otimes \SF\simeq 
\bigoplus_{f\in N_\theta^{>0}}
(\EE_{V_\theta\vbar X}^{\Re f})^{N_\theta(f)}.
\end{align}
\end{definition}

Note that if $F\in\rcEC^{\rmb}(\CC_X)$ 
(resp. $\SF\in\rcEC^{\rmb}(\rmI\CC_X)$) has 
a normal form at $a\in X$, then the multiplicity $N$
for which $F$ (resp. $\SF$) has a normal form
at $a$ is uniquely determined by the following lemma.

\begin{lemma}[{\cite[Lemma 7.15]{Mochi22}}, 
{\cite[Corollary 5.2.3]{DK18}} and {\cite[Proposition 3.16]{IT20a}}]\label{lem-K3}
Let $a\in X$ and 
let $I\subset S_aX$ be an open subset of $S_aX$, 
$f_1,\dots,f_n,g_1,\dots,g_m\in\SP_{S_aX}^\prime(I)$ 
Puiseux germs on $I$ and 
$V \subset X \setminus \{ a \} \subset \tl{X_a}$ 
a sectorial open neighborhood of $I$ 
on which the holomorphic functions 
$f_1,\dots,f_n,g_1,\dots,g_m$ are defined.
Assume that there exists an isomorphism
\begin{align}
\bigoplus_{i=1}^n\EE_{V\vbar X}^{\Re f_i}\simeq
\bigoplus_{i=1}^m\EE_{V\vbar X}^{\Re g_i}.
\end{align}
Then $n=m$ and there exists a bijection
$\sigma\colon\{1,\dots,n\}\simto\{1,\dots,n\}$
such that 
\begin{align}
f_i(z) = g_{\sigma(i)}(z) \quad (z\in V)
\end{align}
for any $1\leq i\leq n$.
\end{lemma}

By Lemma \ref{lem-ITa} and  Lemma \ref{lem-K3}, 
we obtain the following proposition.

\begin{proposition}[{\cite[Lemma 5.4.4]{DK18}}]\label{prop-K4}
Let $\SM$ be a holonomic $\SD_X$-module.
Then for any point $a\in X$ the enhanced solution complex 
$Sol_X^{\rmE}(\SM)\in\rcEC^{\rmb}(\rmI\CC_X)$ of $\SM$ 
has a normal form at $a$.
\end{proposition}

We shall explain some classical notions on meromorphic 
connections in terms of enhanced solution complexes.
Let $\SM$ be a holonomic $\SD_X$-module and $N$ the 
multiplicity for which $Sol_X^{\rmE}(\SM)$ has a normal form at $a\in X$.
We call the sections of 
$N^{>0}:=N^{-1}( \ZZ_{>0} )_{S_aX} \subset\SP_{S_aX}^\prime$ 
the exponential factors of $\SM$ at $a$. 
Then obviously the 
analytic continuation of an exponential factor  
of $\SM$ along the loop $S_aX \simeq S^1$ is 
again an exponential factor. 
The regular rank $r_a\in\ZZ_{\geq0}$ of 
$\SM$ at $a\in X$ is defined by
\begin{equation}
r_a\coloneq N(0) \ \in\ZZ_{\geq0}.
\end{equation}
For two exponential factors 
$f_1,f_2\in N^{>0}$ such that $f_1 \not= f_2$ 
which are holomorphic on 
a sector $S$ along the point $a \in X$, the Stokes curves of 
the pair $(f_1,f_2)$ (over $S$) 
are the irreducible components of 
the real analytic set 
\begin{equation}
\Set*{z\in S }{\Re f_1(z)-\Re f_2(z)=0} \quad \subset S
\end{equation}
and a ray tangent to a Stokes curve at $a$ is 
called a Stokes line of $(f_1,f_2)$.
Denote by $\mathrm{St}_a(f_1,f_2)$ the set of 
the Stokes lines of the pair $(f_1,f_2)$ (over $S$).
Then by taking a union over all the sectors $S$ along 
$a \in X$ and the exponential factors $f_1 \not= f_2$ 
on them, we set
\begin{equation}
\mathrm{St}_a(\SM)\coloneq\bigcup_{f_1,f_2\in N^{>0}, \ f_1 \not= f_2} 
\mathrm{St}_a(f_1,f_2).
\end{equation}
We call the elements of $\mathrm{St}_a(\SM)$ the Stokes lines of $\SM$ at $a\in X$.

\begin{proposition}[{\cite[Proposition 5.4.5]{DK18}}]\label{prop:hariaw}
Let $\SM\in\Modhol(\SD_X)$ and $a\in X$.
If $Sol_X^{\rmE}(\SM)\in\rcEC^{\rmb}(\rmI\CC_X)$ has a normal 
form at $a\in X$ for a multiplicity $N\colon
\SP_{S_aX}^\prime\to(\ZZ_{\geq0})_{S_aX}$, then there exist an open neighborhood 
$\Omega$ of $a$ in $X$ and
an enhanced sheaf $F(a)\in\rcEC^0(\CC_X)$ such that 
$F(a)$ has a normal form at $a$ for the multiplicity $N$ and there exists an isomorphism
\begin{align}
\pi^{-1}\CC_{\Omega\bs\{a\}}\otimes Sol_X^{\rmE}(\SM) 
\simeq \CC_X^{\rmE}\Potimes F(a).
\end{align}
\end{proposition}

\begin{remark}\label{rem-K5-2}
For our use in the next section, we briefly recall the proof 
of Proposition \ref{prop:hariaw}. In fact, here we 
slightly modify the construction of the
$\RR$-constructible enhanced sheaf $F(a)$ in 
the proof of \cite[Proposition 5.4.5]{DK18} as follows.
First, we take sufficiently small open sectors 
$V_1,\dots,V_d$ along $a\in X$ placed in 
the counter-clockwise direction and 
satisfying the following conditions.
\begin{enumerate}
\item[{\rm(i)}] $\Omega\coloneq \{a\}\cup\bigcup_{j=1}^d V_j$ 
is an open neighborhood of $a$.
\item[{\rm(ii)}] For $1\leq j<j^\prime\leq d$, 
$V_j\cap V_{j^\prime}\neq\emptyset$ if and only if 
$j^\prime=j+1$ or $(j,j^\prime)=(1,d)$.
\item[{\rm(iii)}] For $1\leq j\leq d$, 
there exists an isomorphism
\begin{equation}
\pi^{-1}\CC_{V_j}\otimes Sol_X^{\rmE}(\SM)\simeq
\bigoplus_{f\in N^{>0}(V_j)}
\bigl(\EE_{V_j\vbar X}^{\Re f}\bigr)^{N(f)}.
\end{equation}  
\item[{\rm(iv)}] If $V_j\cap V_{j^\prime}\neq\emptyset$, 
then after renumbering the exponential factors  
$f_1,\dots,f_m\in N^{>0}(V_j\cap V_{j^\prime})$ on 
$V_j\cap V_{j^\prime}$ we have 
\begin{equation}
\Re f_1(z)<\Re f_2(z)<\dots<\Re f_m(z)
\quad \(z\in V_j\cap V_{j^\prime}\).  
\end{equation}
\end{enumerate}
Indeed, if each sector $V_j$ contains 
at most one Stokes line of $\SM$ 
and no two different sectors contain the same Stokes line, 
then the conditions (i) and (ii) imply the one (iv). 
Now we set 
\begin{equation}
F_j\coloneq\bigoplus_{f\in N^{>0}(V_j)}
\bigl(\sfE_{V_j\vbar V_j}^{\Re f}\bigr)^{N(f)} 
\quad \in\rcEC^0(\CC_{V_j})
\end{equation}
for $1\leq j\leq d$.
Then there exists an isomorphism in $\BEC(\rmI\CC_X)$
\begin{equation}
\Phi_j^{\rmE}\colon
\pi^{-1}\CC_{V_j}\otimes Sol_X^{\rmE}(\SM)\simto 
\CC_X^{\rmE}\Potimes\bfE\iota_{j!}F_j,
\end{equation}
where $\iota_j\colon V_j\hookrightarrow X$ 
is the natural embedding.
Thus we obtain an isomorphism in $\BEC(\rmI\CC_X)$
\begin{equation}\label{eq-K6}
\begin{split}
\(\pi^{-1}\CC_{V_j\cap V_{j^\prime}}\otimes
\Phi_{j^\prime}^{\rmE}\)&\circ\(\pi^{-1}
\CC_{V_j\cap V_{j^\prime}}\otimes\Phi_j^{\rmE}\)^{-1} \colon \\
&\ \CC_X^{\rmE}\Potimes
\(\pi^{-1}\CC_{V_j\cap V_{j^\prime}}
\otimes\bfE\iota_{j!}F_j\)\simto
\CC_X^{\rmE}\Potimes\(\pi^{-1}\CC_{V_j\cap V_{j^\prime}}
\otimes\bfE\iota_{j^\prime!}F_{j^\prime}\)
\end{split}
\end{equation}
for $1\leq j<j^\prime\leq d$ such that 
$V_j\cap V_{j^\prime}\neq\emptyset$.
Moreover by \cite[Lemma 5.2.1 (ii)]{DK18}, an isomorphism in $\BEC(\CC_X)$
\begin{equation}
\Phi_{jj^\prime}\colon\pi^{-1}\CC_{V_j\cap V_{j^\prime}} 
\otimes\bfE\iota_{j!}F_j\simto
\pi^{-1}\CC_{V_j\cap V_{j^\prime}}
\otimes\bfE\iota_{j^\prime!}F_{j^\prime}
\end{equation}
is induced by \eqref{eq-K6}. 
We set $R=\sum_{f\in N^{>0}(V_j\cap V_{j^\prime})} N(f)$.
Then by the condition (iv), 
the isomorphism $\Phi_{jj^\prime}$ 
is induced by a block upper triangular matrix 
$A_{jj^\prime}\in\mathrm{GL}_R(\CC)$ with respect to 
the decomposition $R=\sum_{f\in N^{>0}(V_j\cap V_{j^\prime})} N(f)$ of $R$. 
Gluing $F_j\vbar_{V_j\cap V_{j^\prime}}$'s by 
$\Phi_{jj^\prime}\vbar_{V_j\cap V_{j^\prime}}$'s, 
we obtain $F^\prime\in\rcEC^0(\CC_{\Omega\bs\{a\}})$.
Let $\iota\colon\Omega\bs\{a\}\hookrightarrow X$ 
be the natural embedding and set 
$F(a)\coloneq\bfE\iota_!F^\prime\in\rcEC^0(\CC_X)$.
Then one verifies that the $\RR$-constructible 
enhanced sheaf $F(a)$ satisfies the conditions in Proposition \ref{prop:hariaw}. 
\end{remark}

\begin{remark}\label{rem-K616}
By the proof of Proposition \ref{prop:hariaw}, we can 
also show that the 
enhanced sheaf $F(a)$ in it is unique up to isomorphisms. 
Indeed, this follows also from \cite[Lemma 5.2.1 (ii)]{DK18}. 
In particular, 
the isomorphism class of $F(a)$ does not depend 
on the choice of the open sectors $V_j$'s. 
\end{remark}

For $\SM\in\Modhol(\SD_X)$ and $a\in X$ such that 
$\SM(\ast a)\simeq\SM$ we say that $\SM$ is 
a meromorphic connection along $a\in X$. 
In this case, if $Sol_X^\rmE(\SM)$ has a normal form at 
$a\in X$ for a multiplicity 
$N\colon\SP_{S_aX}^\prime\to (\ZZ_{\geq0})_{S_aX}$, 
then for any $\theta\in S_aX$ the irregularity 
$\irr_a(\SM) \in \ZZ_{\geq 0}$ of the meromorphic 
connection $\SM$ is equal to 
\begin{equation}
    \sum_{f\in N^{>0}_{\theta}} N_{\theta}(f) \cdot \ord_a(f)
\end{equation}
(see \cite{Sab93} for an excellent review on this notion).
By the proofs of Proposition \ref{prop:hariaw} and 
\cite[Proposition 3.14]{IT20a} we obtain a purely topological proof of the 
following classical result (see \cite[Proposition 3.15]{IT20a}). 
For the meromorphic connection $\SM$ along $a\in X$ we set 
\begin{equation}
    \chi_a(Sol_X(\SM)) \coloneq 
    \sum_{j\in\ZZ}(-1)^j \dim_\CC H^j Sol_X(\SM)_a
\end{equation}
and call it the local Euler-Poincar\'{e} index of 
$Sol_X(\SM)$ at the point $a\in X$.

\begin{proposition}\label{prop:IT20a-3.14}
    In the situation as above, we have 
    \begin{equation}
        \chi_a(Sol_X(\SM)) = -\irr_a(\SM).
    \end{equation}
\end{proposition}
\begin{proof}
    In view of Remark \ref{rem-K5-2}, 
    the proof is very similar to 
    that of \cite[Proposition 3.14]{IT20a}. 
    To kill the monodromies of the enhanced sheaf $F(a)$ 
    in Remark \ref{rem-K5-2}, 
    it suffices to use Mayer-Vietoris exact sequences 
    associated to the covering 
    $\Omega\setminus\{a\}=\bigcup_{j=1}^d V_j$.
\end{proof}

\subsection{Legendre transform for Puiseux germs}
\label{legendre}
Let $X=\CC_z$ be the complex affine line and $Y=\CC_w$ its dual. 
We denote by $\var{X}\simeq\PP^1$ (resp. $\var{Y}\simeq\PP^1$) 
the projective compactification of $X$ (resp. $Y$) and $X^\an$ (resp. 
$Y^\an$, $\var{X}^\an$ and $\var{Y}^\an$) the underlying complex 
manifold of $X$ (resp. $Y$, $\var{X}$ and $\var{Y}$).
We define the symplectic transform $\chi\colon T^\ast X^\an
\simto T^\ast Y^\an$ by
\begin{equation}
T^\ast X^\an= X^\an\times Y^\an\ni(z,w)\longmapsto 
(w,-z)\in Y^\an\times X^\an=T^\ast Y^\an.
\end{equation}
In this subsection, we recall an important correspondence between 
the Puiseux germs on $X^\an$ and $Y^\an$. 
Let $a\in X^\an$ be a point in $X^\an$, $W\subset 
X^\an\setminus\{a\}$ an (open) sector 
along $a$ and $f\colon W\to\CC$ a holomorphic function on $W$. 
Then we define a Lagrangian submanifold $\Lambda_a^f$ of $T^\ast X^\an$ by 
\begin{equation}
\Lambda_a^f\coloneq\{(z,f^\prime(z))\mid z\in W\} \quad \subset T^\ast X^\an.
\end{equation}
Let $g\colon V\rightarrow\CC$ be a holomorphic function 
on an (open) sector $V\subset Y^\an$ along the point $\infty\in 
\var{Y}^\an$ and set $\Lambda^g\coloneq\{(w,g^\prime(w))\mid w\in 
V\}\subset T^\ast Y^\an\simeq Y^\an\times X^\an$. 
Then we have the following result due to D'Agnolo-Kashiwara
(see the proof of \cite[Lemma-Definition 7.4.2]{DK18}). 
For the convenience of the 
readers, we recall also their proof. 

\begin{lemma}[{\cite[Lemma-Definition 7.4.2]{DK18}}]\label{lem-T7}
\begin{enumerate}
\item[\rm{(i)}] Assume that
$\Lambda^{g}\subset\chi(T_{\{a\}}^\ast X^{\an})$.
Then on $V$ we have 
\begin{align}
g(w) \equiv -aw 
\end{align}
modulo constant functions on the open sector $V$.
\item[\rm{(ii)}] 
Assume that $\Lambda^{g}\subset\chi(\Lambda_a^f)$.
Then there exists a unique holomorphic function
$\zeta\colon V\rightarrow W \subset\CC=X^{\an}$
such that
\begin{align} \label{triveq} 
\chi^{-1}(\Lambda^{g})= \Set*{(\zeta(w),w)}{w\in V}
\subset \Lambda_a^f\subset X^{\an}\times Y^{\an}
\end{align}
and on $V$ we have 
\begin{align}
g(w) \equiv f(\zeta(w))-\zeta(w)\cdot w
=-f^w(\zeta(w))
\end{align}
modulo constant functions on the sector $V$, where the 
holomorphic function $f^w$ on $W$ is defined by
\begin{equation}
f^w(z)\coloneq zw-f(z) \quad (z\in W).
\end{equation}
\end{enumerate}
\end{lemma}

\begin{proof}
Since (i) is trivial, we only prove (ii).
The unique existence of
$\zeta\colon V\rightarrow W$ 
is also straightforward. 
Recall that the condition $(\zeta(w),w)\in\Lambda_a^f$ is
equivalent to the ones
\begin{align}\label{crieqi}   
f^\prime(\zeta(w))=w
 \quad
\Longleftrightarrow \quad 
(f^w)^\prime(\zeta(w))=0.
\end{align}
This means that $\zeta(w)\in W$ is
a critical point of the holomorphic function $f^w$.
Let us set
\begin{align}u(w) \coloneq f(\zeta(w))-\zeta (w)\cdot w
\quad(w\in V).\end{align}
Then by \eqref{triveq} and \eqref{crieqi} we obtain 
\begin{align}u^\prime(w)
= -\zeta(w)
= g^\prime(w) \quad(w\in V),\end{align}
from which the assertion immediately follows.
\end{proof}
We call the correspondence between the holomorphic 
functions $f$ and $g$ in Lemma \ref{lem-T7} (ii) 
a Legendre transform. Also for holomorphic functions 
defined on sectors along the point $\infty\in\var{X}^\an$ 
we have a similar correspondence and call it 
a Legendre transform. 
In \cite{DK18}, D'Agnolo and Kashiwara refined 
this result for Puiseux germs as follows. 
For a point $a\in\var{X}^{\an}$, we define the \'{e}tal\'{e} space
$\mathrm{\acute{e}t}(\SP_{S_a\var{X}^{\an}})$
endowed with the natural topology by
\begin{equation}
\mathrm{\acute{e}t}(\SP_{S_a\var{X}^{\an}})\coloneq
\bigsqcup_{\theta\in S_a\var{X}^{\an}}
\SP_{S_a\var{X}^{\an},\theta}.
\end{equation}

\begin{lemma}[{\cite[Lemma 7.4.5]{DK18}}]\label{lem-T2-1}
The Legendre transform induces homeomorphisms of
\'{e}tal\'{e} spaces as follows. 
\begin{enumerate}
\item[{\rm(i)}] If $a\neq\infty$, then it induces 
a homeomorphism 
\begin{equation}
\sfL_{(a,\infty)}\colon\mathrm{\acute{e}t}(\SP_{S_a 
\var{X}^{\an}}^{(0,+\infty)})\simto
\mathrm{\acute{e}t}(-aw+\SP_{S_\infty\var{Y}^{\an}}^{(0,1)})
\end{equation}
such that for $f\in\SP_{S_a\var{X}^{\an},\theta}$ 
$(\theta\in S_a\var{X}^{\an})$ we have $\chi(\Lambda_a^f)=
\Lambda^{\sfL_{(a,\infty)}(f)}$. 
\item[{\rm(ii)}] 
It induces homeomorphisms 
\begin{align}
\sfL_{(\infty,b)}&\colon\mathrm{\acute{e}t}(bz+
\SP_{S_\infty\var{X}^{\an}}^{(0,1)})\simto
\mathrm{\acute{e}t}(\SP_{S_bY^{\an}}^{(0,+\infty)}), \\
\sfL_{(\infty,\infty)}&\colon\mathrm{\acute{e}t}(
\SP_{S_\infty\var{X}^{\an}}^{(1,+\infty)})\simto
\mathrm{\acute{e}t}(\SP_{S_\infty \var{Y}^{\an}}^{(1,+\infty)}),
\end{align}
which satisfy the conditions similar to the one in (i).
\end{enumerate}
\end{lemma}
We call the correspondence in Lemma \ref{lem-T2-1} 
the Legendre transform for Puiseux germs.

\section{Enhanced solution complexes of 
holonomic $\SD$-modules in dimension one}
\label{uni-sec:T1}

In this section, we focus our attention on
holonomic $\SD$-modules in dimension one
and give an explicit description of
their enhanced solution complexes,
which will be effectively used in subsequent sections.
Let $X$ be a closed Riemann surface
i.e. a compact complex manifold of dimension one
and $\SM$ an analytic holonomic $\SD$-module on it.
For simplicity, we denote by $D\subset X$
the singular support $\mathrm{sing.supp}(\SM)$ of $\SM$
and set $\SN\coloneq\SM(\ast D)\in\Modhol(\SD_X)$.
The complex curve $X$ being compact,
$D\subset X$ is a finite subset of $X$.
Our objective here is to give a global
and explicit description of the enhanced solution complex
$Sol_X^{\rmE}(\SN)\in \rcEC^{\rmb}(\mathrm{I}\CC_X)$ of
the meromorphic connection $\SN$.
Recall that by Proposition \ref{prop:hariaw}
for each point $a\in D$ there exist a multiplicity
\begin{align}
N(a)\colon\SP_{S_aX}^{\prime}
\longrightarrow \(\ZZ_{\geq 0}\)_{S_aX}
\end{align}
and an enhanced sheaf $F(a)\in \rcEC^0(\CC_X)$ with
normal form at $a\in D$ for the multiplicity $N(a)$ such that
for a sufficiently small closed disk $D(a)\subset X$
centered at $a\in D$
there exists an isomorphism
\begin{align}
\pi^{-1}\CC_{D(a) \setminus \{ a \}} \otimes Sol_X^{\rmE}(\SN)
\simeq \CC_X^{\rmE}\Potimes F(a).
\end{align}
Then by Proposition \ref{prop-K2} (i) we obtain an isomorphism 
\begin{align}
\pi^{-1}\CC_{D(a)} \otimes Sol_X^{\rmE}(\SN)
\simeq \CC_X^{\rmE}\Potimes F(a).
\end{align}
By Remark \ref{rem-K5-2}
we may assume that
there exist open sectors $V_1,V_2,\dots,V_d\subset X$
along $a\in D$ placed in
the counter-clockwise direction such that
\begin{align}
D(a)^{\circ}:= 
D(a)\bs\{a\} \subset V_1\cup V_2\cup\dots\cup V_d
\end{align}
and for $1\leq j<j^\prime\leq d$
we have $V_j\cap V_{j^\prime}\neq\emptyset$
if and only if $j^\prime=j+1$ or $(j,j^\prime)=(1,d)$
for which we have isomorphisms
\begin{align}
\Phi_j \colon\pi^{-1}\CC_{V_j}\otimes F(a)
\simto \bigoplus_{f\in N(a)^{>0}(V_j)}
\(\sfE_{V_j\cap D(a)\vbar X}^{\Re f}\)^{N(a)(f)}
\quad \(1\leq j\leq d\)
\end{align}
of enhanced sheaves on $X$.
Note that the integer $\sum_{f \in N(a)^{>0}(V_j)} N(a)(f) 
\geq 0$ does not depend on $j$. We denote it by $R$. 
It is clear that $R$ is equal to the rank of 
the meromorphic connection $\SN$. Let us call it 
the generic rank of $\SM$. 
Moreover by Remark \ref{rem-K5-2} we may assume also that
for any $1\leq j<j^\prime\leq d$ such that 
$V_j\cap V_{j^\prime}\neq\emptyset$
after renumbering the exponential factors 
$f_1,f_2,\dots,f_m\in N(a)^{>0}(V_j\cap V_{j^\prime})$ of $\SN$
on the open sector $V_j \cap V_{j^\prime}$
they satisfy the condition
\begin{align}
\Re f_1(z)<\Re f_2(z)<\dots<\Re f_m(z)
\quad \(z\in V_j\cap V_{j^\prime}\)
\end{align}
and the automorphism
\begin{equation}
\begin{split}
\(\pi^{-1}\CC_{V_j\cap V_{j^\prime}}\otimes\Phi_{j^\prime}\)
&\circ \(\pi^{-1}\CC_{V_j\cap V_{j^\prime}}\otimes\Phi_j\)^{-1} 
\colon \\
&\ 
\bigoplus_{i=1}^m
\(\sfE_{V_j\cap V_{j^\prime}\cap D(a)\vbar X}^{\Re f_i}\)^{N(a)(f_i)}
\simto
\bigoplus_{i=1}^m
\(\sfE_{V_j\cap V_{j^\prime}\cap D(a)\vbar X}^{\Re f_i}\)^{N(a)(f_i)}
\end{split}
\end{equation}
of the enhanced sheaf
\begin{align}
\bigoplus_{i=1}^m
\(\sfE_{V_j\cap V_{j^\prime}\cap D(a)\vbar X}^{\Re f_i}\)^{N(a)(f_i)}
\simeq 
\bigoplus_{i=1}^m
\CC_{\left\{ z\in V_j\cap V_{j^\prime}\cap D(a),\,
t+\Re f_i(z)\geq 0\right\}}
^{\oplus N(a)(f_i)}
\end{align}
is induced by
a block upper triangular matrix $A_{j j^\prime}\in \mathrm{GL}_R(\CC)$
with respect to the decomposition $R= \sum_{i=1}^m N(a)(f_i)$
of $R$. From now on, we shall explicitly construct
an enhanced sheaf $G\in \rcEC^0(\CC_X)$ on $X$ such that
\begin{align}
Sol_X^{\rmE}(\SN) \simeq \CC_X^{\rmE}\Potimes G
\end{align}
and for any $a\in D$ we have
\begin{align}
\pi^{-1}\CC_{D(a)}\otimes G
\simeq \pi^{-1}\CC_{D(a)}\otimes F(a).
\end{align}
For this purpose, set $U \coloneq X\bs D\subset X$
and let
\begin{align}
L \coloneq Sol_U(\SN\vbar_U)
\simeq Sol_U(\SM\vbar_U)
\end{align}
be the local system associated to
the integrable connection $\SN\vbar_U\simeq \SM\vbar_U$ on $U$.
For each point $a\in D$
we take a closed disk $D(a)^{\prime}\subset X$ centered at it
such that
$D(a)^{\prime}\subset \Int(D(a))$
and define a closed subset $E\subset X$ by
\begin{align}
E \coloneq X\bs\left\{\bigcup_{a\in D}
\Int(D(a)^{\prime})\right\}.
\end{align}
Then for $a\in X$ the closed subset
\begin{align}
A(a) \coloneq D(a) \cap E
= D(a)\bs\Int(D(a)^{\prime})
\quad \subset X
\end{align}
of $X$ is an annulus centered at it.
Since $\SN$ is an integrable connection and 
hence regular on a neighborhood of $E$ in $X$,
by Proposition \ref{prop-K2} (iv) we obtain an isomorphism
\begin{align}
\pi^{-1}\CC_{E}\otimes Sol_X^{\rmE}(\SN)
\simeq \CC_X^{\rmE}\Potimes \epsilon(L_E).
\end{align}
For $a\in D$ let $L(a)$ be the local system on $D(a)^{\circ}$
defined by gluing the constant sheaves
\begin{align}
\bigoplus_{f \in N(a)^{>0}(V_j)} 
\CC_{V_j\cap D(a)}^{\oplus N(a)(f)}
\end{align}
on $V_j \cap D(a) \subset D(a)^{\circ}$ $\(1\leq j\leq d\)$ by
the block upper triangular matrices 
$A_{j j^{\prime}}\in\mathrm{GL}_R(\CC)$ 
($V_j \cap V_{j^{\prime}} \not= \emptyset$) 
and $j(a) \colon D(a)^{\circ} 
\hookrightarrow X$ the inclusion map.
Then by our construction of
the enhanced sheaf $F(a)\in\rcEC^0(\CC_X)$
there exists a surjective morphism
\begin{align}
\pi^{-1}j(a)_!L(a) \longrightarrow F(a).
\end{align}
Let us take a sufficiently large $c\gg0$
satisfying the condition
\begin{align}
c > \max_{f\in N(a)^{>0}(V_j)}
\left\{  \max_{z\in A(a)\cap\overline{V_j}} \(-\Re f(z)\)\right\}
\end{align}
for any $1\leq j\leq d$.
Then there exists a surjective morphism
\begin{align}
\pi^{-1}\CC_{A(a)}\otimes F(a)
\longrightarrow
\CC_{\{t\geq c\}}\otimes \pi^{-1} \Bigl( j(a)_!L(a) \Bigr)_{A(a)}
\end{align}
and applying $\pi_{\ast}$ to it
we obtain an isomorphism
\begin{equation}\label{eq-T1}
\pi_{\ast}F(a)\vbar_{A(a)} \simto L(a)\vbar_{A(a)}
\end{equation}
of sheaves on the annulus $A(a)$.

\begin{lemma}\label{lem-T1}
There exists an isomorphism
\begin{align}
L\vbar_{A(a)} \simto L(a)\vbar_{A(a)}\end{align}
of sheaves on $A(a)$.
\end{lemma}

\begin{proof}
Let $i_0\colon X\hookrightarrow X\times\RR$
be the inclusion map defined by $i_0(x) \coloneq(x,0)$.
Then by \cite[Lemma 9.5.5]{DK18} there exists an isomorphism
\begin{align}
Sol_X(\SN)
\simeq \alpha_Xi_0^!\bfR^{\rmE}\(Sol_X^{\rmE}(\SN)\). 
\end{align}
For $1\leq j\leq d$ the restriction of
$\bfR^{\rmE}(\CC_X^{\rmE}\Potimes F(a))$ to $V_j \times \RR$
is isomorphic to
\begin{align}
\underset{\alpha \to+\infty}{\inj}
\(\bigoplus_{f\in N(a)^{>0}(V_j)}
\CC_{\{z\in V_j\cap D(a),\,t< \alpha -\Re f(z)\}}^{\oplus N(a)(f)}[1]\)
\end{align}
(see e.g. the proof of \cite[Corollary 4.5]{IT20a}). 
Since the functor $i_0^!$ commutes with limits $\inj$, 
as in the proof of \cite[Corollary 4.5]{IT20a} 
we can easily see that there exist isomorphisms
\begin{align} 
L\vbar_{A(a)}
= Sol_X(\SN)\vbar_{A(a)}
\simeq \pi_{\ast}F(a)\vbar_{A(a)}. 
\end{align}
Then the assertion immediately follows from
the isomorphism \eqref{eq-T1}.
\end{proof}

By Lemma \ref{lem-T1} we obtain a surjective morphism 
\begin{equation}
\pi^{-1}(L \vbar_{A(a)}) \simeq 
\pi^{-1}(L(a) \vbar_{A(a)}) \longrightarrow 
F(a) \vbar_{\pi^{-1} A(a)}
\end{equation}
of sheaves on $\pi^{-1} A(a)=A(a) \times \RR$ obtained 
by cutting the support of the local system $\pi^{-1}(L \vbar_{A(a)})$. 
Namely $F(a) \vbar_{\pi^{-1} A(a)}$ is a quotient of 
$\pi^{-1}(L \vbar_{A(a)})$. Similarly, for 
sufficiently large $c\gg0$ we can also 
construct a quotient $G_0 \in\rcEC^0(\CC_E)$ of 
the local system $\pi^{-1}(L \vbar_{E})$ satisfying 
the conditions 
\begin{equation}
G_0 \vbar_{\pi^{-1}(X \setminus \cup_{a \in D} D(a))} \simeq 
\CC_{\{t\geq c\}} \otimes \pi^{-1} 
(L \vbar_{(X \setminus \cup_{a \in D} D(a))}) 
\end{equation}
and 
\begin{equation}
G_0 \vbar_{\pi^{-1}A(a)} \simeq 
F(a) \vbar_{\pi^{-1} A(a)} \qquad (a \in D). 
\end{equation}
Gluing $G_0 \in\rcEC^0(\CC_E)$ and 
$F(a) \vbar_{\pi^{-1} D(a)} \in\rcEC^0(\CC_{D(a)})$ 
over the sets $\pi^{-1} A(a)=A(a) \times \RR$ 
($a \in D$) we obtain an enhanced sheaf 
$G\in\rcEC^0(\CC_X)$ globally defined on the whole $X$.
Now let us define a closed subset $E^{\prime}\subset X$ by
\begin{align}
E^{\prime}
\coloneq \bigsqcup_{a\in D}D(a).
\end{align}
Note that we have $X=E^\prime\cup E$ and $E^\prime\cap 
E=\displaystyle\sqcup_{a\in D}A(a)$.
Then there exist distinguished triangles 
\begin{align}
    Sol_X^\rmE(\SN)\longrightarrow\bigl(\pi^{-1}
\CC_E \otimes Sol_X^\rmE(\SN)\bigr)
    &\oplus\bigl(\pi^{-1}\CC_{E^\prime}\otimes 
Sol_X^\rmE(\SN)\bigr) \notag \\
    &\longrightarrow\bigoplus_{a\in D}\(\pi^{-1}\CC_{A(a)}
\otimes Sol_X^\rmE(\SN)\)\overset{+1}{\longrightarrow}
\end{align}
and
\begin{align}
    \CC_X^\rmE\Potimes G\longrightarrow\bigl(\pi^{-1}
\CC_E\otimes (\CC_X^\rmE\Potimes G)\bigr)
    &\oplus\bigl(\pi^{-1}\CC_{E^\prime}\otimes
 (\CC_X^\rmE\Potimes G)\bigr) \notag \\
    &\longrightarrow\bigoplus_{a\in D}\bigl(\pi^{-1}\CC_{A(a)}
\otimes (\CC_X^\rmE\Potimes G)\bigr)\overset{+1}{\longrightarrow}
\end{align}
associated to the exact sequence
\begin{equation}
    0\longrightarrow\CC_X\longrightarrow\CC_E\oplus\CC_{E^\prime}
    \longrightarrow\bigoplus_{a\in D}\CC_{A(a)}\longrightarrow0.
\end{equation}
Moreover, by the proof of Lemma \ref{lem-T1} and Lemma \ref{lem-K1} 
there exists a commutative diagram
\begin{equation}
    \vcenter{
    \xymatrix@M=5pt@R=18pt@C=14pt{
    \bigl(\pi^{-1}\CC_E \otimes Sol_X^\rmE(\SN)\bigr)
    \oplus\bigl(\pi^{-1}\CC_{E^\prime}\otimes Sol_X^\rmE(\SN)\bigr)
    \ar[r] \ar[d]^-[@!-90]{\sim} &
    \displaystyle\xyoplus_{a\in D}\bigl(
    \pi^{-1}\CC_{A(a)}\otimes Sol_X^{\rmE}(\SN)\bigr)
    \ar[d]^-[@!-90]{\sim}  \\
    \bigl(\pi^{-1}\CC_E\otimes (\CC_X^\rmE\Potimes G)\bigr)
    \oplus\bigl(\pi^{-1}\CC_{E^\prime}\otimes (\CC_X^\rmE\Potimes G)\bigr)
    \ar[r] &
    \displaystyle\xyoplus_{a\in D}\bigl(
    \pi^{-1}\CC_{A(a)}\otimes(\CC_X^{\rmE}\Potimes G)\bigr).
    }}
\end{equation}
Then by the axiom (TR 4) (see \cite[Definition 1.5.1]{KS90}) 
we obtain an isomorphism
\begin{equation}
    Sol_X^\rmE(\SN)\simeq \CC_X^\rmE\Potimes G.
\end{equation}

\section{Fourier transforms of holonomic $\SD$-modules 
in dimension one and their characteristic cycles}
\label{uni-sec:T2}

In this section, we give some refinements of
the results of D'Agnolo-Kashiwara \cite{DK18} and \cite{DK23} on
the exponential factors of Fourier transforms of
holonomic $\SD$-modules in dimension one.
With our explicit description of
their enhanced solution complexes obtained in
Section \ref{uni-sec:T1} at hands,
we can now apply a twisted Morse theory
(with several Morse functions)
to obtain more precise structures of their Fourier transforms.

\subsection{Fourier transforms of 
holonomic $\SD$-modules}\label{sec:T3}

First of all, we recall the definition of 
Fourier transforms of $\SD$-modules and their basic properties 
and explain their relations with 
the irregular Riemann-Hilbert correspondence of 
D'Agnolo-Kashiwara \cite{DK16}.
Let $X=\CC_z^N$ be the $N$-dimensional complex vector space 
and $Y=\CC_w^N$ its dual space.
Here we first consider them as 
smooth algebraic varieties over $\CC$ 
endowed with the Zariski topology.
We use the notations $\SD_X$ and $\SD_Y$ for 
the rings of ``algebraic'' differential operators on them.
Denote by $\Modcoh(\SD_X)$ (resp. $\Modhol(\SD_X)$, 
$\Modrh (\SD_X)$) 
the category of coherent (resp. holonomic, 
regular holonomic) $\SD_X$-modules. 
Let $W_N := \CC[z, \partial_z]\simeq\Gamma(X; \SD_X)$ and
$W^\ast_N := \CC[w, \partial_w]\simeq\Gamma(Y; \SD_Y)$
be the Weyl algebras over $X$ and $Y$, respectively.
Then by the ring isomorphism
\begin{equation}
W_N\simto W^\ast_N\hspace{30pt}
(z_i\mapsto-\partial_{w_i},\ \partial_{z_i}\mapsto w_i) 
\end{equation}
we can endow a left $W_N$-module $M$ with 
a structure of a left $W_N^\ast$-module.
We call it the Fourier transform of $M$
and denote it by $M^\wedge$.
For a ring $R$ we denote by $\Mod_f(R)$
the category of finitely generated $R$-modules.
Recall that for the affine algebraic varieties $X$
and $Y$ we have the equivalences of categories
\begin{align} 
\Modcoh(\SD_X)
&\simeq
\Mod_f(\Gamma(X; \SD_X)) = \Mod_f(W_N),\\
\Modcoh(\SD_Y)
&\simeq
\Mod_f(\Gamma(Y; \SD_Y)) = \Mod_f(W^\ast_N)
\end{align}
(see e.g. \cite[Propositions 1.4.4 and 1.4.13]{HTT08}).
For a coherent $\SD_X$-module $\SM\in\Modcoh(\SD_X)$
we thus can define its Fourier 
transform $\SM^\wedge\in\Modcoh(\SD_Y)$.
It follows that we obtain an equivalence of categories
\begin{equation*}
( \cdot )^\wedge : \Modhol(\SD_X)\simto \Modhol(\SD_Y)
\end{equation*}
between the categories of holonomic $\SD$-modules 
(see \cite[Proposition 3.2.7]{HTT08}). Let
\begin{align}
X\overset{p}{\longleftarrow}X\times 
Y\overset{q}{\longrightarrow}Y
\end{align}
be the projections.
Then by Katz-Laumon \cite{KL85}, we have the following lemma.

\begin{lemma}\label{lem-DFS}
For a holonomic $\SD_X$-module
$\SM\in\Modhol(\SD_X)$, we have an isomorphism
\begin{align}
\SM^\wedge\simeq
\bfD q_\ast(\bfD p^\ast\SM\Dotimes \SO_{X\times Y}
e^{- \langle z, w \rangle }),
\end{align}
where $\bfD p^\ast, \bfD q_\ast, \Dotimes$ are
the operations for algebraic $\SD$-modules
and $\SO_{X\times Y}
e^{- \langle z, w \rangle }$ stands for the integral connection
of rank one on $X\times Y$ associated to the canonical
paring $\langle \cdot ,  \cdot \rangle : X\times Y\to\CC$. 
In particular the right hand side is concentrated in degree zero.
\end{lemma}

Let $\var{X}\simeq\PP^N$ (resp. $\var{Y}\simeq\PP^N$)
be the projective compactification of $X$ (resp. $Y$).
By the inclusion map $i_X : 
X=\CC^N\xhookrightarrow{\ \ \ }\var{X}=\PP^N$
we extend a holonomic $\SD_X$-module $\SM\in\Modhol(\SD_X)$
to $\tl{\SM} := i_{X\ast}\SM\simeq\bfD i_{X\ast}\SM 
\in\Modhol(\SD_{\var{X}})$. 
Denote by $\var{X}^{\an}$ the underlying complex manifold
of $\var{X}$ and define the analytification
$\tl{\SM}^{\an}\in\Modhol(\SD_{\var{X}^{\an}})$
of $\tl{\SM}$ by
$\tl{\SM}^{\an} = \SO_{\var{X}^{\an}}
\otimes_{\SO_{\var{X}}}\tl{\SM}$. 
Then we set 
\begin{align}
Sol_{\var X}^\rmE(\tl\SM) := 
Sol_{\var{X}^{\an}}^\rmE(\tl{\SM}^{\an})
\qquad \in\BEC(\I\CC_{\var{X}^{\an}}).
\end{align}
Similarly for the Fourier transform $\SM^\wedge\in\Modhol(\SD_Y)$ 
we define $Sol_{\var Y}^\rmE(\tl{\SM^\wedge})
\in\BEC(\I\CC_{\var{Y}^{\an}})$.
Let
\begin{align}
\var{X}^{\an}\overset{\var{p}}{\longleftarrow}
\var{X}^{\an}\times\var{Y}^{\an}\overset{\var{q}}{\longrightarrow}
\var{Y}^{\an}
\end{align}
be the projections. 

\begin{definition}\label{FS-trans} 
For $F\in\BEC(\CC_{\var{X}^{\an}})$ and 
$\SF\in\BEC(\I\CC_{\var{X}^{\an}})$, 
we define their Fourier-Sato 
(Fourier-Laplace) transforms 
${}^\sfL F\in\BEC(\CC_{\var{Y}^{\an}})$ and 
${}^\sfL \SF\in\BEC(\rmI\CC_{\var{Y}^{\an}})$ by
\begin{align}
{}^\sfL F &\coloneq \bfE\var{q}_{*}(\bfE\var{p}^{-1}F\Potimes
\sfE_{X\times Y|\var{X}\times\var{Y}}^{
-{\Re} \langle z, w \rangle}[N])
\qquad \in\BEC(\CC_{\var{Y}^{\an}}), \\
{}^\sfL\SF &\coloneq \bfE\var{q}_{*}(\bfE\var{p}^{-1}\SF
\Potimes\EE_{X\times Y|\var{X}\times\var{Y}}^{
-{\Re} \langle  z, w \rangle }[N])
\qquad \in\BEC(\I\CC_{\var{Y}^{\an}} )
\end{align}
respectively, where we denote $X^{\an}\times Y^{\an}$ etc.
by $X\times Y$ etc. for short.
\end{definition}

Note that these transforms preserve 
the $\RR$-constructibility.
Namely we obtain functors 
${}^{\sfL}(\cdot)\colon$
$\rcEC^{\rmb}(\CC_{\var{X}^{\an}})\to
\rcEC^{\rmb}(\CC_{\var{Y}^{\an}})$ and 
${}^{\sfL}(\cdot)\colon
\rcEC^{\rmb}(\rmI\CC_{\var{X}^{\an}})\to
\rcEC^{\rmb}(\rmI\CC_{\var{Y}^{\an}})$. 
The following lemma is essentially due to
Kashiwara-Schapira \cite{KS16a} and D'Agnolo-Kashiwara \cite{DK18} 
(see also \cite[Theorem 4.2]{IT20a}). 

\begin{lemma}\label{lem-FS}
For $\SM\in\Mod_{\rm hol}(\SD_X)$ there exists an isomorphism
\begin{align}
Sol_{\var Y}^\rmE(\tl{\SM^\wedge})
\simeq{}^\sfL Sol_{\var X}^\rmE(\tl{\SM}).
\end{align}
\end{lemma}

Note that by \cite[Proposition 4.7.14]{DK16} we have 
\begin{align}
{}^{\sfL}\Bigl(\CC_{\var{X}^{\an}}^{\rmE}
\Potimes(\cdot)\Bigr)\simeq\CC_{\var{Y}^{\an}}^{\rmE}
\Potimes{}^{\sfL}(\cdot).
\end{align}
Therefore, if there exists 
$G\in\rcEC^{\rmb}(\CC_{\var{X}^{\an}})$ such that 
\begin{align}
Sol_{\var{X}}^{\rmE}(\tl{\SM})\simeq
\CC_{\var{X}^{\an}}^{\rmE}\Potimes G,
\end{align}
then we obtain an isomorphism
\begin{align}\label{undls} 
Sol_{\var{Y}}^{\rmE}(\tl{\SM^\wedge})
\simeq \CC_{\var{Y}^{\an}}^{\rmE}\Potimes{}^\sfL G
\end{align}
and hence 
for the study of $Sol_{\var{X}}^{\rmE}(\tl{\SM^\wedge})$ 
it suffices to study ${}^\sfL G\in
\rcEC^{\rmb}(\CC_{\var{Y}^{\an}})$.

From now on, we assume that
the dimension $N$ of $X$ and $Y$ is one.
Let $\SM\in\Modhol(\SD_X)$ be
an algebraic holonomic $\SD$-module on
the affine line $X=\CC_z$
and denote by $D\subset X$
its singular support $\mathrm{sing.supp}(\SM)$.
Then there exists a distinguished triangle
\begin{equation}
\rsect_D(\SM) \longrightarrow \SM \longrightarrow
\rsect_{X\bs D}(\SM) \overset{+1}{\longrightarrow}.
\end{equation}
Since $U:= X\bs D$ is an affine open subset of $X$,
we have
\begin{align}
H^j\rsect_U(\SM) \simeq
\begin{cases}
\ \sect_U(\SM) & (j=0), \\
\ 0 & (j\neq0)
\end{cases}
\end{align}
and $\sect_U(\SM)$ is nothing but
the localization of $\SM$
along the divisor $D\subset X$.
We thus obtain an exact sequence
\begin{equation}
0 \longrightarrow \sect_D(\SM) \longrightarrow
\SM \longrightarrow \sect_U(\SM) \longrightarrow
H^1\rsect_D(\SM) \longrightarrow 0.
\end{equation}
Note that $\sect_D(\SM)$ and $H^1\rsect_D(\SM)$
are supported in the finite set $D\subset X$
and hence their Fourier transforms are 
integrable connections on $X= \CC$. 
Since the Fourier transform $(\cdot)^\wedge$ is an exact functor,
this implies that for the initial study of $\SM^\wedge$
it suffices to study the Fourier transform of $\sect_U(\SM)$.
For this reason, in this paper we first assume that 
\begin{align}\label{eq:localize}
\SM \simto \sect_U(\SM)
\end{align}
and discuss the general case in Section \ref{sec:holD}. 
We call $\SM\in \Modhol(\SD_X)$ satisfying (\ref{eq:localize})
a localized holonomic $\SD$-module
or a meromorphic connection on $X=\CC_z$. 
We denote its generic rank by ${\rm rk} \SM$.

\subsection{The proof of Theorem \ref{Thm-T3} 
and related results}
\label{sec:T4}

For the proof of Theorem \ref{Thm-T3},
recall the notation in Section \ref{uni-sec:0}. 
If an exponential factor $f \in N_i^{>0}$ at a point $a_i \in 
\tl{D}$ is holomorphic on a sector $S$ along it, 
by abuse of notations we write $f \in N_i^{>0}(S)$. 
For such $f\in N_i^{>0}(S)$ and $w \in Y^{\an}= \CC$ 
we define a holomorphic function $f^w$ on $S$ by
\begin{align}
f^w(z) \coloneq zw-f(z) \quad \(z\in S \).
\end{align}
For $a_i\in D^{\an}$ let $D(a_i)$ be
a closed disk centered at it 
such that $D(a_i)^\circ := D(a_i) \setminus \{ a_i \}
\subset B(a_i)^\circ$. Shrinking $B(a_i)^\circ$ if 
necessary, we may assume that 
for any non-zero exponential 
factor $f \in N_i^{>0}$ of $\SM^{\an}$ at $a_i$ 
the morphism $f^{\prime}: B(a_i)^{\circ} \longrightarrow 
Y^{\an}= \CC$ is an unramified finite covering 
over the punctured disk $B(b_{\infty})^{\circ}$. 
Recall that we have 
\begin{equation}
(f^w)^\prime(z)=0  \quad \iff \quad 
f^\prime(z)=w.
\end{equation} 
Then we can take the disk 
$D(a_i)$ so that for any non-zero $f \in  N_i^{>0}$ 
and $w \in V \subset B(b_{\infty})^{\circ}$ 
all the critical points of the 
(possibly multi-valued) holomorphic 
function $f^w$ on $B(a_i)^\circ$ are contained in 
${\rm Int} D(a_i)^{\circ}$. 

Then as in Section \ref{uni-sec:T1}
we can explicitly construct an enhanced sheaf
$F_i=F(a_i)\in\rcEC^0(\CC_{\var{X}^{\an}})$ on
$\var{X}^{\an}$ such that
\begin{align}
\pi^{-1}\CC_{D(a_i)^\circ}\otimes
Sol_{\var{X}}^{\rmE}(\tl{\SM})
\simeq \CC_{\var{X}^{\an}}^{\rmE}\Potimes F_i.
\end{align}
Similarly, for the point $a_\infty=\infty\in\tl{D}$
we take a closed disk
$D(a_\infty)$ centered at it such that 
$D(a_{\infty})^\circ := D(a_{\infty}) \setminus \{ a_{\infty} \}
\subset B(a_{\infty})^\circ$, 
and construct 
$F_\infty=F(a_\infty)\in\rcEC^0(\CC_{\var{X}^{\an}})$ 
for which we have an isomorphism
\begin{align}
\pi^{-1}\CC_{D(a_{\infty})^\circ}\otimes
Sol_{\var{X}}^{\rmE}(\tl{\SM})
\simeq \CC_{\var{X}^{\an}}^{\rmE}\Potimes F_\infty.
\end{align}
We set $L\coloneq Sol_U(\SM\vbar_U)$.
As in Section \ref{uni-sec:T1},
then by gluing $F_i$ $(1\leq i\leq l)$, $F_\infty$ and 
$\CC_{\{t\geq c\}}\otimes\pi^{-1}L$ ($c\gg0$),
we obtain an enhanced sheaf
$G\in\rcEC^0(\CC_{\var{X}^{\an}})$ such that
\begin{align}
Sol_{\var{X}}^{\rmE}(\tl{\SM})
\simeq \CC_{\var{X}^{\an}}^{\rmE}\Potimes G.
\end{align}
It follows from the results in Section \ref{sec:T3}  
also that there exists an isomorphism
\begin{align}\label{undls} 
Sol_{\var{Y}}^{\rmE}(\tl{\SM^\wedge})
\simeq \CC_{\var{Y}^{\an}}^{\rmE}\Potimes{}^\sfL G.
\end{align}
Let
\begin{align}
\(X^{\an} \times\RR_s\) \overset{\ p_1}{\longleftarrow}     
\(X^{\an} \times\RR_s\)\times\(Y^{\an} \times\RR_t\)
\overset{p_2}{\longrightarrow} \(Y^{\an} \times\RR_t\)
\end{align}
be the projections and set $G^{\circ}:= 
G|_{X^{\an} \times\RR_s}$. 
Then by \cite[Lemma 7.2.1]{DK18} on
$Y^{\an}\subset\var{Y}^{\an}$ there exists an isomorphism
\begin{align}\label{qisoms} 
{}^\sfL G \simeq \bfQ\(\rmR p_{2!}(p_1^{-1} G^{\circ}
\otimes\CC_{\{t-s-\Re zw\geq0\}}[1])\),
\end{align}
where $\bfQ\colon\BDC(\CC_{Y^{\an}\times\RR_t})
\rightarrow\BEC(\CC_{Y^{\an}})$
is the quotient functor.
Moreover, for a point
$(w,t)\in V\times\RR\subset Y^{\an}\times\RR$
we have an isomorphism
\begin{align}\label{stalkfs} 
\(\rmR p_{2!}(p_1^{-1}G^{\circ} \otimes
\CC_{\{t-s-\Re zw\geq0\}}[1])\)_{(w,t)}
\simeq \rsect_c \(X^{\an};
\rmR\pi_!(G^{\circ} \otimes\CC_{\{t-s-\Re zw\geq0\}})[1]\).
\end{align}
In view of \eqref{undls}, \eqref{qisoms} and 
\eqref{stalkfs}, for the proof of Theorem \ref{Thm-T3} 
it suffices to calculate the right hand side of 
\eqref{stalkfs} for each point $(w,t)\in V\times\RR$. 

Let us fix a point $(w,t)\in V\times\RR$ and describe the structure of
\begin{align}
G(w,t) \coloneq
\rmR\pi_!(G^{\circ} \otimes\CC_{\{t-s-\Re zw\geq0\}} [1])\
\in\BDC(\CC_{X^{\an}}).
\end{align}
By our construction of $G$, we can easily see that 
the restriction of $G(w,t)$ to
$D^{\an}=\{a_1,\dots,a_l\}\subset X^{\an}$ is zero 
and the restriction of $\pi$ to the support of 
$G^{\circ} \otimes\CC_{\{t-s-\Re zw\geq0\}}$ is 
proper on $U^{\an}=X^{\an} \setminus D^{\an}$. 
Then we obtain an isomorphism 
\begin{align}
\rmR\pi_!(G^{\circ} \otimes\CC_{\{t-s-\Re zw\geq0\}}) 
\simto 
\rmR\pi_*(G^{\circ} \otimes\CC_{\{t-s-\Re zw\geq0\}})  
\end{align}
on $U^{\an} \subset X^{\an}$ 
and hence a morphism 
\begin{align}\label{surmorp} 
L \simeq H^0( \rmR\pi_* G^{\circ})|_{U^{\an}}  \longrightarrow 
H^{-1}G(w,t)|_{U^{\an}} 
\end{align}
of $\RR$-constructible sheaves on $U^{\an}$ 
is induced by the canonical morphism $G^{\circ} \longrightarrow 
G^{\circ} \otimes\CC_{\{t-s-\Re zw\geq0\}}$. 
We will see below that it is surjective. Namely 
we show that $H^{-1}G(w,t)|_{U^{\an}}$ 
is a quotient sheaf of the local system $L$. 
For a point $a_i\in D^{\an}$ by our construction of
the enhanced sheaf $F_i=F(a_i)\in\rcEC^0(\CC_{X^{\an}})$
there exist open sectors
$V_1,V_2,\dots,V_d$
along $a_i\in D^{\an}$ placed in
the counter-clockwise direction such that
\begin{align}
D(a_i)^\circ  \subset V_1\cup V_2\cup\dots\cup V_d 
\subset B(a_i)^\circ
\end{align}
and for $1\leq j< j^{\prime} \leq d$
we have $V_j\cap V_{j^{\prime}} \not= \emptyset$ if and only if
$j^{\prime} =j+1$ or $(j, j^{\prime})=(1,d)$
for which we have isomorphisms
\begin{align}
\Phi_j \colon \pi^{-1}\CC_{V_j}\otimes F_i
\simto \bigoplus_{f\in N_i^{>0}(V_j)}
\( \CC_{ \{ (z,s)  |  z \in V_j \cap D(a_i)^\circ, \ 
s+ \Re f(z) \geq 0 \} } \)^{N_i(f)}
\quad \(1\leq j\leq d\).
\end{align}
We also impose the condition (iv) in 
Remark \ref{rem-K5-2} on the covering 
$D(a_i)^\circ  \subset V_1\cup\dots\cup V_d$. 
Let $f_1,f_2, \ldots, f_{m_i} \in N_i^{>0}(V_j)$ be
the exponential factors of $\SM^{\an}$ on the
sector $V_j$. 
Then we can easily show that
the restriction of $G(w,t)$ to 
$V_j \cap D(a_i)^\circ  \subset  D(a_i)^\circ$
is isomorphic to
\begin{align}\label{locstru} 
\bigoplus_{k=1}^{m_i}
\Bigl(\CC_{\left\{z\in V_j \cap D(a_i)^\circ  \mid 
\Re(f_k^w)(z) \leq t\right\}} [1] \Bigr)^{N_i(f_k)}
\end{align}
(see \eqref{funcfw}). 
Moreover the restriction of $G(w,t)$ to
the punctured disk $D(a_i)^\circ$
is obtained by gluing these $\RR$-constructible sheaves
on $V_j \cap D(a_i)^\circ$ ($1 \leq j \leq d$) by the transition
matrices $A_{j j^{\prime}} \in {\rm GL}_{R}( \CC )$
($V_j \cap V_{j^{\prime}} \not= \emptyset$), where we set
 $R:= {\rm rk} \SM$. Recall that if 
$V_j \cap V_{j^{\prime}} \not= \emptyset$ after renumbering the 
exponential factors on $V_j \cap V_{j^{\prime}}$ 
the matrix $A_{j j^{\prime}}$ becomes block upper triangular. 
For $j \not= j^{\prime}$ such that $V_j \cap V_{j^{\prime}} 
\not= \emptyset$ and $V_j \cup V_{j^{\prime}}$ is 
simply connected, we can easily show that there exists an isomorphism 
\begin{align}\label{twolocdesc} 
G(w,t)|_{V_j \cup V_{j^{\prime}}} \simeq 
\bigoplus_{f \in N_i^{>0}(V_j \cup V_{j^{\prime}})} 
\Bigl(\CC_{\left\{z\in V_j \cup V_{j^{\prime}} \mid 
\Re f^w(z) \leq t\right\}} [1] \Bigr)^{N_i(f)}.
\end{align}
Let $S \subset D(a_i)^\circ$ be a simply connected 
sector in $D(a_i)^\circ$. Then, the restriction of 
$L$ to $S$ being isomorphic to the constant sheaf 
$\CC_S^R$, similarly we obtain an isomorphism 
\begin{align}\label{locdesc} 
G(w,t)|_S \simeq 
\bigoplus_{f \in N_i^{>0}(S)} 
\Bigl(\CC_{\left\{z\in S  \mid 
\Re f^w(z) \leq t\right\}} [1] \Bigr)^{N_i(f)}.
\end{align}
In the following example, for the reader's understanding of the proofs 
of Lemma \ref{lem-vanish} and Theorem \ref{Thm-T3}, we illustrate how the 
sublevel sets $\{\Re f^w(z)\leq t\}$ of $\Re f^w$ change 
as $t\in\RR$ increases.

\begin{example}\label{ex:levelset}
Set 
\begin{equation}
D(0) \coloneq\Set*{z\in X^{\an}}{ \abs*{z}\leq\varepsilon}, \quad        
D(0)^\circ\coloneq\Set*{z\in X^{\an}}{0<\abs*{z}\leq\varepsilon} \quad (0<\varepsilon\ll1)
\end{equation}
and let $\alpha>0$ be a positive real number and $f(z)=-\alpha/z$ 
the holomorphic function on $D(0)^\circ$ associated to it. 
We consider $f$ as an exponential factor at the origin $0 \in X^{\an}= \CC$. 
Then for sufficiently large $w\gg0$, 
the critical points of the holomorphic function 
\begin{equation}
f^w(z)=zw+\frac{\alpha}{z} \quad (z\in D(0)^\circ)
\end{equation}
are 
\begin{equation}
\gamma_1(w)\coloneq-\alpha^{\frac{1}{2}}w^{-\frac{1}{2}},
\quad
\gamma_2(w)\coloneq\alpha^{\frac{1}{2}}w^{-\frac{1}{2}}
\ \in D(0)^\circ
\end{equation}
and hence the critical values of the real-valued function $\Re f^w$ are
\begin{align}
c_1(w)&\coloneq\Re(f^w)(\gamma_1(w))= -2\alpha^{\frac{1}{2}}w^{\frac{1}{2}}, \\
c_2(w)&\coloneq\Re(f^w)(\gamma_2(w))=2\alpha^{\frac{1}{2}}w^{\frac{1}{2}}.
\end{align}
Let $A^+(w)\in\RR$ (resp. $A^-(w)$) be the maximal value (resp. minimal value) 
of the real-valued function $\Re(f^w\vbar_{\partial D(0)})$ on $\partial D(0) \simeq S^1$.
Then the sublevel sets of the real-valued function $\Re f^w$ at 
$t<A^-(w)$ and $t=A^-(w), c_1(w), c_2(w), A^+(w)$ 
are shown in gray in Figure \ref{fig:level}. 
\begin{figure}[h]
    \centering
    \begin{tikzpicture}[scale=0.4]
        \coordinate (O) at (0,0);
        \coordinate (Xm) at(-5,0);
        \coordinate (XM) at(5,0);
        \coordinate(Ym) at (0,-5);
        \coordinate(YM) at (0,5);
        \draw (O) circle[radius=3];
        \draw (2,2) node[above right]{$\partial D(0)$};
        \draw (0,-6) node{(i) $t<A^-(w)$};
        \begin{scope}
            \clip (5,5)--(5,-5)--(-5,-5)--(-5,5)--(5,5);
            \fill[lightgray] (O)..controls (-0.6,1) and (-0.95,0.6)..(-1,0);
            \fill[lightgray] (O)..controls (-0.6,-1) and (-0.95,-0.6)..(-1,0);
            \draw (O)..controls (-0.6,1) and (-0.95,0.6)..(-1,0);
            \draw (O)..controls (-0.6,-1) and (-0.95,-0.6)..(-1,0);
            \draw[dash pattern=on 3pt off 3pt] (-5,6)..controls (-4,3) and (-4,-3)..(-5,-6);
        \end{scope}
        \draw[semithick,->,>=stealth](Ym)--(YM);
        \draw[semithick,->,>=stealth](Xm)--(XM);
        \fill[white] (O) circle[radius=0.1] ;
        \draw (O) circle[radius=0.1] node[above right]{$0$};

        \begin{scope}[xshift=12cm]
            \draw (-6,-6) node{$\longrightarrow$};
            \coordinate (O) at (0,0);
            \coordinate (Xm) at(-5,0);
            \coordinate (XM) at(5,0);
            \coordinate(Ym) at (0,-5);
            \coordinate(YM) at (0,5);
            \draw (O) circle[radius=3];
            \draw (2,2) node[above right]{$\partial D(0)$};
            \draw (0,-6) node{(ii) $t=A^-(w)$};
            \begin{scope}
                \clip (5,5)--(5,-5)--(-5,-5)--(-5,5)--(5,5);
                \fill[lightgray] (O)..controls (-0.6,1.25) and (-1.4,1)..(-1.5,0);
                \fill[lightgray] (O)..controls (-0.6,-1.25) and (-1.4,-1)..(-1.5,0);
                \draw (O)..controls (-0.6,1.25) and (-1.4,1)..(-1.5,0);
                \draw (O)..controls (-0.6,-1.25) and (-1.4,-1)..(-1.5,0);
                \draw[dash pattern=on 3pt off 3pt] (-4,6)..controls (-2.68,-3) and (-2.68,3)..(-4,-6);
            \end{scope}
            \coordinate (T2) at (-3,-0);
            \fill[black] (T2) circle[radius=0.2]; 
            \draw[semithick,->,>=stealth](Ym)--(YM);
            \draw[semithick,->,>=stealth](Xm)--(XM);
            \fill[white] (O) circle[radius=0.1] ;
            \draw (O) circle[radius=0.1] node[above right]{$0$};
        \end{scope}

        \begin{scope}[xshift=24cm]
            \draw (-6,-6) node{$\longrightarrow$};
            \coordinate (O) at (0,0);
            \coordinate (Xm) at(-5,0);
            \coordinate (XM) at(5,0);
            \coordinate(Ym) at (0,-5);
            \coordinate(YM) at (0,5);
            \draw (2,2) node[above right]{$\partial D(0)$};
            \draw (0,-6) node{(iii) $t=c_1(w)$};
            \begin{scope}
                \clip (5,5)--(5,-5)--(-5,-5)--(-5,5)--(5,5);
                \fill[lightgray] (O)..controls (-0.6,1.5) and (-1.5,1.5)..(-2,0);
                \fill[lightgray] (O)..controls (-0.6,-1.5) and (-1.5,-1.5)..(-2,0);
                \draw (O)..controls (-0.6,1.5) and (-1.5,1.5)..(-2,0);
                \draw (O)..controls (-0.6,-1.5) and (-1.5,-1.5)..(-2,0);
                \draw[dash pattern=on 3pt off 3pt] (-3,6)..controls (-1.68,-5) and (-1.68,5)..(-3,-6);
                \clip (O) circle[radius=3];
                \fill[lightgray] (-3,6)..controls (-1.68,-5) and (-1.68,5)..(-3,-6);
                \draw (-3,6)..controls (-1.68,-5) and (-1.68,5)..(-3,-6);
            \end{scope}
            \draw (-3,-2) node[below left] (S) {$\gamma_1(w)$};
            \coordinate (T) at (-2,0);
            \draw (S)--(T);
            \fill[black] (T) circle[radius=0.1];
            \draw (O) circle[radius=3];
            \draw[semithick,->,>=stealth](Ym)--(YM);
            \draw[semithick,->,>=stealth](Xm)--(XM);
            \fill[white] (O) circle[radius=0.1] ;
            \draw (O) circle[radius=0.1] node[above right]{$0$};
        \end{scope}

        \begin{scope}[xshift=6cm,yshift=-13cm]
            \draw (-6,-6) node{$\longrightarrow$};
            \coordinate (O) at (0,0);
            \coordinate (Xm) at(-5,0);
            \coordinate (XM) at(5,0);
            \coordinate(Ym) at (0,-5);
            \coordinate(YM) at (0,5);
            \draw (0,-6) node{(iv) $t=c_2(w)$};
            \begin{scope}
                \clip (5,5)--(5,-5)--(-5,-5)--(-5,5)--(5,5);
                \fill[lightgray] (O) circle[radius=3];
                \fill[white] (O)..controls (0.6,1.5) and (1.5,1.5)..(2,0);
                \fill[white] (O)..controls (0.6,-1.5) and (1.5,-1.5)..(2,0);
                \draw (O)..controls (0.6,1.5) and (1.5,1.5)..(2,0);
                \draw (O)..controls (0.6,-1.5) and (1.5,-1.5)..(2,0);
                \draw[dash pattern=on 3pt off 3pt] (3,6)..controls (1.68,-5) and (1.68,5)..(3,-6);
                \clip (O) circle[radius=3];
                \fill[white] (3,6)..controls (1.68,-5) and (1.68,5)..(3,-6);
                \draw (3,6)..controls (1.68,-5) and (1.68,5)..(3,-6);
            \end{scope}
            \draw (3,-2) node[below right] (S) {$\gamma_2(w)$};
            \coordinate (T) at (2,0);
            \draw (S)--(T);
            \fill[black] (T) circle[radius=0.1];
            
            \draw (-2,2) node[above left]{$\partial D(0)$};
            \draw (O) circle[radius=3];
            \draw[semithick,->,>=stealth](Ym)--(YM);
            \draw[semithick,->,>=stealth](Xm)--(XM);
            \fill[white] (O) circle[radius=0.1] ;
            \draw (O) circle[radius=0.1] node[above left]{$0$};
        \end{scope}

        \begin{scope}[xshift=18cm,yshift=-13cm]
            \draw (-6,-6) node{$\longrightarrow$};
            \coordinate (O) at (0,0);
            \coordinate (Xm) at(-5,0);
            \coordinate (XM) at(5,0);
            \coordinate(Ym) at (0,-5);
            \coordinate(YM) at (0,5);
            \draw (0,-6) node{(v) $t=A^+(w)$};
            \begin{scope}
                \clip (5,5)--(5,-5)--(-5,-5)--(-5,5)--(5,5);
                \fill[lightgray] (O) circle[radius=3];
                \fill[white] (O)..controls (0.6,1.25) and (1.4,1)..(1.5,0);
                \fill[white] (O)..controls (0.6,-1.25) and (1.4,-1)..(1.5,0);
                \draw (O)..controls (0.6,1.25) and (1.4,1)..(1.5,0);
                \draw (O)..controls (0.6,-1.25) and (1.4,-1)..(1.5,0);
                \draw[dash pattern=on 3pt off 3pt] (4,6)..controls (2.68,-3) and (2.68,3)..(4,-6);
            \end{scope}
            \draw (-2,2) node[above left]{$\partial D(0)$};
            \draw (O) circle[radius=3];
            \draw[semithick,->,>=stealth](Ym)--(YM);
            \draw[semithick,->,>=stealth](Xm)--(XM);
            \fill[white] (O) circle[radius=0.1] ;
            \draw (O) circle[radius=0.1] node[above left]{$0$};
        \end{scope}
    \end{tikzpicture}
    \caption{The sublevel set $\{z\in D(0)^\circ\mid\Re f^w(z)\leq t\}$ in gray.}
    \label{fig:level}
\end{figure}
\end{example}

Also for the point $a_\infty=\infty\in\tl{D}$
we have a similar description of $G(w,t)$ on
a neighborhood of it and can define
holomorphic functions $f^w\ (f\in N_\infty^{>0})$ defined on
some open sectors along it. 
For $w\in V$ we define a real-valued function $\xi^w$ 
on the open subset 
\begin{align}
\Omega \coloneq X^{\an}\bs \bigl(D(a_i)\cup\dots\cup D(a_l)\cup 
D(a_\infty)\bigr)\ \subset U^{\an}
\end{align} 
of $U^{\an}$ by 
\begin{align}
\xi^w(z)\coloneq\Re(zw)+c \quad (z\in\Omega),
\end{align}
where $c\in\RR$ is the sufficiently large number used 
in the construction of the enhanced sheaf $G$ 
(see Section \ref{uni-sec:T1}).
Then it is easy to see that for any $w\in V$ and 
$t\in\RR$  the restriction of $G(w,t)$ to 
$\Omega\subset U^{\an}$ is isomorphic to 
\begin{align}
(L\vbar_\Omega)\otimes
\CC_{\{ z\in \Omega | \xi^w(z)\leq t \}} [1].
\end{align}
From these local descriptions of $G(w,t)$, we see 
that the morphism in \eqref{surmorp} is surjective.  

\begin{lemma}\label{lem-vanish}
For $t\ll0$ we have the vanishing
\begin{align}
\rsect_c\(X^{\an};G(w,t)\) \simeq 0.
\end{align}
\end{lemma}

\begin{proof}
By the above local descriptions of $G(w,t)$, 
it is clear that for $t\ll0$
the support of $G(w,t)$ is contained in
\begin{align}
D(a_1) \cup D(a_2)\cup\dots\cup D(a_l)\cup D(a_{\infty}). 
\end{align}
Hence it suffices to show that for $t\ll0$ we have
\begin{align}
\rsect_c\( D(a_i);G(w,t)\)\simeq 0 \quad (1\leq i\leq l)
\end{align}
and
\begin{align}
\rsect_c\( D(a_{\infty});G(w,t)\) \simeq 0
\end{align}
We only show the vanishing for $D(a_i)$ $(1\leq i\leq l)$.
The proof for $D(a_{\infty})$ is similar.
For a point $a_i \in D^{\an}$
we can slightly perturb
the boundaries of the sectors $V_j$ along it and
replace the covering
$D(a_i)^\circ \subset V_1\cup\dots\cup V_d$ 
by its refinement 
if necessary and may assume that for $t\ll0$ the set
\begin{align}
\Set*{z\in V_j}{\Re f^w(z)\leq t}\subset V_j
\end{align}
is empty or isomorphic to one of $(0,1)\times(0,1]$ 
and $(0,1]\times(0,1]$
for any $1\leq j\leq d$ and $f\in N_i^{>0}(V_j)$ 
(see Remark \ref{rem-K616} and Example \ref{ex:levelset}). 
Moreover we may assume that for $t\ll0$ the same is 
true also for the set 
\begin{align}
\Set*{z\in V_j \cap V_{j^{\prime}} }{\Re f^w(z)\leq t}\subset V_j \cap V_{j^{\prime}}
\end{align}
for any $1\leq j < j^{\prime} \leq d$ and 
$f\in N_i^{>0}(V_j \cap V_{j^{\prime}} )$. 
Then by a Mayer-Vietoris exact sequence associated to
the covering 
$D(a_i)^\circ  \subset V_1\cup\dots\cup V_d$
the assertion immediately follows from
the above descriptions of $G(w,t)$ on the 
open subsets $V_j \cap D(a_i)^\circ, V_j \cap V_{j^{\prime}} \cap D(a_i)^\circ 
 \subset D(a_i)^\circ$.
\end{proof}

To prove Theorem \ref{Thm-T3}, for $w \in V$ 
we consider the real-valued functions
$\Re f^w\colon V_j  \cap D(a_i)^\circ  \rightarrow\RR\
 (f\in N_i^{>0}(V_j))$
on the open subsets $V_j \cap D(a_i)^\circ  \subset D(a_i)^\circ$ as Morse functions
and apply a Morse theory associated to them.
In contrast to the usual Morse theory,
that we use here relies on
the sublevel sets of ``several" Morse functions.
Note that the critical points of
$\Re f^w\colon V_j\rightarrow\RR$ are
those of the holomorphic function
$f^w\colon V_j\rightarrow\CC$.
For $w \in V$ and $a_i\in D^{\an}$ let
$\gamma_{i j}(w)\in {\rm Int}D (a_i)^{\circ}$ $(1\leq j\leq n_i)$
be the points in the punctured disk
${\rm Int}D(a_i)^{\circ}$ such that
\begin{align}
(f^w)^\prime( \gamma_{i j}(w)) = 0
\quad \Longleftrightarrow \quad 
f^\prime ( \gamma_{i j}(w)) = w 
\end{align}
for some non-zero exponential factor
$f\in N_i^{>0}$ of $\SM^{\an}$ at $a_i$.
Note that 
for each point $\gamma_{i j}(w)$ 
the non-zero exponential factor $f \in N_i^{>0}$ satisfying 
the condition $f^\prime ( \gamma_{i j}(w)) = w$ 
is unique.
Then by using such $f\in N_i^{>0}$ we set
\begin{align}
c_{i j}(w) \coloneq \Re(f^w)(\gamma_{i j}(w))
\ \in\RR \quad (1\leq j\leq n_i)
\end{align}
and
\begin{align}
N_{i j} \coloneq N_i(f) \ >0 \quad (1\leq j\leq n_i).
\end{align}
Namely $c_{i j}(w)$ $(1\leq j\leq n_i)$ are
the critical values of the (possibly multi-valued) functions
$\Re(f^w)$ $(f\in N_i^{>0}, f\neq0)$.  Also for $w \in V$ and $a_\infty=\infty\in\tl{D}$
we can define points
$\gamma_{\infty j}(w)\in {\rm Int}D(a_{\infty})^{\circ}
\subset X^{\an}=\CC_z$ $(1\leq j\leq n_\infty)$ and
$c_{\infty j}(w)\in\RR$, $N_{\infty j}$ 
$(1\leq j\leq n_\infty)$ similarly.

\begin{lemma}\label{lem-Morseind}
For any point $w\in V$
all the critical points of the functions
$\Re f^w\colon V_j\rightarrow\RR$ $(f\in N_i^{>0}(V_j))$ 
are (Morse) non-degenerate and have the Morse index 1.
\end{lemma}

\begin{proof}
Note that for any $f \in N_i^{>0}$ 
the morphism 
$f^\prime \colon 
B(a_i)^{\circ} \cap (f^\prime )^{-1}(V) 
\longrightarrow V ( \subset Y^{\an}= \CC ) 
$ is locally biholomorphic. 
This implies that for any $w\in V$ and any 
critical point $\gamma \in V_j$ of $f^w$ 
($f\in N_i^{>0}(V_j)$) we have $f^{\prime}( \gamma )=w$ and 
\begin{equation}
(f^w)^{\prime\prime}( \gamma )= - 
f^{\prime\prime}( \gamma ) \not= 0. 
\quad (z\in V_j),
\end{equation}
Moreover, by the Cauchy-Riemann equation we have 
\begin{equation}
\det H(\Re f^w)(z) = -\abs{(f^w)^{\prime\prime}(z)}^2
\quad (z\in V_j),
\end{equation}
where the left hand side is the determinant of 
the Hesse matrix of $\Re f^w$ at $z\in V_j$.
Hence, for any $w\in V$ and any 
critical point $\gamma \in V_j$ of $f^w$ 
 ($f\in N_i^{>0}(V_j)$) we obtain 
\begin{equation}
\det H(\Re f^w)(\gamma) <0
\end{equation} 
and conclude that the Morse index at $\gamma$ is $1$ (see also 
 (\ref{eq:Morseeq}) below).   
\end{proof}

For $w\in V$ and $a_i\in D^{\an}$  
we define a subset $Q(a_i)_+^w$ (resp. $Q(a_i)_-^w$) of 
$\partial D(a_i)\simeq S^1$ to be the set of 
the points $z\in\partial D(a_i)$ such that for 
some $f\in N_i^{>0}$ the (possibly multi-valued) real-valued function 
$\Re(f^w\vbar_{\partial D(a_i)})$ on 
$\partial D(a_i)\simeq S^1$ has a 
local maximum (resp. local minimum) at $z\in\partial D(a_i)$. 
Note that by our definition $f^w(z)=zw-f(z)$ in 
\eqref{funcfw} 
if $w \rightarrow \beta \infty$ in the sector $V$ 
for some non-zero $\beta \in \CC$, then the points in 
$Q(a_i)_+^w$ (resp. $Q(a_i)_-^w$) converge to 
a point in $\partial D(a_i)\simeq S^1$. 
Hence by shrinking the punctured disk 
$B(b_{\infty})^{\circ}$ centered at 
$b_{\infty} = \infty\in\var{Y}^{\an}$ if necessary, 
we may assume that there 
exists a simply connected 
sector $V_+$ 
(resp. $V_-$) along the point $a_i$ such that 
$Q(a_i)_+^w\subset V_+$ (resp. $Q(a_i)_-^w\subset V_-$) 
for any $w\in V$. 
We may assume also that for any $w\in V$ and $f\in N_i^{>0}$ 
the (possibly multi-valued) real-valued function 
$\Re(f^w\vbar_{\partial D(a_i)})$ on $\partial D(a_i)$ has no critical point on 
$\partial D(a_i)\bs(Q(a_i)_+^w\sqcup Q(a_i)_-^w)$. 
Specifically, by the exponential factors  
$f_1,\dots,f_{m_i}\in N_i^{>0}(V_+)$ of $\SM^{\an}$ on the 
sector $V_+$ we define real-valued functions 
$\phi_{+,1}^w,\dots,\phi_{+,m_i}^w$ on 
$\partial D(a_i)\cap V_+$ by
\begin{align}
\phi_{+,k}^w(z)\coloneq\Re f_k^w(z)=\Re(zw)-\Re f_k(z) 
\quad (z\in\partial D(a_i)\cap V_+, 1 \leq k \leq m_i)
\end{align}
and set 
\begin{align}
N_i^+(k)\coloneq N_i(f_k) \, \in\ZZ_{>0} \quad 
(1 \leq k \leq m_i).
\end{align}
Similarly we define real-valued functions 
$\phi_{-,1}^w,\dots,\phi_{-,m_i}^w$ on 
$\partial D(a_i)\cap V_-$ and 
$N_i^-(k)\in\ZZ_{>0}$ $(1\leq k \leq m_i)$.
For $1\leq k\leq m_i$ and $w\in V$ let 
$a_{ik}^+(w)\in\partial D(a_i)\cap V_+$ 
(resp. $a_{ik}^-(w)\in\partial D(a_i)\cap V_-$) be the (unique) 
point where the function $\phi_{+,k}^w$ (resp. $\phi_{-,k}^w$) 
takes its maximum (resp. minimum) and set 
\begin{align}
A_{ik}^+(w)\coloneq\phi_{+,k}^w(a_{ik}^+(w)) \quad 
\textrm{(resp. $A_{ik}^-(w)\coloneq\phi_{-,k}^w(a_{ik}^-(w)$)}.
\end{align}
Also for the point $a_\infty=\infty\in\tl{D}$ we define real-valued 
functions $A_{\infty k}^+(w)$, $A_{\infty k}^-(w)$ 
($1 \leq k \leq m_{\infty}$) 
on the sector $V\subset Y^{\an}$ along the point 
$b_\infty=\infty\in\var{Y}^{\an}$ and 
$N_\infty^+(k),N_\infty^-(k)\in\ZZ_{>0}$ ($1 \leq k \leq m_{\infty}$) similarly.

In what follows, we use also $\xi^w\colon \Omega =
 X^{\an}\bs \bigl(D(a_i)\cup \dots \cup D(a_l)\cup 
D(a_\infty)\bigr) \rightarrow \RR$ as a 
Morse function to calculate ${}^{\sfL}G$.
Let $\tl{\xi^w}\colon\var{\Omega} \rightarrow \RR$ be the 
(unique) continuous extension of $\xi^w\colon\Omega\to\RR$ to 
$\var{\Omega}$ and for $a_i\in\tl{D}=D^{\an}\sqcup\{\infty\}$ set 
\begin{align}
L_i^+(w)&\coloneq\max_{z\in\partial D(a_i)}\tl{\xi^w}(z), \\
L_i^-(w)&\coloneq\min_{z\in\partial D(a_i)}\tl{\xi^w}(z).
\end{align}
Then by our choice of $c\in\RR$ in the construction 
of $G$, we have the following result.

\begin{lemma}\label{lem-bounded}
For any point $a_i\in\tl{D}=D^{\an}\sqcup\{\infty\}$ 
and $1 \leq k \leq m_i$ we have 
\begin{align}
A_{ik}^-(w)<L_i^-(w)<A_{ik}^+(w)<L_i^+(w).
\end{align}
Moreover the functions $L_i^+(w)-A_{ik}^+(w)$, 
$L_i^-(w)-A_{ik}^-(w)$ on $V$ are bounded.
\end{lemma} 

By applying the Morse theoretical method in the proof of 
\cite[Theorem 4.4]{IT20a} to the Morse functions 
${\rm Re} f^w$ ($f \in N_i^{>0}$) 
and $\xi^w$ we obtain the following result.

\begin{proposition}\label{prop-infty-enh}
We have an isomorphism
\begin{equation}
\begin{split}
\pi^{-1}\CC_V \otimes {}^{\sfL}G \quad 
\simeq \quad &\bigoplus_{i=1}^l\biggl\{
\bigoplus_{j=1}^{n_i}\bigl(\sfE_{V\vbar \var{Y}^{\an}}^{-c_{ij}}\bigr)
^{N_{ij}}\oplus\bigl(\sfE_{V\vbar \var{Y}^{\an}}^{-c_i}\bigr)^{r_i}
\oplus
\bigoplus_{k=1}^{m_i} \Bigl(\sfE_{V\vbar \var{Y}^{\an}}
^{-A_{ik}^-\vartriangleright-L_i^-}[1]\Bigr)^{N_i^-(k)}\biggr\} \\ 
&\oplus \biggl\{\bigoplus_{j=1}^{n_\infty}
\bigl(\sfE_{V\vbar \var{Y}^{\an}}^{-c_{\infty j}}\bigr)^{N_{\infty j}} 
\oplus
\bigoplus_{k=1}^{m_{\infty}}
\Bigl(\sfE_{V\vbar \var{Y}^{\an}}
^{-A_{\infty k}^+\vartriangleright-L_\infty^+}\Bigr)^{N_\infty^+(k)}
\biggr\},
\end{split}
\end{equation}
where for $1\leq i\leq l$ we set 
$c_i(w)\coloneq\Re(a_iw)\quad (w\in V)$.
\end{proposition}

\begin{proof}
We fix $w\in V$ and calculate 
$\rsect_c(X^{\an};G(w,t))$ for all $t\in\RR$.
First, by Lemma \ref{lem-vanish} for $t\ll0$ we have 
\begin{align}
\rsect_c(X^{\an};G(w,t))\simeq0
\end{align}
By our local descriptions of $G(w,t)$ the cohomology 
groups of $\rsect_c(X^{\an};G(w,t))$ may jump only at 
$t\in\RR$ in the finite set 

\begin{equation}
\begin{split}
\{c_{ij}(w)\mid a_i\in\tl{D},1\leq j\leq n_i\}\cup
\{c_i(w)\mid a_i\in D^{\an}\} \\
\cup\{A_{ik}^\pm(w)\mid a_i\in\tl{D}, 1 \leq k \leq m_i 
\}\cup\{L_i^\pm(w)\mid a_i\in\tl{D}\}.
\end{split}
\end{equation}
On the other hand, by Lemma \ref{lem-T2-1}, 
as $\abs{w}\to+\infty$ in the sector 
$V\subset Y^{\an}$ we have 
\begin{align}
\abs{\gamma_{ij}(w)-a_i}\longrightarrow0 
\quad (1\leq i\leq l,1\leq j\leq n_i)
\end{align}
and 
\begin{align}
\abs{\gamma_{\infty j}(w)}\longrightarrow+
\infty\quad(1\leq j\leq n_\infty).
\end{align}
This implies that for the sufficiently small sector 
$V\subset Y^{\an}$ along the point 
$b_\infty=\infty\in\var{Y}^{\an}$ we can deal with 
the cohomology jumps coming from the critical points 
$\gamma_{ij}(w)$ and those from the points 
$a_{ik}^\pm(w)$ etc. in $\partial D(a_i)$ separately. 
Then as in the proof of 
\cite[Theorem 4.4]{IT20a}, by Lemma \ref{lem-Morseind} we can 
prove the assertion as follows.
For $t^+,t^-\in\RR$ such that $t^-<t^+$ we define 
an $\RR$-constructible sheaf $G(w,t^+,t^-)$ on $X^{\an}$ 
by the exact sequence
\begin{equation}
0 \longrightarrow G(w,t^+,t^-) \longrightarrow 
G(w,t^+) \longrightarrow G(w,t^-)  \longrightarrow 0.
\end{equation}
Then for any $t\in\RR$ and $\varepsilon>0$ there 
exists a distinguished triangle 
\begin{equation}
\rsect_c(X^{\an};G(w,t,t-\varepsilon)) \longrightarrow  
\rsect_c(X^{\an};G(w,t)) \longrightarrow
\rsect_c(X^{\an};G(w,t-\varepsilon)) \overset{+1}{\longrightarrow}.
\end{equation}
Hence it suffices to calculate the cohomology groups 
$H^p\rsect_c(X^{\an};G(w,t,t-\varepsilon))$ $(p\in\ZZ)$ for each 
$t\in\RR$ and $0<\varepsilon\ll1$. 

First let us consider the case where $t\in\RR$ is 
not contained in the finite set 
\begin{align}
\{c_i(w)\mid a_i\in D^{\an}\}\cup
\{A_{ik}^\pm(w)\mid a_i\in\tl{D}, 1 \leq k \leq m_i \}\cup
\{L_i^\pm(w)\mid a_i\in\tl{D}\}.
\end{align}
For $a_i\in\tl{D}=D^{\an}\sqcup\{\infty\}$ and 
$1\leq j\leq n_i$ let $D_{ij}$ be a sufficiently small closed disk 
centered at the point $\gamma_{ij}(w)\in\Int D(a_i)^\circ$. 
Recall that by the (unique) non-zero exponential 
factor $f\in N_i^{>0}$ such that 
$(f^w)^\prime(\gamma_{ij}(w))=0$ we set 
$N_{ij}= N_i(f)$ and let $K_{ij} \simeq 
\CC_{D_{ij}}^{N_{ij}}$ be a constant subsheaf of 
the local system $L|_{D_{ij}}$ on $D_{ij}$ 
associated to $f$ (see \eqref{locstru}). 
For $s\in\RR$ we set 
\begin{align}M_{ij,s}\coloneq
\Set*{z\in D_{ij}}{\Re f^w(z)\leq s} \quad \subset D_{ij}.
\end{align}
Then the restriction of $G(w,s)$ to the disk $D_{ij}$ 
has a direct summand isomorphic to 
$(K_{ij})_{M_{ij,s}} [1] \simeq \CC_{M_{ij,s}}^{N_{ij}} [1]$ 
(see \eqref{locstru}). 
Since $\gamma_{ij}(w)$ is a non-degenerate 
critical point of $f^w$ by the proof of 
Lemma \ref{lem-Morseind} i.e.  
\begin{align}
(f^w)^{\prime}( \gamma_{ij}(w) )=0, \qquad 
(f^w)^{\prime \prime}( \gamma_{ij}(w) ) \not= 0, 
\end{align}
there exists a holomorphic coordinate 
$\zeta=x+\sqrt{-1}y$ $(x,y\in\RR)$ on $D_{ij}$ such 
that $\gamma_{ij}(w)=\{\zeta=0\}$ and 
\begin{align}
f^w(\zeta)=f^w(\gamma_{ij}(w))+\zeta^2.
\end{align}
This implies that we have 
\begin{align}\label{eq:Morseeq}
\Re f^w(\zeta)=c_{ij}(w)+x^2-y^2
\end{align}
on $D_{ij}$. 
We thus now clearly see that the function 
$\Re f^w$ has a Morse (non-degenerate) critical 
point of Morse index $1$ at the point 
$\gamma_{ij}(w)=\{\zeta=0\} = 
\{ x=y=0 \}$ (see also Lemma \ref{lem-Morseind}). 
Hence for $0<\varepsilon\ll1$ there exist isomorphisms
\begin{align}
H^p\rsect_c(X^{\an}; (K_{ij})_{M_{ij,t}\bs M_{ij,t-\varepsilon}})\simeq
\begin{cases*}
\ \CC^{N_{ij}}  & $(p=1$ and $c_{ij}(w)=t)$, \\
\\
\ 0 & (otherwise).
\end{cases*}
\end{align}
Then by our assumption on $t\in\RR$ for $0< \varepsilon \ll 1$ 
we obtain isomorphisms
\begin{align}
\rsect_c(X^{\an};G(w,t,t-\varepsilon)) 
&\simeq \bigoplus_{(i,j)\colon c_{ij}(w)=t}
\rsect_c(X^{\an}; (K_{ij} )_{M_{ij,t}\bs M_{ij,t-\varepsilon}}[1] ) \\
&\simeq \bigoplus_{(i,j)\colon c_{ij}(w)=t} \CC^{N_{ij}}.
\end{align}
This explains the reason why we have the factors 
$(\sfE_{V\vbar \var{Y}^{\an}}^{-c_{ij}})^{N_{ij}}$ 
$(a_i \in\tl{D}, 1\leq j\leq n_i)$ in the assertion.

Next we consider $t\in\RR$ such that there exist 
$a_i\in D^{\an}$ and $1 \leq k \leq m_i$ such that $A_{ik}^+(w)=t$.
For such $1 \leq i \leq l$ and $1 \leq k \leq m_i$ let 
$D_{ik}^+\subset V_+$ be a sufficiently small closed 
disk centered at $a_{ik}^+(w)\in\partial D(a_i)\cap V_+$.
Recall that by the exponential factor 
$f_k\in N_i^{>0}(V_+)$ on $V_+$ such that the function 
$\phi_{+,k}^w=\Re(f_k^w\vbar_{\partial D(a_i)\cap V_{+}})$ 
on $\partial D(a_i)\cap V_{+}$ takes its maximum 
$t=A_{ik}^+(w)$ at $a_{ik}^+(w)$ we set $N_i^+(k)=N_i(f_k)$ and 
let $K_{ik}^+ \simeq 
\CC_{D_{ik}^+}^{N_i^+(k)}$ be a constant subsheaf of  
the local system $L|_{D_{ik}^+}$ on $D_{ik}^+$ 
associated to $f_k$ (see \eqref{locstru}). For $s\in\RR$ we set 
\begin{align}
M_{ik,s}^+\coloneq
\Set*{z\in D(a_i) \cap D_{ik}^+}{\Re f_k^w(z)\leq s} 
\cup \{ z \in \Omega \cap D_{ik}^+ \ | \ \xi^w(z)\leq s \} 
\quad \subset  D_{ik}^+.
\end{align}
Then by the condition $t= A_{ik}^+(w)=\phi_{+,k}^w(a_{ik}^+(w))
<L_i^+(w)$, for 
$0<\varepsilon\ll1$ we obtain a vanishing 
\begin{align}
\rsect_c(X^{\an}; (K_{ik}^+)_{M_{ik,t}^+\bs M_{ik,t-\varepsilon}^+})\simeq 0
\end{align}
(see Figure \ref{fig:Aplus}). 
\begin{figure}
    \centering
\begin{tikzpicture}[scale=0.4]
    \draw (0,0) node{$\longrightarrow$};
    \begin{scope}[xshift=-9.5cm]
        \draw (5,3) node[right]{$\partial D(a_i)$};
        \draw (0,3.5) node{$D_{ik}^+$};
        \draw (5,-5) circle[radius=0.075] node[above]{$a_i$};
        \clip (5,5)--(5,-5)--(-5,-5)--(-5,5)--(5,5);
        \begin{scope}[rotate=-45]
            \coordinate (O) at(0,0);
            \coordinate (R) at(9,0);
            \coordinate (A) at(0,9);
            \coordinate (L) at(-9,0);
            \coordinate (B) at(0,-9);
            \begin{scope}
                    \clip (R) circle[radius=9];
                    \fill[lightgray] (O) circle [radius=4.5];
                    \fill[white] (0.5,-6) rectangle (-1,6);
                    \draw (0.5,-6)--(0.5,6);   
            \end{scope}
                \draw (R) circle[radius=9];
                \draw (O) circle[radius=4.5];
                \draw (O) node{\scriptsize $+$}; 
                \draw (O) node[above left]{$a_{ik}^+(w)$}; 
                \draw (2.75,0) node{$M_{ik,t-\varepsilon}^+$};
        \end{scope}
    \end{scope}
    
    \begin{scope}[xshift=7cm]
        \draw (5,3) node[right]{$\partial D(a_i)$};
        \draw (0,3.5) node{$D_{ik}^+$};
        \draw (5,-5) circle[radius=0.075] node[above]{$a_i$};
        \clip (5,5)--(5,-5)--(-5,-5)--(-5,5)--(5,5);
        \begin{scope}[rotate=-45]
            \coordinate (O) at(0,0);
            \coordinate (R) at(9,0);
            \coordinate (A) at(0,9);
            \coordinate (L) at(-9,0);
            \coordinate (B) at(0,-9);
            \begin{scope}
                    \clip (R) circle[radius=9];
                    \fill [lightgray] (O) circle [radius=4.5];   
            \end{scope}
            \draw (R) circle[radius=9];
            \draw (O) circle [radius=4.5];
            \draw (O) node{\scriptsize $+$}; 
            \draw (O) node[above left]{$a_{ik}^+(w)$};     
            \draw (2.5,0) node{$M_{ik,t}^+$};
        \end{scope}    
    \end{scope}
\end{tikzpicture}
\caption{$t=A_{ik}^+(w)$}
\label{fig:Aplus}
\end{figure}
This explains the reason why the functions 
$A_{ik}^+$ $(1\leq i\leq l, 1 \leq k \leq m_i)$ do not appear in the assertion. 
Similarly, we can neglect the functions 
$L_i^+$ $(1\leq i\leq l)$.

Now let us consider $t\in\RR$ such that there 
exist $a_i\in D^{\an}$ and $1 \leq k \leq m_i$ 
such that $A_{ik}^-(w)=t$.
For such $1\leq i\leq l$ and 
$1 \leq k \leq m_i$ let $D_{ik}^-\subset V_{-}$ be a sufficiently 
small closed disk centered at 
$a_{ik}^-(w)\in\partial D(a_i)\cap V_{-}$. 
Recall that by the exponential factor 
$f_k\in N_i^{>0}(V_-)$ on $V_{-}$ such that 
the function $\phi_{-,k}^w=\Re (f_k^w\vbar_{\partial D(a_i)\cap V_-})$ on
$\partial D(a_i)\cap V_-$ takes its minimum $t=A_{ik}^-(w)$ at $a_{ik}^-(w)$ 
we set $N_i^-(k)=N_i(f_k)$ and 
let $K_{ik}^- \simeq 
\CC_{D_{ik}^-}^{N_i^-(k)}$ be a constant subsheaf of 
the local system $L|_{D_{ik}^-}$ on $D_{ik}^-$ 
associated to $f_k$ (see \eqref{locstru}).
\begin{figure}
    \centering    
    \begin{tikzpicture}[scale=0.4]
        \draw (0,0) node{$\longrightarrow$};
        \begin{scope}[xshift=-7cm]
            \draw (-5,-3) node[left]{$\partial D(a_i)$};
            \draw (0,3.5) node{$D_{ik}^-$};
            \draw (-5,5) circle[radius=0.075] node[below]{$a_i$};
            \draw (-0.5,-3) node(Z) {$M_{ik,t-\varepsilon}^-$};
            \clip (5,5)--(5,-5)--(-5,-5)--(-5,5)--(5,5);
            \begin{scope}[rotate=135]
                \coordinate (O) at(0,0);
                \coordinate (R) at(9,0);
                \coordinate (A) at(0,9);
                \coordinate (L) at(-9,0);
                \coordinate (B) at(0,-9);
                \begin{scope}
                    \clip (O) circle[radius=4.5];
                    \fill[lightgray] (-4,-9)rectangle(-6,9);
                    \draw (-4,-9)--(-4,9);  
                \end{scope}
                    \draw (R) circle[radius=9];
                    \draw (O) circle[radius=4.5];
                    \draw (O) node{\scriptsize $+$}; 
                    \draw (O) node[above left]{$a_{ik}^-(w)$}; 
            \end{scope}
            \coordinate (S)at(2.75,-3.25);
            \draw (Z)--(S);
        \end{scope}
        
        \begin{scope}[xshift=9.5cm]
            \draw (-5,-3) node[left]{$\partial D(a_i)$};
            \draw (0,3.5) node{$D_{ik}^-$};
            \draw (-5,5) circle[radius=0.075] node[below]{$a_i$};
            \draw (-1.5,-3) node(Z) {$M_{ik,t}^-$};
            \clip (5,5)--(5,-5)--(-5,-5)--(-5,5)--(5,5);
            \begin{scope}[rotate=135]
                \coordinate (O) at(0,0);
                \coordinate (R) at(9,0);
                \coordinate (A) at(0,9);
                \coordinate (L) at(-9,0);
                \coordinate (B) at(0,-9);
                \begin{scope}
                        \clip (O) circle[radius=4.5];
                        \fill[lightgray] (-3,-9)rectangle(-6,9);
                        \draw (-3,-9)--(-3,9);  
                \end{scope}
                \draw (R) circle[radius=9];
                \draw (O) circle[radius=4.5];
                \fill[black] (O) circle[radius=0.14];
                \draw (O) node{\scriptsize $+$};
                \draw (O) node[above left]{$a_{ik}^-(w)$} node[above] (P) {}; 
            \end{scope}
            \draw (Z)--(0,0);
            \coordinate (S)at(2.5,-2.5);
            \draw (Z)--(S);
        \end{scope}
    \end{tikzpicture}
    \caption{$t=A_{ik}^-(w)$}
    \label{fig:Aminus}
\end{figure}
For $s\in\RR$ we set 
\begin{align}
M_{ik,s}^-\coloneq\Set*{z\in D(a_i) \cap D_{ik}^-}{\Re f^w(z)\leq s}
\cup \{ z \in \Omega \cap D_{ik}^- \ | \ \xi^w(z)\leq s \} 
\quad \subset D_{ik}^-.
\end{align}
Then by the condition $t=A_{ik}^-(w)=\phi_{-,k}^w(a_{ik}^-(w))<L_i^-(w)$, 
for $0<\varepsilon\ll1$ we have $M_{ik,t}^-=\{a_{ik}^-(w)\}
\sqcup \{ z \in \Omega \cap D_{ik}^+ \ | \ \xi^w(z)\leq t \}$, 
$M_{ik,t-\varepsilon}^- = 
\{ z \in \Omega \cap D_{ik}^+ \ | \ \xi^w(z)\leq t -\varepsilon \}$ 
(see Figure \ref{fig:Aminus}) and obtain an isomorphism
\begin{align}
\rsect_c(X^{\an};  (K_{ik}^-)_{M_{ik,t}^-\bs 
M_{ik,t-\varepsilon}^-})\simeq\CC.
\end{align}
This means that the rank of 
$H^{-1}\rsect_c(X^{\an};G(w,s))$ increases by one 
over the disk $D_{ik}^-$ 
at $s=t$. Set $t^\prime\coloneq L_i^-(w)>t=A_{ik}^-(w)$. 
Then for any $t\leq s<t^\prime$ there is no 
jump in the cohomology groups of $\rsect_c(X^{\an};G(w,s))$ 
from those of $\rsect_c(X^{\an};G(w,t))$ caused by 
the changes of the level sets of the Morse functions 
$\Re f^w$ and $\xi^w$ on a neighborhood of the point 
$a_{ik}^-(w)\in\partial D(a_i)$.
Indeed, for $t\leq s<t^\prime$ we have 
\begin{equation}
\{z\in D(a_i) \cap D_{ik}^- \ | \ \Re f^w(z)\leq s \} \cap 
\{ z \in \Omega \cap D_{ik}^- \ | \ \xi^w(z)\leq s \} 
=\emptyset. 
\end{equation}
But for $s=t^\prime$ this intersection is 
a one point set (see Figure \ref{fig:Lminus})
\begin{figure}
    \centering
    \begin{tikzpicture}[scale=0.4]
        \draw (0,0) node{$\longrightarrow$};
        \begin{scope}[xshift=-7cm]
            \draw (-5,-3) node[left]{$\partial D(a_i)$};
            \draw (-5,4) circle[radius=0.075] node[below]{$a_i$};
            \draw (-4,-6) node(Z) {$\{ z \in \Omega \cap D_{ik}^- \ | \ \xi^w(z)\leq t^\prime-\varepsilon \} $};
            \draw (-3,6.5) node(W) {$\{z\in D(a_i) \cap D_{ik}^- \ | \ \Re f^w(z)\leq t^\prime-\varepsilon \}$};
            \coordinate (S)at(-1,3);
            \coordinate (T)at(0,-4);
            \draw (W)--(S);
            \clip (5,5)--(5,-5)--(-5,-5)--(-5,5)--(5,5);
            \begin{scope}[rotate=135]
                \coordinate (O) at(0,0);
                \coordinate (R) at(9,0);
                \coordinate (A) at(0,9);
                \coordinate (L) at(-9,0);
                \coordinate (B) at(0,-9);
                \begin{scope}
                    \clip (O) circle[radius=4.5];
                    \fill[lightgray] (R) circle[radius=9]; 
                    \fill[lightgray] (-2.25,-9)rectangle(-6,9);
                    \draw (-2.25,-9)--(-2.25,9);  
                \end{scope}
                    \draw (R) circle[radius=9];
                    \draw (O) circle[radius=4.5];
                    \draw (O) node{\scriptsize $+$}; 
            \end{scope}
            \draw (-3,0) node{$D_{ik}^-$};
            \draw (W)--(S);
            \draw (Z)--(T);
        \end{scope}
        
        \begin{scope}[xshift=9.5cm]
            \draw (-5,-3) node[left]{$\partial D(a_i)$};
            \draw (-5,4) circle[radius=0.075] node[below]{$a_i$};
            \draw (4,-6) node(Z) {$\{ z \in \Omega \cap D_{ik}^- \ | \ \xi^w(z)\leq t^\prime \} $};
            \draw (1,6.5) node(W) {$\{z\in D(a_i) \cap D_{ik}^- \ | \ \Re f^w(z)\leq t^\prime \}$};
            \coordinate (S)at(0,3);
            \coordinate (T)at(2,-3);
            \draw (W)--(S);
            \clip (5,5)--(5,-5)--(-5,-5)--(-5,5)--(5,5);
            \begin{scope}[rotate=135]
                \coordinate (O) at(0,0);
                \coordinate (R) at(9,0);
                \coordinate (A) at(0,9);
                \coordinate (L) at(-9,0);
                \coordinate (B) at(0,-9);
                \begin{scope}
                    \clip (O) circle[radius=4.5];
                    \fill[lightgray] (R) circle[radius=9]; 
                    \fill[lightgray] (0,-9)rectangle(-6,9);
                    \draw (0,-9)--(0,9);   
                \end{scope}
                    \draw (R) circle[radius=9];
                    \draw (O) circle[radius=4.5];
                    \fill[black] (O) circle[radius=0.1];
                    \draw (O) node{\scriptsize $+$}; 
            \end{scope}
            \draw (-3,0) node{$D_{ik}^-$};
            \draw (W)--(S);
            \draw (Z)--(T);
        \end{scope}
    \end{tikzpicture}
    \caption{$t^\prime=L_i^-(w)$}
    \label{fig:Lminus}
\end{figure}
and the rank of 
$H^{-1}\rsect_c(X^{\an}\\;G(w,s))$ decreases 
by one over the disk $D_{ik}^-$ 
at $s=t^\prime$. 
Namely at $s=t^\prime$ we get back to 
the situation at $s=t-\epsilon$ $(0<\epsilon\ll1)$. 
This explains the reason why we have the factors 
$(\sfE_{V\vbar \var{Y}^{\an}}^{-A_{ik}^-\vartriangleright-L_i^-}[1])
^{N_i^-(k)}$ $(1\leq i\leq l, 1 \leq k \leq m_i)$ in the assertion.
We can treat the other functions  
$A_{\infty k}^\pm$ ($1 \leq k \leq m_{\infty}$) 
and $L_\infty^\pm$ similarly. 

Finally, we consider 
$t\in\RR$ such that there exists 
$a_i\in D^{\an}$ satisfying the conditions 
$r_i>0$ and  $t=c_i(w)= \Re (a_iw)$. 
For $a_i\in D^{\an}$ such that $r_i>0$ and 
$t=c_i(w)$, we take a sufficiently small closed disk 
$D_i^0 \subset {\rm Int}D(a_i)$ centered at 
$a_i\in D^{\an}$ and for $s\in\RR$ we set 
\begin{equation}
M_{i,s} \coloneq\Set*{z\in D_i^0\bs\{a_i\}}{\Re(zw)\leq s}, 
\qquad 
M_{i,s}^{\circ} \coloneq\Set*{z\in D_i^0\bs\{a_i\}}{\Re(zw) < s}.
\end{equation}
Let $K_i \simeq \CC_{M_{i,t}}^{r_i}$ be a constant subsheaf 
of the local system $L|_{M_{i,t}}$ on $M_{i,t}$ associated 
to the exponential factor $0 \in N_i^{>0}$ 
(see \eqref{locdesc}). Then for $0< \varepsilon \ll 1$ 
the restriction of $G(w,t, t- \varepsilon )$ to 
$M_{i,t}$ has a direct summand isomorphic to 
$(K_i)_{M_{i,t} \setminus M_{i, t- \varepsilon }} [1] 
\simeq 
\CC^{r_i}_{M_{i,t} \setminus M_{i, t- \varepsilon }} [1]$ 
(see \eqref{locdesc}). 
As in the proof of Lemma \ref{lem-vanish} 
by a Mayer-Vietoris exact sequence, we can 
easily show 
\begin{equation}
\rsect_c( M_{i,t} ; (K_i)_{M_{i,t} \setminus M_{i, t- \varepsilon }^{\circ}} ) 
\simeq 0. 
\end{equation}
For the set $L_{i, t- \varepsilon }:= 
M_{i,t- \varepsilon } \setminus M_{i, t- \varepsilon }^{\circ}$ 
isomorphic to the closed interval $[0,1] \subset \RR$, 
let us consider the exact sequence 
\begin{equation}
0 \longrightarrow 
(K_i)_{M_{i,t} \setminus M_{i, t- \varepsilon }} 
\longrightarrow 
(K_i)_{M_{i,t} \setminus M_{i, t- \varepsilon }^{\circ}} 
\longrightarrow 
(K_i)_{L_{i, t- \varepsilon }} \longrightarrow 0. 
\end{equation}
Then we obtain an isomorphism 
\begin{equation}
\rsect_c( M_{i,t} ; (K_i)_{M_{i,t} \setminus M_{i, t- \varepsilon }} [1]) 
\simeq 
\rsect_c( M_{i,t} ; (K_i)_{L_{i, t- \varepsilon }}) 
\simeq \CC^{r_i}. 
\end{equation}
This explains the reason why we have the factors 
$\bigl(\sfE_{V\vbar \var{Y}^{\an}}^{-c_i}\bigr)^{r_i}$ $(1\leq i\leq l)$ 
in the assertion.
This completes the proof.
\end{proof}

Now we apply Lemma \ref{lem-K1} and 
Lemma \ref{lem-bounded} to Proposition \ref{prop-infty-enh}.
Then we obtain an isomorphism
\begin{align}
\pi^{-1}\CC_V\otimes Sol_{\var{Y}}^{\rmE}(\tl{\SM^\wedge}) \quad &\simeq \quad 
\CC_{\var{Y}^{\an}}^{\rmE} 
\Potimes \( \pi^{-1}\CC_V\otimes{}^{\sfL}G \) \\
&\simeq \quad \bigoplus_{i=1}^l\biggl\{
\bigoplus_{j=1}^{n_i}
\bigl(\EE_{V\vbar \var{Y}^{\an}}^{-c_{ij}}\bigr)^{N_{ij}}
\oplus\bigl(\EE_{V\vbar \var{Y}^{\an}}^{-c_i}\bigr)^{r_i}\biggr\} 
\oplus \bigoplus_{j=1}^{n_\infty}
\bigl(\EE_{V\vbar \var{Y}^{\an}}^{-c_{\infty j}}\bigr)^{N_{\infty j}}.
\end{align}
We can also rewrite the right hand side to 
\begin{align}
\bigoplus_{i=1}^n(\EE_{V\vbar\var{Y}^{\an}}^{\Re g_i})^{\oplus d_i}
\end{align}
by Lemma \ref{lem-T7}. This completes the proof of Theorem \ref{Thm-T3}.
\qed \\ 

We can describe the generic rank $d(\SM)$ of
the Fourier transform $\SM^\wedge$
in terms of the irregularities of $\SM$ as follows.
First, note that for
$\theta\in S_{a_i}X^{\an}$ $(1\leq i\leq l)$
the non-negative number
\begin{align}
\sum_{f\in(N_i^{>0})_\theta}N_i(f)\cdot\mathrm{ord}_{a_i}(f)\,\geq0
\end{align}
does not depend on the choice of $\theta\in S_{a_i}X^{\an}$.
We know moreover that it is an integer and
call it the irregularity of the meromorphic connection $\SM^{\an}$
at $a_i\in D^{\an}$ (see e.g. Sabbah \cite{Sab93}).
We denote it by $\mathrm{irr}_{a_i}(\SM)$.
Note also that for $\theta\in S_\infty\var{X}^{\an}$
the non-negative number
\begin{align}
\sum_{f\in(N_\infty^{>0})_\theta}N_\infty(f)\cdot
\min\left\{\mathrm{ord}_\infty(f)-1,0\right\}\,\geq0
\end{align}
does not depend on $\theta\in S_\infty\var{X}^{\an}$.
We can easily show that it is an integer
(see e.g. \cite[Lemma 2.4]{Sab08}).
Then we denote it by $e_\infty(\SM)\in\ZZ_{\geq0}$ and
obtain the following corollary of Theorem \ref{Thm-T3}.

\begin{corollary}\label{Cor-T8}
The generic rank $\rk(\SM^\wedge)$ of
the Fourier transform $\SM^\wedge$ of $\SM$ is equal to
\begin{align}
\left\{\sum_{i=1}^{l}
\(\irr_{a_i}(\SM)+\rk(\SM)\)\right\}
+e_\infty(\SM).
\end{align}
\end{corollary}

\subsection{The proofs of Theorems 
\ref{thm-on-bnd} and \ref{thm-CharCyc}
and related results} 
\label{sec:newT4}

First, let us prove Theorem \ref{thm-on-bnd}. 
Since the sector $W$ is simply connected, 
in view of the proof of Proposition \ref{prop:hariaw} 
it suffices to treat the case where it is 
sufficiently narrow. 
For $w\in W\subset B(b)^\circ$ let
$\gamma_{\infty j}^b(w)\in
{\rm Int}D (a_\infty)^\circ$ $(1\leq j\leq n_\infty(b))$ be the points in the punctured
disk
${\rm Int}D (a_\infty)^\circ$ such that
\begin{equation}
(f^w)^\prime(\gamma_{\infty j}^b(w))=0
\end{equation}
for some (non-zero) exponential factor
$f\in N_{\infty, b}^{>0}$ such that $f\neq bz$ and set
\begin{equation}
c_{\infty j}^b\coloneq
\Re (f^w)(\gamma_{\infty j}^b(w))\in\RR
\quad (1\leq j\leq n_\infty(b))
\end{equation}
and
\begin{equation}
N_{\infty j}^b\coloneq N_\infty(f)>0 \quad (1\leq j\leq n_\infty(b)).
\end{equation}
Then we have 
\begin{equation}
    \sum_{j=1}^{n_\infty(b)} N_{\infty j}^b
    =\sum_{i=1}^m e_i
    =d(\SM)^b.
\end{equation}
For a point $a_i\in D^{\an}$ and
$w\in W\subset B(b)^\circ$ let us count
the total number of the critical points of
the real-valued functions
$\Re f^w\vert_{\partial D(a_i)}\colon
\partial D(a_i)\to\RR$ $(f\in N_i^{>0},f\neq0)$.
Namely we count them with
the multiplicities $N_i(f)$.
Since $\Re f^w\vert_{\partial D(a_i)}$ is
a linear perturbation of $- \Re f
\vert_{\partial D(a_i)}$,
after shrinking the disk $D(a_i)$
if necessary, it suffices to count the total number of
the critical points of the real-valued functions
$\Re f\vert_{\partial D(a_i)}\colon
\partial D(a_i)\to\RR$ $(f\in N_i^{>0},f\neq0)$.
For a point $\theta\in S_{a_i}X^{\an}$
we say that $f_1,f_2\in (N_i^{>0})_\theta$ are
equivalent if $f_1$ is analytically continued to
$f_2$ along some path in the punctured disk $B(a_i)^\circ$.
Let $(N_i^{>0})_\theta^\sim$ be
the quotient set of $(N_i^{>0})_\theta$
obtained by this equivalence relation.
By analytic continuations,
for $\theta_1,\theta_2\in S_{a_i}X^{\an}$ there exists
a (unique) bijection between
$(N_i^{>0})_{\theta_1}^\sim$ and $(N_i^{>0})_{\theta_2}^\sim$.
Hence by fixing a point
$\theta\in S_{a_i}X^{\an}$ we set
$(N_i^{>0})^\sim\coloneq(N_i^{>0})_\theta^\sim$ for short.
For $[f]\in(N_i^{>0})^\sim$ $(f\in(N_i^{>0})_\theta)$
denote by $v([f])$ the number of the elements of
$(N_i^{>0})_\theta$ equivalent to $f$.
Note that if $f_1,f_2\in(N_i^{>0})_\theta$ are
equivalent then $N_i(f_1)=N_i(f_2)$ and
$\ord_{a_i}(f_1)=\ord_{a_i}(f_2)$.
We thus obtain morphisms
\begin{equation}
N_i\colon(N_i^{>0})^\sim
\longrightarrow\ZZ_{>0},\quad
\ord_{a_i}(\cdot)\colon(N_i^{>0})^\sim\bs\{0\}
\longrightarrow\QQ_{>0}.
\end{equation}
Then we can easily show that for
$w\in W\subset B(b)^\circ$ the total number of
the critical points of the functions
$\Re f^w\vert_{\partial D(a_i)}\colon
\partial D(a_i)\to\RR$ $(f\in N_i^{>0},f\neq0)$ is equal to
\begin{equation}
2\times\sum_{[f]\in (N_i^{>0})^\sim\bs\{0\}}
N_i([f])\cdot v([f])\cdot\ord_{a_i}([f]).
\end{equation}
The number of the critical points with
local maximal value is the half of it i.e.
\begin{equation}
m_i(b)\coloneq\sum_{[f]\in (N_i^{>0})^\sim\bs\{0\}}
N_i([f])\cdot v([f])\cdot\ord_{a_i}([f])
=\irr_{a_i}(\SM)
\end{equation}
and we denote by $B_{ik}^+(w)$ $(1\leq k\leq m_i(b))$
the corresponding local maximal values. Note that 
they are bounded continuous functions on $W$. 

Similarly, also for the point $a_\infty=\infty\in\tl{D}$,
we define $(N_\infty^{>0})^\sim$ and
the morphisms $v\colon(N_\infty^{>0})^\sim\to\ZZ_{>0}$,
$N_{\infty}\colon(N_\infty^{>0})^\sim\to\ZZ_{>0}$,
$\ord_{a_\infty}(\cdot)\colon
(N_\infty^{>0})^\sim\bs\{0\}\to\QQ_{>0}$.
Let $(N_{\infty,b}^{>0})^\sim$ be
the image of $N_{\infty,b}^{>0}$ in
$(N_\infty^{>0})^\sim$.
Then we can easily show that
the total number of the critical points with
local maximal value of the functions
$\Re f^w\vert_{\partial D(a_\infty)}\colon
\partial D(a_\infty)\to\RR$ $(f\in N_\infty^{>0})$
is equal to
\begin{align}
m_\infty(b)&\coloneq
\sum_{[f]\in(N_\infty^{>0})^\sim,\ \ord_{\infty}(f)>1}
N_\infty([f])\cdot v([f])\cdot\ord_\infty([f]) \notag \\
&+ \sum_{[f]\in(N_\infty^{>0})^\sim\bs(N_{\infty,b}^{>0})^\sim,
\ \ord_\infty([f])\leq1}
N_\infty([f])\cdot v([f]) \notag \\
&+ \sum_{[f]\in(N_{\infty,b}^{>0})^\sim,\ f\neq bz}
N_\infty([f])\cdot v([f])\cdot\ord_\infty([f-bz]) \notag \\
&+N_\infty(bz).
\end{align}
Note that we have $m_\infty(b)\geq N_\infty(bz)$.
We denote by
$B_{\infty k}^+(w)$ $(1\leq k\leq m_\infty(b))$
the corresponding local maximal values. 
As in the proof of Theorem \ref{Thm-T3},
by the Morse function $\xi^w\colon\Omega\to\RR$ $(w\in W)$
we define also real-valued functions
$c_i\colon W\to\RR$, $L_i^\pm\colon
W\to\RR$ $(1\leq i\leq l)$ and
$L_\infty^\pm\colon W\to\RR$. Also these functions 
are bounded and continuous on $W$. 
\begin{figure}
    \centering    
    \begin{tikzpicture}[scale=0.45]
        \begin{scope}
            \draw (-4.5,3) node[above]{$B(a_\infty)^\circ$};
            \draw (5,5) circle[radius=0.075] node[below]{$a_\infty$};
            \draw (5,3) node[right]{$U(W)$};
            \clip (5,5)--(5,-5)--(-5,-5)--(-5,5)--(5,5);
            \begin{scope}[rotate=225]
                \coordinate (O) at(0,0);
                \coordinate (R) at(9,0);
                \coordinate (A) at(0,9);
                \coordinate (L) at(-9,0);
                \coordinate (B) at(0,-9);
                \begin{scope}
                    \clip (O) circle[radius=4.5];
                    \fill[lightgray] (-3.5,-9)rectangle(-6,9);
                    \draw[dash pattern=on 3pt off 3pt] (-3.5,-9)--(-3.5,9);  
                \end{scope}
                    \draw[dash pattern=on 3pt off 3pt] (O) circle[radius=4.5];
                    \fill[white] (0,-2) circle[radius=0.075];
                    \draw (0,-2) circle[radius=0.075]; 
                    \draw[dash pattern=on 3pt off 3pt] (0,-2) circle[radius=0.5];

                    \fill[white] (2,1) circle[radius=0.075];
                    \draw[dash pattern=on 3pt off 3pt] (2,1) circle[radius=1.2];
                    \draw (2,1) circle[radius=0.075] node[below]{$a_i$};
                    \draw (1.2,-1.7) node[left,below]{$B(a_i)^\circ$};

                    \fill[white] (-1,1) circle[radius=0.075] ;
                    \draw (-1,1) circle[radius=0.075] ;
                    \draw[dash pattern=on 3pt off 3pt] (-1,1) circle[radius=0.5];
            \end{scope}
            \draw (5,3)--(3.25,2.5);
        \end{scope}
    \end{tikzpicture}
    \caption{$U(W)$}
    \label{fig:UW}
\end{figure}
Let $\beta\in\CC^\ast=\CC\bs\{0\}$ be
a non-zero complex number such that the sector
$W\subset B(b)^\circ$ is a sectorial neighborhood of
the point $b+\beta\cdot0\in S_bY^{\an}$ and
fix a point $w_0\in(b+\RR_{>0}\beta)\cap W$.
For $w\in W$ let $P(a_\infty)_+^w\in\partial D(a_\infty)$
be the point where the real-valued function
$\tl{\xi^w}\vert_{\partial D(a_\infty)}\colon
\partial D(a_\infty)\to\RR$ takes
its maximal value $L_\infty^+(w)$.
Then for a sufficiently small $\varepsilon_0>0$
we define an open subset
$U(W)\subset\Omega=X^{\an}\bs(D(a_1)\cup
\dots\cup D(a_l)\cup D(a_\infty))\subset U^{\an}$ by
\begin{equation}
U(W)\coloneq\Set*{z\in\Omega}
{\tl{\xi^{w_0}}(z)>L_\infty^+(w_0)-\varepsilon_0}.
\end{equation}
It is clear that $U(W)$ is convex and
$P(a_\infty)_+^{w_0}\in\overline{U(W)}
\cap\partial D(a_\infty)$ (see Figure \ref{fig:UW}).
Moreover, shrinking the sector
$W\subset B(b)^\circ$ if necessary,
we may assume that $W$ is a
sectorial open neighborhood of the point
$b+\beta\cdot0\in S_bY^{\an}$ and
for any $w\in W$ we have
$P(a_\infty)_+^w\in\overline{U(W)} \cap\partial D(a_\infty)$.
In order to describe the restriction
$\sfE_W$ of the enhanced sheaf
$\pi^{-1}\CC_W\otimes{}^\sfL G\in\BEC(\CC_{\var{Y}^{\an}})$
to $Y^{\an}\subset\var{Y}^{\an}$ we set
\begin{align}
\begin{cases}
\ \sfE_W^\prime\coloneq\pi^{-1}\CC_W\otimes
\rmR p_{2!}(p_1^{-1}G_{(X^{\an}\bs U(W))\times\RR}^\circ
\otimes\CC_{\{t-s-\Re zw\geq0\}}[1]), \\
\ \sfE_W^{\prime\prime}\coloneq\pi^{-1}\CC_W\otimes
\rmR p_{2!}(p_1^{-1}G_{U(W)\times\RR}^\circ
\otimes\CC_{\{t-s-\Re zw\geq0\}}[1]), \\
\end{cases}
\end{align}
and consider the distinguished triangle
\begin{equation}
\sfE_W^{\prime\prime}\longrightarrow\sfE_W
\longrightarrow\sfE_W^\prime\overset{+1}{\longrightarrow}.
\end{equation}
We need this decomposition of $\sfE_W$, because 
we do not know so far how to calculate it directly. 
Then by applying the Morse theoretical method in
the proof of Theorem \ref{Thm-T3} to our situation,
we obtain an isomorphism
\begin{align}\label{eq:sfEW}
\sfE_W^\prime&\simeq
\bigoplus_{i=1}^l\Biggl\{\bigoplus_{k=1}^{m_i(b)}
\sfE_{W\vbar Y^{\an}}^{-B_{ik}^+}
\oplus\bigl(\sfE_{W\vbar Y^{\an}}
^{-L_i^-+c\,\vartriangleright -L_i^-}[1]
\bigr)^{r_i} \notag \\
&\oplus\bigl(\sfE_{W\vbar Y^{\an}}^{-L_i^-}\bigr)^{\rk(\SM)-r_i}
\oplus\bigl(\sfE_{W\vbar Y^{\an}}^{-c_i}\bigr)^{r_i}\Biggr\}
\oplus\Biggl(\bigoplus_{k=1}^{m_\infty(b)}
\sfE_{W\vbar Y^{\an}}^{-B_{\infty k}^+}\Biggr) \notag \\
&\oplus \Biggl\{\bigoplus_{j=1}^{n_\infty(b)}
\bigl(\sfE_{W\vbar Y^{\an}}^{-c_{\infty j}^b}\bigr)
^{N_{\infty j}^b}\Biggr\}.
\end{align}
Let us set
\begin{align}
\begin{cases}
(\sfE_W^\prime)_\reg \coloneq
&\displaystyle\bigoplus_{i=1}^l\Biggl\{
\bigoplus_{k=1}^{m_i(b)}\sfE_{W\vbar Y^{\an}}^{-B_{ik}^+}
\oplus\bigl(\sfE_{W\vbar Y^{\an}}
^{-L_i^-+c\,\vartriangleright -L_i^-}[1]\bigr)^{r_i} \\
&\oplus\bigl(\sfE_{W\vbar Y^{\an}}^{-L_i^-}\bigr)^{\rk(\SM)-r_i}
\oplus\bigl(\sfE_{W\vbar Y^{\an}}^{-c_i}\bigr)
^{r_i}\Biggr\}
\oplus\Biggl(
\displaystyle\bigoplus_{k=1}^{m_\infty(b)}
\sfE_{W\vbar Y^{\an}}^{-B_{\infty k}^+}\Biggr),
 \\
(\sfE_W^\prime)_\irr \coloneq
&\Biggl\{\displaystyle\bigoplus_{j=1}^{n_\infty(b)}
\bigl(\sfE_{W\vbar Y^{\an}}^{-c_{\infty j}^b}\bigr)
^{N_{\infty j}^b}\Biggr\}
\end{cases}
\end{align}
so that we have an isomorphism
$\sfE_W^\prime\simeq(\sfE_W^\prime)_\reg
\oplus(\sfE_W^\prime)_\irr$.
Then by Lemma \ref{lem-K1} it is easy to see that
the enhanced ind-sheaf
$\CC_{Y^\an}^\rmE\Potimes(\sfE_W^\prime)_\reg$
is isomorphic to $( \EE_{W\vbar Y^{\an}}^0)^{\nu_b}$,
where we set
\begin{equation}\label{eq:natural}
\nu_b\coloneq
\Biggl\{\sum_{i=1}^l(\irr_{a_i}(\SM)+\rk(\SM))\Biggr\}
+m_\infty(b).
\end{equation}
We can also show that
\begin{equation}\label{eq:star}
m_\infty(b)+
\biggl(\sum_{j=1}^{n_\infty(b)}N_{\infty j}^b\biggr)
=e_\infty(b)+\rk(\SM).
\end{equation}
Indeed, for $f\in N_{\infty,b}^{>0}$ such that $f\neq bz$ we set 
\begin{equation}
    \lambda\coloneq\ord_\infty(f-bz) \quad \in(0,1).
\end{equation}
Then there exists a non-zero complex number $\alpha\in\CC$ such that
\begin{equation}
    f(z)=bz+\alpha z^\lambda+
    (\textit{lower order terms})
\end{equation}
and hence for $w\in W$ and $z\in B(a_\infty)^\circ$ the condition
$(f^w)^\prime(z)=0$ is equivalent to 
\begin{equation}
    w-b=\alpha\lambda z^{\lambda-1}+(\textit{lower order terms}).
\end{equation}
This implies that we have 
\begin{equation}
    \sum_{j=1}^{n_\infty(b)}N_{\infty j}^b
    =\sum_{[f]\in(N_{\infty,b}^{>0})^\sim,\ f\neq bz}
    N_\infty([f])\cdot v([f])\cdot(1-\ord_\infty([f-bz]))
\end{equation}
from which (\ref{eq:star}) immediately follows.
Similarly, for $\sfE_W^{\prime\prime}$
we obtain an isomorphism
\begin{equation}\label{eq:Ewprpr}
\sfE_W^{\prime\prime}[1]\simeq
(\sfE_{W\vbar Y^{\an}}^{-L_\infty^+})^{\rk(\SM)}.
\end{equation}
Then it follows from the distinguished triangle
\begin{equation}
\sfE_W\longrightarrow\sfE_W^\prime\longrightarrow
\sfE_W^{\prime\prime} [1] \overset{+1}{\longrightarrow}
\end{equation}
that there exists an exact sequence
\begin{equation}
0\longrightarrow H^0\sfE_W\simeq\sfE_W\longrightarrow
H^0\sfE_W^\prime\overset{\Phi_W}{\longrightarrow}
H^0\sfE_W^{\prime\prime}[1]\simeq\sfE_W^{\prime\prime}[1]
\longrightarrow0
\end{equation}
of enhanced sheaves on $Y^{\an}$.
Note that by Lemma \ref{lem-T7} (ii) we have 
\begin{equation}
\CC_{Y^\an}^\rmE\Potimes H^0(\sfE_W^\prime)_\irr\simeq
\bigoplus_{j=1}^{n_\infty(b)}
(\EE_{W\vbar Y^\an}^{-c_{\infty j}^b})^{N_{\infty j}^b}
\simeq\bigoplus_{i=1}^m(\EE_{W\vbar Y^\an}^{\Re h_i})^{e_i}
\end{equation}
and
\begin{equation}\label{eq:dMb}
\sum_{j=1}^{n_{\infty}(b)}N_{\infty j}^b=
\sum_{i=1}^m e_i=d(\SM)^b.
\end{equation}
Moreover there exist isomorphisms
\begin{equation}
\CC_{Y^\an}^\rmE\Potimes 
H^0(\sfE_W)\simeq\pi^{-1}\CC_W\otimes Sol_Y^\rmE(\SM^\wedge)
\end{equation}
and
\begin{equation}
\CC_{Y^\an}^\rmE\Potimes H^0(\sfE_W^{\prime\prime}[1])
\simeq(\EE_{W\vbar Y^\an}^0)^{\rk(\SM)}.
\end{equation}
We thus obtain an exact sequence
\begin{equation}
\begin{split}
0\longrightarrow\pi^{-1}\CC_W\otimes 
Sol_Y^\rmE(\SM^\wedge)\longrightarrow
\biggl\{\bigoplus_{i=1}^m
(\EE_{W\vbar Y^\an}^{\Re h_i})^{e_i}\biggr\}
\oplus(\EE_{W\vbar Y^\an}^0)^{\nu_b} \\
\longrightarrow (\EE_{W\vbar Y^\an}^0)^{\rk(\SM)}\longrightarrow0
\end{split}
\end{equation}
of enhanced ind-sheaves on $Y^\an$.
Then in view of (\ref{eq:natural}) and (\ref{eq:star}) and 
Lemma \ref{lem-ITa}, we can apply 
the multiplicity test functor 
in \cite[Section 6.3]{DK18} to 
obtain the first assertion of Theorem \ref{thm-on-bnd}.
If $b\neq0$ and $N_{\infty,b}^{>0}=\emptyset$,
then we can replace $W\subset B(b)$ by
the open disk $B(b)$ to obtain
the last assertion of the theorem.
This completes the proof of Theorem \ref{thm-on-bnd}.
\qed \\ 

Next, let us prove Theorem \ref{thm-CharCyc}. 
Recall that we have 
\begin{equation}
    H^j_{\{b\}}(\SM^\wedge)\simeq0 \quad (j\neq0,1)
\end{equation}
and for $j=0,1$ the holonomic 
$\SD_Y$-modules $H^j_{\{b\}}(\SM^\wedge)$ are 
direct sums of some copies of the standard one 
\begin{equation}
    \SB_{\{b\}\vbar Y}\coloneq 
    H^1_{\{b\}}(\SO_Y)\simeq\SD_Y / \SD_Y(w-b).
\end{equation}
Let $k_0,k_1\in\ZZ_{\geq0}$ be the 
non-negative integers such that 
\begin{equation}
    H_{\{b\}}^j(\SM^\wedge)\simeq 
    \SB_{\{b\}\vbar Y}^{\oplus k_j}\quad (j=0,1).
\end{equation}
Then by applying the functor $Sol_Y(\cdot)$ 
to the distinguished triangle
\begin{equation}\label{eq:dt-rsectb1}
    \tau^{\leq0}\rsect_{\{b\}}(\SM^\wedge)
    \simeq\Gamma_{\{b\}}(\SM^\wedge)\longrightarrow
    \rsect_{\{b\}}(\SM^\wedge)\longrightarrow
    \tau^{\geq1}\rsect_{\{b\}}(\SM^\wedge)
    \simeq H_{\{b\}}^1(\SM^\wedge)[-1]
    \overset{+1}{\longrightarrow},
\end{equation} 
we obtain 
\begin{equation}
    \chi_b(Sol_Y(\rsect_{\{b\}}(\SM^\wedge)))=k_1-k_0
\end{equation}
where $\chi_b(Sol_Y(\rsect_{\{b\}}(\SM^\wedge)))$ is the local Euler-Poincar\'{e} index 
\begin{equation}
    \chi_b(Sol_Y(\rsect_{\{b\}}(\SM^\wedge)))\coloneq
    \sum_{j\in\ZZ}(-1)^j\dim_\CC H^j Sol_Y(\rsect_{\{b\}}(\SM^\wedge))_b
\end{equation}
of $Sol_Y(\rsect_{\{b\}}(\SM^\wedge))$ at $b\in Y^\an$.
On the other hand, it follows from the distinguished triangle 
\begin{equation}
    \rsect_{\{b\}}(\SM^\wedge)\longrightarrow
    \SM^\wedge\longrightarrow
    \rsect_{Y\setminus\{b\}}(\SM^\wedge)
    \simeq\Gamma_{Y\setminus\{b\}}
    (\SM^\wedge)\overset{+1}{\longrightarrow}
\end{equation}
that we have 
\begin{equation}
    \chi_b(Sol_Y(\SM^\wedge))
    =\chi_b(Sol_Y(\Gamma_{Y\setminus\{b\}}(\SM^\wedge)))
    +\chi_b(Sol_Y(\rsect_{\{b\}}(\SM^\wedge))).
\end{equation}
Moreover by Kashiwara's index theorem for holonomic 
$\SD$-modules (see \cite{Kas83}) we have 
\begin{align}
    \chi_b(Sol_Y(\SM^\wedge)) 
    &= \mult_{T_Y^\ast Y}(\SM^\wedge)
    -\mult_{T_b^\ast Y}(\SM^\wedge) \notag \\
    &=\rk(\SM^\wedge)-\mult_{T_b^\ast Y}(\SM^\wedge)  \notag \\
    &=d(\SM)-\mult_{T_b^\ast Y}(\SM^\wedge)
\end{align}
and by applying Proposition \ref{prop:IT20a-3.14} 
to the meromorphic connection 
$\Gamma_{Y\setminus\{b\}}(\SM^\wedge)$ along $b\in Y$ we obtain 
\begin{equation}
    \chi_b(Sol_Y(\Gamma_{Y\setminus\{b\}}(\SM^\wedge)))
    =-\irr_b(\Gamma_{Y\setminus\{b\}}(\SM^\wedge)).
\end{equation}
Combining these results together, we have 
\begin{equation}\label{eq:indexthm}
    \mult_{T_b^\ast Y}(\SM^\wedge) 
    =d(\SM)+\irr_b(\Gamma_{Y\setminus\{b\}}(\SM^\wedge))
    -(k_1-k_0).
\end{equation}
Therefore, for the proof of the first assertion, it suffices to compute 
the local Euler-Poincar\'{e} index 
\begin{equation}
    \chi_b(Sol_Y(\rsect_{\{b\}}(\SM^\wedge)))=k_1-k_0
\end{equation}
of $Sol_Y(\rsect_{\{b\}}(\SM^\wedge))$ at $b\in Y^\an$.
Moreover by \cite[Lemma 4.1]{IT20b}, we may assume that $b\neq0$.
As in the proof of Theorem \ref{thm-on-bnd}, 
for a sufficiently small $\varepsilon_0>0$ we define an open subset $U(b)\subset\Omega$ by 
\begin{equation}
    U(b)\coloneq\Set*{z\in\Omega}
    {\tl{\xi^b}(z)>L_\infty^+(b)-\varepsilon_0}
\end{equation}
and set
\begin{align}
    \begin{cases}
        \ \sfE_b\coloneq\pi^{-1}\CC_{\{b\}}\otimes{}^\sfL G, \\
        \ \sfE_b^\prime\coloneq\pi^{-1}\CC_{\{b\}}\otimes
        \rmR p_{2!}(p_1^{-1}G_{(X^{\an}\bs U(b))\times\RR}^\circ
        \otimes\CC_{\{t-s-\Re zw\geq0\}}[1]), \\
        \ \sfE_b^{\prime\prime}\coloneq\pi^{-1}\CC_{\{b\}}\otimes
        \rmR p_{2!}(p_1^{-1}G_{U(b)\times\RR}^\circ
        \otimes\CC_{\{t-s-\Re zw\geq0\}}[1]). 
    \end{cases}
\end{align}
Note that we have 
\begin{align}
    \CC_{Y^\an}^\rmE\Potimes\sfE_b&\simeq\pi^{-1}\CC_{\{b\}}\otimes Sol_Y^\rmE(\SM^\wedge) \notag \\
    &\simeq Sol_Y^\rmE(\rsect_{\{b\}}(\SM^\wedge)).
\end{align}
Let us consider the distinguished triangle
\begin{equation}\label{eq:sfEb}
    \sfE_b^{\prime\prime}\longrightarrow\sfE_b
    \longrightarrow\sfE_b^\prime\overset{+1}{\longrightarrow}.
\end{equation}
By applying the Morse theoretical method 
in the proof of Theorem \ref{Thm-T3} to our situation,
we can show that there exist isomorphisms
\begin{equation}\label{eq:sfEb1}
    \CC_{Y^\an}^\rmE\Potimes\sfE_b^\prime\simeq
    \pi^{-1}\CC_{\{b\}}\otimes(\CC_{Y^\an}^\rmE)^{\nu_b-N_\infty(bz)}
\end{equation}
and 
\begin{equation}\label{eq:sfEb2}
    \CC_{Y^\an}^\rmE\Potimes(\sfE_b^{\prime\prime}[1])\simeq
    \pi^{-1}\CC_{\{b\}}\otimes(\CC_{Y^\an}^\rmE)^{\rk(\SM)}.
\end{equation}
Indeed, for $f\in N_{\infty,b}^{>0}$ such that 
$f\neq bz$ there exists a non-zero Puiseux germ 
$f_\red\in\SP_{S_\infty\var{Y}^\an}^\prime$ 
such that $0<\ord_\infty(f_\red)<1$ and 
\begin{equation}
    f(z)=bz+f_\red(z).
\end{equation}
This implies that for a sufficiently small 
punctured disk $D(a_\infty)^\circ$ 
centered at the point $a_\infty=\infty\in\var{Y}^\an$, 
the Morse function 
\begin{equation}
    \Re f^b(z)=-\Re f_\red(z)
\end{equation}
has no critical point in 
$B(a_\infty)^\circ=\Int D(a_\infty)^\circ$.
Moreover, for the linear factor 
$f=bz\in N_{\infty,b}^{>0}$ the Morse function $\Re f^b(z)$ 
is identically zero on $D(a_\infty)^\circ$ and hence for any $t\in\RR$ 
its level set at $t$ is equal to 
\begin{equation}
    \Set*{z\in D(a_\infty)^\circ}{\Re f^b(z)\leq t}=
    \begin{cases}
        \ D(a_\infty)^\circ & (t\geq0), \\
        \ \emptyset & (t<0).
    \end{cases}
\end{equation}
As in the proof of Lemma \ref{lem-vanish}, 
for any local system $L_0$ on $D(a_\infty)^\circ$, we can show the vanishing
\begin{equation}
    \rsect_c(D(a_\infty)^\circ; L_0)\simeq0.
\end{equation}
Therefore, the isomorphism (\ref{eq:sfEb1}) 
is obtained in the same way as (\ref{eq:sfEW}).
We can show the isomorphism (\ref{eq:sfEb2}) 
as we showed (\ref{eq:Ewprpr}).
By applying the functor $Sol_Y^\rmE(\cdot)$ to 
the distinguished triangle (\ref{eq:dt-rsectb1}) 
we obtain a distinguished triangle 
\begin{equation}\label{eq:dt-rsectb2}
    Sol_Y^\rmE(\SB_{\{b\}\vbar Y}^{\oplus k_1})[1]\longrightarrow
    Sol_Y^\rmE(\rsect_{\{b\}}(\SM^\wedge))\longrightarrow
    Sol_Y^\rmE(\SB_{\{b\}\vbar Y}^{\oplus k_0})
    \overset{+1}{\longrightarrow}.
\end{equation}
Note that for the regular holonomic 
$\SD_Y$-modules $\SB_{\{b\}\vbar Y}$ by 
Proposition \ref{prop-K2} (iv) there exists an isomorphism
\begin{equation}
    Sol_Y^\rmE(\SB_{\{b\}\vbar Y})
    \simeq \pi^{-1}\CC_{\{b\}}\otimes \CC_{Y^\an}^\rmE[-1].
\end{equation}
Then we obtain
\begin{align}
    H^j(\CC_{Y^\an}^\rmE\Potimes\sfE_b)
    &\simeq H^j Sol_Y^\rmE(\rsect_{\{b\}}(\SM^\wedge)) \notag \\
    &\simeq 
    \begin{cases}
        \ \pi^{-1}\CC_{\{b\}}\otimes(\CC_{Y^\an}^\rmE)^{k_{1-j}} 
        & (j=0,1), \\
        \ 0 & (otherwise).
    \end{cases}
\end{align}
Therefore, it follows from 
the distinguished triangle (\ref{eq:sfEb})
that we have 
\begin{equation}
    k_1-k_0=\nu_b-N_\infty(bz)-\rk(\SM).
\end{equation}
Combining this equation with 
Corollary \ref{Cor-T8}, (\ref{eq:natural}), (\ref{eq:star}) and (\ref{eq:dMb}), we obtain 
the first assertion of Theorem \ref{thm-on-bnd}.
Now let us consider the case where $N_{\infty,b}^{>0}$ 
does not contain the linear factor $bz$. 
In this case, as in the proof of (\ref{eq:sfEb1}),
we can directly show 
\begin{align}\label{eq:nolinear}
    Sol_Y^\rmE(\rsect_{\{b\}}(\SM^\wedge))
    &\simeq \CC_{Y^\an}^\rmE\Potimes\sfE_b \notag \\
    &\simeq \pi^{-1}\CC_{\{b\}}\otimes(\CC_{Y^\an}^\rmE)^{d(\SM)-d(\SM)^b}.
\end{align}
Then by comparing (\ref{eq:nolinear}) with (\ref{eq:dt-rsectb2}) 
we obtain $k_0=0$, $k_1=d(\SM)-d(\SM)^b$ and 
hence the second assertion. 
This completes the proof of Theorem \ref{thm-CharCyc}.
\qed \\ 

We recall the following definition of Verdier \cite{Ver83} in 
the simplest case of $N=1$.

\begin{definition}[Verdier \cite{Ver83}]
We say that a $\CC$-constructible sheaf $\SG\in\BDC_{\mathrm{c}}(Y^{\an})$ 
on $Y^{\an}=\CC$ is monodromic if 
$H^j\SG\vbar_{\CC\bs\{0\}}$ is a local system on 
$\CC\bs\{0\}\subset Y^{\an}=\CC$ for any $j\in\ZZ$.
\end{definition}

Then we obtain the following very simple consequence of 
Theorem \ref{thm-CharCyc} (see \cite{IT20b} for a related 
result in higher dimensions). 

\begin{corollary}\label{cor:monod}
For the meromorphic connection $\SM$ on $X=\CC$, the solution 
complex $Sol_Y(\SM^\wedge)\in\BDC_{\mathrm{c}}(Y^\an)$ of 
its Fourier transform $\SM^\wedge$ is monodromic if and only if 
$N_{\infty,b}^{>0}=\emptyset$ for any $b\in\CC\bs\{0\}$. 
\end{corollary}

\begin{proof}
By Theorem \ref{thm-CharCyc}, $Sol_Y(\SM^\wedge)\in 
\BDC_{\mathrm{c}}(Y^\an)$ is monodromic if and only if for any $b\in\CC\bs\{0\}$
\begin{equation}
    d(\SM)^b+\irr_b(\Gamma_{Y\setminus\{b\}}(\SM^\wedge))+N_\infty(bz)=0.
\end{equation}
We fix $b\in\CC\bs\{0\}$. 
Note that $d(\SM)^b$, $\irr_b(\Gamma_{Y\setminus\{b\}}(\SM^\wedge))$ 
and $N_\infty(bz)$ are non-negative integers.
Clearly, $N_\infty(bz)=0$ is equivalent to $bz\notin N_{\infty,b}^{>0}$.
Moreover we can easily check that $d(\SM)^b=
\irr_b(\Gamma_{Y\setminus\{b\}}(\SM^\wedge))=0$ if and only if $N_{\infty,b}^{>0}\bs\{bz\}=\emptyset$. 
This completes the proof.
\end{proof}

We also obtain the following consequence of the results of this section.

\begin{corollary}\label{cor:regular}
For the meromorphic connection $\SM$ on $X=\CC$, 
the extension $\tl{\SM^\wedge}$ of 
its Fourier transform $\SM^\wedge$ is 
a regular holonomic $\SD_{\var{Y}}$-module 
if and only if the following conditions are satisfied.
\begin{enumerate}
    \item [\rm (i)] The solution complex 
    $Sol_X(\SM)\in \BDC_{\mathrm{c}}(X^\an)$ 
    of $\SM$ is monodromic.
    \item [\rm (ii)] The regular rank  
    of $\SM$ at $0\in X^\an=\CC$ is 
    equal to the generic rank of $\SM$.
    \item [\rm (iii)] The set $N_\infty^{>0}$ of exponential factors 
    of $\SM$ at $\infty\in\var{X}^\an$ 
    consists only of linear factors.
\end{enumerate}
\end{corollary}

\begin{proof}
Recall that $\tl{\SM^\wedge}$ is regular if and only if for any $b\in\var{Y}^\an$ 
the regular rank $r_b^\prime\in\ZZ_{\geq0}$ of $\tl{\SM^\wedge}$ at $b$ is 
equal to the generic rank $\rk(\SM^\wedge)$ of $\SM^\wedge$.
For $b\in Y^\an=\CC$, by Theorems \ref{thm-on-bnd} and \ref{thm-CharCyc} 
we can show that $r_b^\prime=\rk(\SM^\wedge)$ is equivalent to $N_{\infty,b}^{>0}\subset\{bz\}$.
Moreover, it follows from Theorem \ref{Thm-T3} that 
$r_\infty^\prime=\rk(\SM^\wedge)$ if and only if (i), (ii) and 
\begin{equation}
    N_\infty^{>0}=\bigsqcup_{b\in \CC} N_{\infty,b}^{>0}.    
\end{equation}
This completes the proof. 
\end{proof}

\subsection{Generalizations of Theorems \ref{Thm-T3} and \ref{thm-on-bnd} 
to arbitrary holonomic $\SD$-modules}\label{sec:holD}
In this subsection, 
we prove results similar to 
Theorems \ref{Thm-T3} and \ref{thm-on-bnd} 
in the case where the algebraic holonomic $\SD$-module 
$\SM\in\Modhol(\SD_X)$ on $X=\CC_z$ is 
not necessarily a meromorphic connection. 
Let $D\subset X$ be its 
singular support $\mathrm{sing.supp}(\SM)$ and 
$a_1,\dots a_l\in D^\an$ the points 
in $D^\an\subset X^\an$ and set $U\coloneq X\bs D$. 
Recall that there exists a distinguished triangle
\begin{equation}\label{eq:locM}
\rsect_D(\SM)\longrightarrow \SM \longrightarrow
\Gamma_U(\SM)\overset{+1}{\longrightarrow}.
\end{equation}
For $1\leq i\leq l$, let $k_0(a_i)$, $k_1(a_i)\geq0$ be 
the non-negative integers such that 
\begin{equation}
H_D^j(\SM)\simeq \bigoplus_{i=1}^l \SB_{\{a_i\}\vbar X}^{\oplus k_j(a_i)} \quad (j=0,1).
\end{equation}
Then it follows from the distinguished triangle 
\begin{equation}
\Gamma_D(\SM)\longrightarrow\rsect_D(\SM)
\longrightarrow H_D^1(\SM)[-1]\overset{+1}{\longrightarrow},
\end{equation}
the isomorphisms
\begin{equation}
\(H_D^j(\SM)\)^\wedge\simeq \bigoplus_{i=1}^l \(\SO_Y e^{-a_i w}\)^{\oplus k_j(a_i)} \quad (j=0,1)
\end{equation}
and Proposition \ref{prop-K2} (iii) 
that there exists a distinguished triangle 
\begin{equation}\label{eq:holD1}
\bigoplus_{i=1}^l\bigl(\EE_{Y^\an\vbar \var{Y}^\an}^{-\Re a_i w}\bigr)^{\oplus k_1(a_i)}[1]
\longrightarrow Sol_{\var{Y}}^\rmE\bigl(\tl{\rsect_D(\SM)^\wedge}\bigr)
\longrightarrow
\bigoplus_{i=1}^l\bigl(\EE_{Y^\an\vbar \var{Y}^\an}^{-\Re a_i w}\bigr)^{\oplus k_0(a_i)}
\overset{+1}{\longrightarrow}.
\end{equation}
Applying the Fourier transform and 
the functor $\Sol_{\var{Y}}^\rmE(\tl{(\cdot)})$ 
to (\ref{eq:locM}), we also obtain 
\begin{equation}\label{eq:holD2}
Sol_{\var{Y}}^\rmE\bigl(\tl{\Gamma_U(\SM)^\wedge}\bigr)
\longrightarrow
Sol_{\var{Y}}^\rmE(\tl{\SM^\wedge})
\longrightarrow
Sol_{\var{Y}}^\rmE\bigl(\tl{\rsect_D(\SM)^\wedge}\bigr)
\overset{+1}{\longrightarrow}.
\end{equation}
Note that $\Gamma_U(\SM)$ is an algebraic 
meromorphic connection along 
the divisor $D\subset X$.
Then we obtain the following lemma.

\begin{lemma}\label{eq:ineq}
For any $1\leq i\leq l$ we have
\begin{equation}
\mult_{T_{a_i}^\ast X}(\SM) + r_i - \rk(\SM) - \irr_{a_i}(\Gamma_U(\SM))
\geq 0, 
\end{equation}
where $r_i\geq0$ is the regular rank of 
the meromorphic connection $\Gamma_U(\SM)$ at $a_i\in X$.
\end{lemma}

\begin{proof}
By applying Theorem \ref{Thm-T3} 
to $Sol_{\var{Y}}^\rmE\bigl(\tl{\Gamma_U(\SM)^\wedge}\bigr)$ 
and the multiplicity test functor 
in \cite[Section 6.3]{DK18} to 
(\ref{eq:holD1}) and (\ref{eq:holD2}), we 
can easily show that the multiplicity of 
the exponential factor $- a_i w$ of 
$\tl{\SM^\wedge}$ at $\infty\in \var{Y}^\an$ is 
equal to $r_i+k_0(a_i)-k_1(a_i)$.
This in particular implies that 
\begin{equation}
r_i+k_0(a_i)-k_1(a_i)\geq0.
\end{equation}
Similarly to the proof of (\ref{eq:indexthm}), by Kashiwara's index theorem 
we also obtain an equality 
\begin{equation}
k_0(a_i)-k_1(a_i)=\mult_{T_{a_i}^\ast X}(\SM) -\rk(\SM) - \irr_{a_i}(\Gamma_U(\SM)).
\end{equation}
Then the assertion immediately follows. 
\end{proof}

By the multiplicities $N_i$ ($1\leq i\leq l$) and 
$N_{\infty}$ of the meromorphic connection $\Gamma_U(\SM)$ we set 
\begin{align}
\CCirr(\SM)_i &\coloneq
\sum_{f\in N_i^{>0}}N_i(f)\cdot[\Lambda_i^f] \\
&+ \Bigl\{ \mult_{T_{a_i}^\ast X}(\SM) + r_i -
\rk(\SM) - \irr_{a_i}(\Gamma_U(\SM)) \Bigr\} \cdot[T_{\{a_i\}}^\ast X^{\an}] 
\quad (1\leq i\leq l)
\end{align} 
and 
\begin{equation}
\CCirr(\SM)_\infty \coloneq
\sum_{f\in N_\infty^{>0}}N_\infty(f)\cdot[\Lambda_\infty^f] 
=\CCirr(\Gamma_U(\SM))_\infty.
\end{equation}
Then we define the irregular characteristic 
cycle $\CCirr(\SM)$ of the 
holonomic $\SD$-module $\SM\in\Modhol(\SD_X)$ by 
\begin{equation}
\CCirr(\SM)\coloneq\CCirr(\SM)_\infty +
\sum_{i=1}^l\CCirr(\SM)_i.
\end{equation}
Define $g_i, d_i$ ($1\leq i\leq n$) and $d(\SM)= \sum_{i=1}^n d_i$ 
by $\CCirr(\SM)$ as in 
the case of meromorphic connections. 
Then, as in the proof of Lemma \ref{eq:ineq}, we obtain 
the following generalization of Theorem \ref{Thm-T3} to 
holonomic $\SD$-modules.

\begin{theorem}\label{thm:holD_inf}
For any simply connected open sector $V\subset B( b_\infty )^\circ$
along the point $b_{\infty}= \infty\in\var{Y}^{\an}$, 
there exists an isomorphism
\begin{align}
\pi^{-1}\CC_{V}\otimes Sol_{\var{Y}}^{\rmE}(\tl{\SM^\wedge})
\simeq \bigoplus_{i=1}^n
\(\EE_{V\vbar\var{Y}^{\an}}^{\Re g_i}\)^{\oplus d_i}.
\end{align}
In particular, the generic rank $\rk(\SM^\wedge)$ of $\SM^\wedge$ is equal to 
\begin{equation}
    d(\SM)=\biggl\{\sum_{i=1}^l \mult_{T_{a_i}^\ast X}(\SM)\biggr\}
    + e_\infty(\SM).
\end{equation}
\end{theorem}

For a point $b\in Y^\an$ we also set 
\begin{equation}
\CCirr(\SM)^b\coloneq\sum_{f\in N_{\infty,b}^{>0}}N_\infty(f)\cdot[\Lambda_\infty^f]
=\CCirr(\Gamma_U(\SM))^b.
\end{equation}
Then by defining $h_i, e_i$ ($1\leq i\leq m$) and $d(\SM)^b= \sum_{i=1}^m e_i$ 
by $\CCirr(\SM)^b$, Theorem \ref{thm-on-bnd} is also extended to the case of 
holonomic $\SD$-modules as follows.

\begin{theorem}\label{thm:holD_bnd}
For any simply connected  
open sector $W\subset B(b)^\circ$ along the point $b\in Y^{\an}$, 
there exists an isomorphism
\begin{equation}
\pi^{-1}\CC_W\otimes Sol_{\var{Y}}^{\rmE}(\tl{\SM^\wedge}) 
\quad \simeq \quad 
\Bigl\{\bigoplus_{i=1}^m\bigl(
\EE_{W\vbar\var{Y}^{\an}}^{\Re h_i}\bigr)^{\oplus e_i}\Bigr\}
\oplus\bigl(\EE_{W\vbar\var{Y}^{\an}}^0\bigr)^{d(\SM)-d(\SM)^b}.
\end{equation}
\end{theorem}

\begin{remark}
In view of Theorems \ref{thm:holD_inf}, \ref{thm:holD_bnd} and 
\ref{thm-CharCyc}, Corollaries \ref{cor:monod} and \ref{cor:regular} 
hold true for any holonomic $\SD$-module on $X= \CC$.
\end{remark}

\section{Fourier transforms of
$\SD$-modules and rapid decay homology cycles}
\label{uni-sec:T5}

In this section, we first describe the stalks of
the solution complexes of
the Fourier transforms of holonomic $\SD$-modules at
generic points by the theory of
rapid decay homology groups developed by
Bloch-Esnault \cite{BE04a} and Hien \cite{Hi07,Hi09}.
Then we construct their natural bases via
a twisted Morse theory similar to
the one that we used in the proof of Theorem \ref{Thm-T3}. 
For the basic properties of twisted homology groups, 
we refer to Aomoto-Kita \cite{AK11}  
and Pajitnov \cite{Paj06}.

\subsection{Rapid decay homology groups}\label{sec:T6}

In this subsection, we recall
the theory of rapid decay homology groups developed by
\cite{BE04a}, \cite{Hi07} and \cite{Hi09}
in the simplest case of dimension one.
Let $U$ be a smooth algebraic curve
over $\CC$ and $(\SE,\nabla)$
$(\nabla\colon\SE\to\Omega_U^1\otimes_{\SO_U}\SE)$
an algebraic integrable connection on it.
Let $i\colon U\hookrightarrow Z$ be
a smooth compactification of $U$ and
set $D\coloneq Z\bs U$.
Then for the underlying complex manifold
$Z^{\an}$ of $Z$ and the subset $D^{\an}\subset Z^{\an}$
let $\varpi\colon\tl{Z}\to Z^{\an}$ be
the real oriented blow-up of $Z^{\an}$ along $D^{\an}$.
Recall that for a point $a\in D^{\an}$ the subset
$\varpi^{-1}(a)\subset\tl{Z}$ of $\tl{Z}$ is
isomorphic to a circle $S^1$.
For $p\geq0$ and a subset $B\subset\tl{Z}$
denote by $S_p(B)$ the $\CC$-vector space generated
by the piecewise smooth maps $C\colon\Delta^p\to B$ from
the $p$-dimensional simplex $\Delta^p$.
We denote by $\SC_{\tl{Z},\varpi^{-1}(D^{\an})}^{-p}$
the sheaf on $\tl{Z}$ associated to the presheaf
\begin{align}
V \longmapsto S_p\bigl(\tl{Z},(\tl{Z}\bs V)
\cup\varpi^{-1}(D^{\an})\bigr)
= S_p\bigl(\tl{Z}\bigr)/S_p\bigl((\tl{Z}\bs V)
\cup\varpi^{-1}(D^{\an})\bigr).
\end{align}
Now let
\begin{align}
L\coloneq H^{-1}DR_U(\SE)=
\Ker\Bigl\{\nabla^{\an}\colon\SE^{\an}\rightarrow
\Omega_{U^{\an}}^1\otimes_{\SO_{U^{\an}}}\SE^{\an}\Bigr\}
\end{align}
be the sheaf of the horizontal sections of
the analytic connection $(\SE^{\an},\nabla^{\an})$
associated to $(\SE,\nabla)$ and
$\iota\colon U^{\an}\hookrightarrow\tl{Z}$ the inclusion map.
Then $\tl{L}\coloneq\iota_\ast L$ is a local system on $\tl{Z}$.
We define the sheaf $\SC_{\tl{Z},
\varpi^{-1}(D^{\an})}^{-p}(\tl{L})$ of
relative twisted $p$-chains on the pair
$\bigl(\tl{Z},\varpi^{-1}(D^{\an})\bigr)$
with value in $\tl{L}$ by
\begin{align}
\SC_{\tl{Z},\varpi^{-1}(D^{\an})}^{-p}(\tl{L})
\coloneq \SC_{\tl{Z},\varpi^{-1}(D^{\an})}^{-p}
\otimes_{\CC_{\tl{Z}}}\tl{L}.
\end{align}

\begin{definition}[{\cite{BE04a},
\cite{Hi07} and \cite{Hi09}}]\label{def-rdh}
A section $\displaystyle\sigma=\sum_{i=1}^{m}C_i\otimes s_i\,
\in\Gamma(V;\SC_{\tl{Z},\varpi^{-1}(D^{\an})}^{-p}(\tl{L}))$ is
called a rapid decay $p$-chain on $V$
if for any $1\leq i\leq m$ and any point
$q\in C_i(\Delta^p)\cap\varpi^{-1}(D^{\an})\cap V$
the following condition holds :

For a local coordinate $x$ on a neighborhood of
$q$ in $Z$ such that $\varpi (q)=\{x=0\}\subset D^{\an}$ by
taking a local trivialization
\begin{align}
(i_\ast\SE)^{\an}\simeq
\bigoplus_{j=1}^r\SO_{Z^{\an}}(\ast D^{\an}) \ e_j
\end{align}
of the analytic meromorphic connection
$(i_\ast\SE)^{\an}$ with respect to
a basis $e_1,e_2,\dots,e_r\in(i_\ast\SE)^{\an}$ and
setting $\displaystyle s_i=\sum_{j=1}^{r}f_j(x)\cdot\iota_\ast
i^{-1}e_j$ $(f_j(x)\in\iota_\ast\SO_{U^{\an}})$,
for any $1 \leq j\leq r$ and $N\in\ZZ_{>0}$
there exists $M\gg0$ such that
\begin{align}
\abs{f_j(x)}\leq M\abs{x}^N
\end{align}
for any $x\in
\left(C_i(\Delta^p)\bs\varpi^{-1}(D^{\an})\right)\cap V$
with small $\abs{x}$.
\end{definition}

Note that this definition does not depend on
the local coordinate $x$ nor
the local trivialization of $(i_\ast\SE)^{\an}$.
We denote by $\SC_{\tl{Z},\varpi^{-1}(D^{\an})}^{\rd,-p}(\tl{L})$
the subsheaf of $\SC_{\tl{Z},\varpi^{-1}(D^{\an})}^{-p}(\tl{L})$
consisting of rapid decay $p$-chains and set
\begin{align}
S_p^{\rd}(U^{\an};\SE,\nabla) \coloneq
\Gamma(\tl{Z};\SC_{\tl{Z},\varpi^{-1}(D^{\an})}^{\rd,-p}(\tl{L})).
\end{align}
We thus obtain a complex
\begin{align}
S_{\bullet}^{\rd}(U^{\an};\SE,\nabla) :=
\Bigl[ 0 \longleftarrow S_{0}^{\rd}(U^{\an};\SE,\nabla)
\longleftarrow S_{1}^{\rd}(U^{\an};\SE,\nabla)
\longleftarrow \cdots \cdots \Bigr]. 
\end{align}
We call
\begin{align}
H_p^{\rd}(U^{\an};\SE,\nabla) \coloneq
H_p\left[S_{\bullet}^{\rd}(U^{\an};\SE,\nabla)\right]
\quad (p\in\ZZ)
\end{align}
the rapid decay homology groups of
the integrable connection $(\SE,\nabla)$ on $U$.
For $(\SE,\nabla)$ we denote by
$H_{\dR}^p(U;\SE,\nabla)$ $(p\in\ZZ)$
the algebraic de Rham cohomology groups associated to it.
Then we have the following celebrated theorem of
\cite{BE04a}, \cite{Hi07} and \cite{Hi09}.

\begin{theorem}[{\cite{BE04a},\cite{Hi07} and \cite{Hi09}}]
\label{RDH} For any $p\in\ZZ$ there exists a perfect pairing
\begin{align}
H_{\dR}^p(U;\SE,\nabla)\times
H_p^{\rd}(U^{\an};\SE^\ast,\nabla^\ast)\longrightarrow\CC,
\end{align}
where $(\SE^\ast,\nabla^\ast)$ is
the dual connection of $(\SE,\nabla)$ on $U$.
\end{theorem}

By (the proof of) \cite[Proposition 3.4]{ET15},
we have also a purely topological reinterpretation of
$H_p(U^{\an};\SE,\nabla)$ $(p\in\ZZ)$ in terms of
relative twisted homology groups,
which will be used in Section \ref{sec:T8}.

\subsection{Fourier transforms of $\SD$-modules and
rapid decay homologies}\label{sec:T7}
In this subsection, we describe the stalks of
the solution complexes of the Fourier transforms of
holonomic $\SD$-modules at generic points. 
Our results below are inspired by those in 
Hien-Roucairol \cite{HR08}. 
Assume that $X=\CC_z^N$ and $Y=\CC_w^N$ and
regard them algebraic varieties over
$\CC$ endowed with the Zariski topology.
Let $U\subset X$ be an affine open subset and
$j\colon U\hookrightarrow X$ the inclusion map.
Then for an integrable connection
$\SN \in \Modhol(\SD_U)$ on $U$ we set
$\SM\coloneq\bfD j_\ast(\SN)\simeq j_\ast\SN\in\Modhol(\SD_X)$
and consider its Fourier transform
$\SM^\wedge\in\Modhol(\SD_Y)$.
Let
\begin{align}
X \overset{\,p}{\longleftarrow} X\times Y
\overset{q}{\longrightarrow} Y
\end{align}
be the projections.
Then by Lemma \ref{lem-DFS} we have an isomorphism
\begin{align}
\SM^\wedge \simeq \bfD q_\ast
(\bfD p^\ast\SM\Dotimes\SO_{X\times Y}e^{-\inprod{z,w}}).
\end{align}
Let $b\in Y$ be a point such that
$b\notin\mathrm{sing.supp}(\SM^\wedge)$.
Since $\SM^\wedge$ is an integrable connection and
hence regular on a Zariski open neighborhood of $b$,
by \cite[Theorem 7.1.1]{HTT08} there exist isomorphisms
\begin{align}\label{eq-T9}
Sol_Y(\SM^\wedge)_b &\simeq
\Bigl\{DR_Y(\DD_Y(\SM^\wedge))\Bigr\}_b[-N] \notag \\
&\simeq DR_{\{b\}}(\DD_{\{b\}}\circ\bfD i_b^\ast\circ\DD_Y)
(\DD_Y(\SM^\wedge)) \notag \\
&\simeq \bigl[\bfD i_b^\ast(\SM^\wedge)\bigr]^\ast.
\end{align}
Let $\tl{i_b}\colon X\times\{b\}\hookrightarrow X\times Y$ be
the inclusion map and consider the Cartesian diagram
\begin{equation}
\vcenter{
\xymatrix@M=7pt{
X\times\{b\} \ar@{^{(}->}[r]^-{\tl{i_b}}
\ar@{->}[d]_-{q_b} \ar@{}[dr]|\square &
X\times Y \ar@{->}[d]^-{q} \\
\{b\}\ar@{^{(}->}[r]_-{i_b} & Y.       
}}\end{equation}
Let us identity the projection
$q_b\colon X\times\{b\}\rightarrow\{b\}$ with
the map $a_X\colon X\rightarrow \{\pt\}$ to a point.
Then for the map $a_U\colon U\rightarrow\{\pt\}$ to
a point by \cite[Theorem 1.7.3 and Corollary 1.7.5]{HTT08}
we obtain isomorphisms
\begin{align}
\bfD i_b^\ast(\SM^\wedge) &\simeq
\bfD i_b^\ast\bfD q_\ast
(\bfD p^\ast\SM\Dotimes\SO_{X\times Y}e^{-\inprod{z,w}}) \notag \\
&\simeq \bfD a_{X\ast}\bfD\tl{i_b}^\ast
(\bfD p^\ast\SM\Dotimes\SO_{X\times Y}e^{-\inprod{z,w}}) \notag \\
&\simeq \bfD a_{X\ast}(\SM\Dotimes\SO_Xe^{-\inprod{z,b}}) \notag \\
&\simeq \bfD a_{X\ast}
(\bfD j_\ast(\SN)\Dotimes\SO_Xe^{-\inprod{z,b}}) \notag \\
&\simeq \bfD a_{X\ast}\bfD j_\ast
(\SN\Dotimes\SO_Ue^{-\inprod{z,b}}) \notag \\
&\simeq \bfD a_{U\ast}(\SN\Dotimes\SO_Ue^{-\inprod{z,b}}).
\end{align}
Combining this with (\ref{eq-T9}) and Theorem \ref{RDH} we finally get
\begin{align}
Sol_Y(\SM^\wedge)_b &\simeq H^0Sol_Y(\SM^\wedge)_b \notag \\
&\simeq \Bigl[H^0\bfD a_{U\ast}
(\SN\Dotimes\SO_Ue^{-\inprod{z,b}})\Bigr]^\ast \notag \\
&\simeq H_N^{\rd}(U^{\an};\SE_b^\ast,\nabla_b^\ast),
\end{align}
where we used the standard fact that
\begin{align}
H^j\bfD a_{U\ast}(\SN\Dotimes\SO_Ue^{-\inprod{z,b}})
\simeq H^j\rmR\Gamma(U;\SD_{\{\pt\}\leftarrow U}
\Lotimes{\SD_U} (\SN\Dotimes\SO_Ue^{-\inprod{z,b}}))
\quad (j\in\ZZ)
\end{align}
are the algebraic de Rham cohomology groups of
the integrable connection
$(\SE_b,\nabla_b)$ on $U$ defined by
\begin{align}
\SE_b \coloneq \SN\Dotimes\SO_Ue^{-\inprod{z,b}}
\simeq \SN\otimes_{\SO_U}\SO_Ue^{-\inprod{z,b}}
\end{align}
and denoted $(\SE_b^\ast,\nabla_b^\ast)$ its dual connection.
See \cite[Proposition 1.5.28 (i)]{HTT08} for the details.

\subsection{A Morse theoretical construction of
rapid decay cycles in dimension one}\label{sec:T8}

In this subsection, inheriting the notations in
Subsection \ref{sec:T7} we assume also that
the dimension $N$ of $X=\CC^N$ is one and
construct a natural basis of
the rapid decay homology groups
$H_1^{\rd}(U^{\an};\SE_b^\ast,\nabla_b^\ast)$ by
using a twisted Morse theory similar to
the one that we used in the proof of Theorem \ref{Thm-T3}.
For the meromorphic connection
$\SM=\bfD j_\ast\SN\simeq
j_\ast\SN\in\Modhol(\SD_X)$ set
$D\coloneq {\rm sing.supp} ( \SM ) \subset X$ and let
$a_1,a_2, \dots,a_l$ be
the points in $D^{\an}\subset X^{\an}$. Moreover we set
$a_\infty\coloneq\infty\in\var{X}^{\an}$ and
$\tl{D}\coloneq D^{\an}\sqcup\{a_\infty\}$ as in
Section \ref{uni-sec:T2} and inherit
the notations related to
$\SM$ there.
Set $Z\coloneq\var{X}^{\an}$ and let
$\varpi\colon\tl{Z}\rightarrow Z$ be
the real oriented blow-up of $Z$ along
the divisor $\tl{D}=\{a_1,a_2,\dots,a_l,a_\infty\}$.
Let $B(b_{\infty})^\circ \subset Y^{\an}=\CC$ be
a sufficiently small punctured disk in
$\var{Y}^{\an}$ centered at the point
$b_\infty=\infty\in\var{Y}^{\an}$ and
assume that $b\in B(b_{\infty})^\circ$.
Then in order to construct a basis of
the rapid decay homology group
$H_1^{\rd}(U^{\an};\SE_b^\ast,\nabla_b^\ast)$
for the integrable connection
\begin{align}
\SE_b = \SN\otimes_{\SO_U}\SO_Ue^{-bz}
\end{align}
on $U=X \setminus D$ we first consider the problem on
a neighborhood of each point $a_i\in D^{\an}$.
Let $D(a_{i})^\circ \subset V_1\cup\dots\cup V_d$ be the open covering of
the punctured disk $D(a_i)^\circ$ centered at $a_i$ in
the proof of Theorem \ref{Thm-T3} and for the inclusion map
$\iota \colon U^{\an} \hookrightarrow\tl{Z}$ consider
the local system $\tl{L}\coloneq \iota_\ast Sol_U(\SN)$ on $\tl{Z}$.
Here we assume also that 
\begin{align}
B(a_i)^\circ = V_1 \cup V_2 \cup \dots \cup V_d 
\end{align}
and set $B(a_i):= 
B(a_i)^\circ \sqcup \{ b \}$. For the sector $V_j$ set
$\tl{V_j}\coloneq\Int\bigl(\var{i(V_j)}\bigr)\subset\tl{Z}$
and let $f_1, f_2, \dots,f_{m_i} \in N_i^{>0}(V_j)$ be
the exponential factors of $\SM$ on $V_j$.
Moreover for each $f_k$ $(1\leq k\leq m_i)$ we set
\begin{align}
P_k \coloneq \varpi^{-1}(a_i)\cap
\var{ \iota \Set*{z\in V_j}{\Re (f_k(z)-bz)\geq0}} \qquad
\subset \varpi^{-1}(a_i) \simeq S^1 
\end{align}
and
\begin{align}
Q_k\coloneq ( \varpi^{-1} (a_i) \cap \tl{V_j} )
\setminus P_k \quad  \subset  \varpi^{-1} (a_i) \cap \tl{V_j}
\quad \subset \varpi^{-1}(a_i)\simeq S^1.
\end{align}
Then in view of (the proof of) \cite[Proposition 3.4]{ET15}, 
for the reinterpretation of 
$H_1^{\rd}(U^{\an};\SE_b^\ast,\nabla_b^\ast)$
in terms of relative twisted homology groups as in it
on the open subset $\tl{V_j}\subset\tl{Z}$ of $\tl{Z}$,
it suffices to consider the relative twisted chains
\begin{align}
\bigoplus_{k=1}^{m_i}S_p(V_j\cup Q_k ,Q_k; \CC_{\tl{Z}}^{N_i(f_k)} )
\qquad (p \in \ZZ ).
\end{align}
Namely we consider the singular 1-chains on
$V_j\cup Q_k$ modulo those on $Q_k$.
Set $R_i:= \sum_{k=1}^{m_i} N_i(f_k)$. Then we recall that
 for two adjacent sectors $V_j$ and
$V_{j^\prime}$ such that $V_j\cap V_{j^\prime}\neq\emptyset$
after renumbering the exponential factors
$f_1, f_2, \dots,f_{m_i} \in N_i^{>0}(V_j\cap V_{j^\prime})$ we have the condition
\begin{align}
\Re f_1(z)<\Re f_2(z)<\dots<\Re f_{m_i}(z)
\quad (z\in V_j\cap V_{j^\prime})
\end{align}
and the transition matrix $A_{jj^\prime}\in\mathrm{GL}_{R_i}(\CC)$ is
block upper triangular with respect to the
decomposition $R_i= \sum_{k=1}^{m_i} N_i(f_k)$ of $R_i$.
In particular, for $1\leq k_1<k_2\leq m_i$ we have
\begin{align}
Q_{k_1}\cap( \tl{V_j} \cap  \tl{V}_{j^\prime})
\supset Q_{k_2}\cap(  \tl{V_j} \cap  \tl{V}_{j^\prime}).
\end{align}
Since the local system ${\tl{L} \vbar_{\varpi^{-1}B(a_i)}}$
on $\varpi^{-1} B(a_i) \subset \tl{Z}$ is obtained by
gluing the constant sheaves $\CC_{\tl{V}_j}^{R_i}$ on
$\tl{V_j}$ by the transition matrices $A_{jj^\prime}\in
\mathrm{GL}_{R_i}(\CC)$ and the condition on singular
chains for $Q=(Q_1,Q_2, \ldots, Q_{m_i})$ is preserved
by it, we can thus define relative twisted chains
modulo $Q$ that we denote by
\begin{align}
S_p(B(a_i)^{\circ} \cup Q,Q ;\tl{L}) \qquad (p \in \ZZ )
\end{align}
for short. The same is true also over the disk
$B(a_{\infty})= B(a_{\infty})^{\circ} \sqcup \{ a_{\infty} \} 
\subset \var{X}^{\an}$ in $\var{X}^{\an}$
centered at the point $a_\infty=\infty\in\var{X}^{\an}$.
As we do not impose any condition on twisted chains
outside $D= {\rm sing.supp} ( \SM ) \subset X$,
similarly we obtain relative twisted chains
modulo $Q$
\begin{align}
S_p(U^{\an} \cup Q,Q ;\tl{L}) \qquad (p \in \ZZ )
\end{align}
globally defined over $\tl{Z}$. We denote by
$H_p(U^{\an}\cup Q,Q;\tl{L})$ ($p \in \ZZ$)
the homology groups associated to them.
Then by a Mayer-Vietoris exact sequence and
the proof of \cite[Proposition 3.4]{ET15} we obtain
isomorphisms
\begin{align}
H_p^{\rd}(U^{\an};\SE_b^\ast,\nabla_b^\ast)
\simeq H_p(U^{\an}\cup Q,Q;\tl{L}) \qquad
(p \in \ZZ ).
\end{align}
Now our basic idea for the construction of a basis of
$H_1(U^{\an}\cup Q,Q;\tl{L})$ is to use
a Morse theory (with several Morse functions)
as in the proof of Theorem \ref{Thm-T3}.
But this time, for $i=1,2, \ldots,l$ or $\infty$
and the exponential factors $f_1,f_2, \ldots, f_m \in
N_i^{>0}(V_j)$ on a sector $V_j$ along $a_i$
we consider the Morse functions
$\psi_k(z):= -\Re f_k^b(z)$ instead of $\phi_k^b(z)= \Re f_k^b(z)$
on the open sector $V_j \subset B(a_i)^{\circ}$ along $a_i$.
For $t \in \RR$ we set
$W_t^{\psi_k} := \{ z \in V_j \ | \ \psi_{k}(z)<t \}
\subset V_j$. We define also
a function $\psi: B(a_1)^{\circ} \cup  \cdots \cup
B(a_l)^{\circ} \longrightarrow \RR$ by
\begin{align}
\psi (z):= -\Re (bz) \qquad
(z \in B(a_1)^{\circ} \cup  \cdots \cup B(a_l)^{\circ}).
\end{align}
Note that if $1 \leq i \leq l$ and $f_k=0$ for
some $1 \leq k \leq m_i$, we have $r_i>0$ and
$\psi_k(z)= \psi (z)$ on $V_j$.
Moreover, on the closed subset
\begin{align}
K \coloneq X^{\an}\bs\left\{ B(a_1) \cup\dots\cup B(a_l) 
\cup B(a_{\infty}) \right\}
\end{align}
of $X^{\an}$ we use the Morse function
\begin{align}
\eta (z) \coloneq - \Re (bz)+c \quad (z\in K),
\end{align}
where $c>0$ is a sufficient large real number.
For $t \in \RR$ we set
$W_t^{\eta} := \{ z \in K \ | \ \eta (z)<t \} \subset K$.
Then for $i=1,2, \ldots,l$ or $\infty$ and $t \in \RR$
on the open subset $\tl{V_j}\subset\tl{Z}$ of $\tl{Z}$,
we consider the relative twisted chains
\begin{align}
\bigoplus_{k=1}^{m_i}S_p( W_t^{\psi_k} \cup Q_k ,Q_k; 
\CC_{\tl{Z}}^{N_i(f_k)})
\qquad (p \in \ZZ ). 
\end{align}
It also follows from the condition
\begin{align}
\Re f_1(z)<\Re f_2(z)<\dots<\Re f_{m_i}(z)
\quad (z\in V_j\cap V_{j^\prime})
\end{align}
that for $1\leq k_1<k_2\leq m_i$ we have
\begin{align} 
W_t^{\psi_{k_1}}
\cap( \tl{V_j} \cap  \tl{V}_{j^\prime})
\supset W_t^{\psi_{k_2}}
\cap(  \tl{V_j} \cap  \tl{V}_{j^\prime}).
\end{align}
We thus can define relative twisted chains
with value in the local system $\tl{L}$
contained in the level set
$W_t^i= ( W_t^{\psi_1},  W_t^{\psi_2},
\ldots,  W_t^{\psi_{m_i}} )$
modulo $Q=(Q_1,Q_2, \ldots, Q_{m_i})$ that we denote by
\begin{align}
S_p( W_t^i \cup Q,Q ;\tl{L}) \qquad (p \in \ZZ ) 
\end{align}
for short. We can regard them also as
relative twisted chains with value 
in an $\RR$-constructible subsheaf of $\tl{L}$.
For any $t \in \RR$
we thus can define relative twisted chains
with value in the local system $\tl{L}$
contained in the level set $W_t=
(W_t^1, \ldots, W_t^l,
W_t^{\infty}, W_t^{\eta} )$
modulo $Q$
\begin{align}
S_p( W_t \cup Q,Q ;\tl{L}) \qquad (p \in \ZZ ). 
\end{align}
We denote by
$H_p( W_t \cup Q,Q;\tl{L})$ ($p \in \ZZ$)
the homology groups associated to them.
Then as in the proof of Lemma \ref{lem-vanish} for
any $p \in \ZZ$ and
$t \ll 0$ we can easily show the vanishing
\begin{align}
H_p(W_t \cup Q,Q;\tl{L}) \simeq 0
\end{align}
Moreover for any $p \in \ZZ$ and
$t \gg 0$ there exists an isomorphism
\begin{align}
H_p(W_t \cup Q,Q;\tl{L}) \simeq
H_p(U^{\an}\cup Q,Q;\tl{L}). 
\end{align}
Hence we can now apply the arguments
in \cite[Section 5]{ET15} to construct a basis of
$H_1(U^{\an}\cup Q,Q;\tl{L})$ indexed by
the critical points of the Morse functions $\psi_k$ as follows.
As in the proof of Theorem \ref{Thm-T3},
for a point $a_i\in D^{\an}$ let
$\gamma_{i j}(b)\in B(a_i)^{\circ}$ $(1\leq j\leq n_i)$ be
the (non-degenerate) critical points of
the (possibly multi-valued) functions
$\Re(f^b)$ $(f\in N_i^{>0},f\neq0)$ and set
\begin{align}
c_{ij}(b) \coloneq \Re(f^b)(\gamma_{i j}(b))
\ \in\RR \quad (1\leq j\leq n_i).
\end{align}
Then for the point $\gamma_{i j}(b)\in B(a_i)^{\circ}$ and
the (non-zero) exponential factor $f\in N_i^{>0}$
such that $(f^b)^\prime(\gamma_{i j}(b))=0$
there exists a holomorphic Morse coordinate
$\zeta =x +\sqrt{-1} y$ $(x,y \in\RR)$ on
a neighborhood $\Omega_{i j}$ of $\gamma_{i j}(b)$ such that
$\gamma_{i j}(b)=\{\zeta=0\}$ and
\begin{align}
f^b(\zeta)=f^b(\gamma_{i j}(b))+\zeta^2.
\end{align}
This implies that we have
\begin{align}
-\Re(f^b)(\zeta)=-c_{i j}(b)+y^2-x^2
\end{align}
on $\Omega_{i j}$.
Hence we can regard the 1-dimensional smooth submanifold
$S_{i j}\coloneq\{ y=0\}\subset \Omega_{i j}$ of $\Omega_{i j}$
as the stable submanifold of
the gradient flow of our Morse function
\begin{align}
-\Re(f^b) \colon \Omega_{i j}\longrightarrow\RR.
\end{align}
Shrinking it if necessary we may assume that it is
homeomorphic to an open interval
$(-\varepsilon_{i j},\varepsilon_{i j})
\subset\RR$ $(\varepsilon_{i j}>0)$.
Also for the point $a_\infty=\infty\in \var{X}^{\an}$
we define points
$\gamma_{\infty j}(b)\in B(a_{\infty})^\circ
\subset X^{\an}$ $(1\leq j\leq n_\infty)$,
$N_{\infty j}>0$ and $c_{\infty j}(b)\in\RR$
$(1\leq j\leq n_{\infty})$ and obtain
1-dimensional submanifolds
$S_{\infty j}\subset B(a_{\infty})^\circ$ similarly.
For $1 \leq i \leq l$, let
$D_i^\prime\subset X^{\an}$ be
a sufficiently small closed disk in
$X^{\an}$ centered at $a_i$ such that
$D_i^\prime\subset B(a_i)$ and set
$c_i(b)\coloneq\psi(a_i)\in\RR$ and
\begin{align}
S_i \coloneq \partial D_i^\prime\cap
\psi^{-1}\bigl((c_i(b),+\infty)\bigr)\
\subset \partial D_i^\prime\simeq S^1.
\end{align}
Then as in the proof of \cite[Theorem 5.5]{ET15} for
any $t\in\RR$ and $0<\varepsilon\ll1$ there exist sequence
\begin{equation}
\begin{split}
0 \longrightarrow &H_1(W_{t-\varepsilon}\cup Q,Q;\tl{L})
\longrightarrow H_1(W_{t+\varepsilon}\cup Q,Q;\tl{L})
\longrightarrow \\
&\Biggl\{\xyoplus_{(i,j)\colon-c_{i j}(b)=t}
H_1(\var{S_{i j}},\partial S_{i j};
\CC_{X^{\an}}^{N_{i j}})\Biggr\}
\oplus \Biggl\{\xyoplus_{i\colon c_i(b)=t}
H_1(\var{S_i},\partial S_i;
\CC_{X^{\an}}^{r_i})\Biggr\} \longrightarrow 0,
\end{split}
\end{equation}
where in the first sum $\oplus$ the pair $(i,j)$
ranges through the set
\begin{align}
\Set*{(i,j)}{1\leq i\leq l,1\leq j\leq n_i,
-c_{i j}(b)=t} \cup
\Set*{(\infty,j)}{1\leq j\leq n_\infty,
-c_{\infty j}(b)=t}.
\end{align}
Moreover, for each such pair $(i,j)$
we have an isomorphism
\begin{align}
H_1(\var{S_{i j}},\partial S_{i j};
\CC_{X^{\an}}^{N_{i j}}) \simeq \CC^{N_{i j}}
\end{align}
and can take a basis
\begin{align}
\sigma_{i j k}^0 \in
H_1(\var{S_{i j}},\partial S_{i j};\CC_{X^{\an}}^{N_{i j}})
\quad (1\leq k\leq N_{i j})
\end{align}
of $H_1(\var{S_{i j}},\partial S_{i j};
\CC_{X^{\an}}^{N_{i j}})$.
Similarly, for each $1\leq i\leq l$ such that
$c_i(b)=t$ we can take a basis
\begin{align}
\sigma_{i j}^0 \in
H_1(\var{S_i},\partial S_i;\CC_{X^{\an}}^{r_i})
\simeq \CC^{r_i} \quad (1\leq j\leq r_i)
\end{align}
of $H_1(\var{S_i},\partial S_i;\CC_{X^{\an}}^{r_i})$.
From now, we will show that $\sigma_{i j k}^0$ and
$\sigma_{i j}^0$ can be lifted to some elements of
the rapid decay homology group
$H_1(W_{t+\varepsilon}\cup Q,Q;\tl{L})$.
To explain our idea better, first we consider
the case where we have the condition
\begin{align}
\mathrm{ord}_\infty(f)\leq1
\end{align}
for any exponential factor $f\in N_\infty^{>0}$ of
$\SM^{\an}$ at $a_\infty=\infty$.
In this case, by our assumption
$\abs{b}\gg1$ we have $n_\infty=0$.
For $1\leq i\leq l$, $1\leq j\leq n_i$
and $f \in N_i^{>0}$ such that
$(f^b)^{\prime}( \gamma_{ij}(b))=0$,
let $C_{i j}^\circ\subset B(a_i)^\circ$ be
the maximal integral curve of the gradient flow of
the function $\Re(f^b)$ on $B(a_i)^\circ$ such that
$C_{i j}^\circ\supset S_{i j}$.
Then each boundary point of $C_{i j}^\circ$ is
either $a_i$ or contained in
$\partial B(a_i) 
\simeq S^1$
and we can naturally extend
$\sigma_{i j k}^0$ $(1\leq k\leq N_{i j})$ to
some twisted 1-chains
\begin{align}
\sigma_{i j k}^\circ =
\var{C_{i j}^\circ}\otimes s_{i j k}
\quad (s_{i j k}\in\tl{L})
\end{align}
with value in the local system $\tl{L}$ where
$\var{C_{i j}^\circ}$ is the closure in
the real oriented blow-up $\tl{Z}$.
Since we assume here that $\abs{b}\gg 0$,
if a boundary point of $\var{C_{i j}^\circ}$ is
contained in $\partial B(a_i)$
the tangent vector of $\var{C_{i j}^\circ}$ at it is
almost parallel to the vector $\grad\Re(bz)$ and
we can add a half line $(\simeq\RR_{>0})$ in $U^{\an}$
emanating from it to $\var{C_{i j}^\circ}$ to
extend $\sigma_{i j k}^\circ$ to
rapid decay 1-chains $\sigma_{i j k}$ such that
\begin{align}
[\sigma_{i j k}] \in
H_1(W_{t+\varepsilon}\cup Q,Q;\tl{L})
\quad (1\leq k\leq N_{i j}).
\end{align}
Similarly for $1\leq i\leq l$,
at the two boundary points of the curve
$S_i\subset\partial D_i^\prime\simeq S^1$
the tangent vectors of $\var{S_i}$ are parallel to
the vector $\grad\Re(bz)$ and we can add
two half lines $(\simeq\RR_{>0})$ in $U^{\an}$
emanating from them to $\var{S_i}$
to extend $\sigma_{i j}^0$ $(1\leq j\leq r_i)$ to
rapid decay 1-chains $\sigma_{i j}$ such that
\begin{align}
[\sigma_{i j}] \in
H_1(W_{t+\varepsilon}\cup Q,Q;\tl{L})
\quad (1\leq j\leq r_i).
\end{align}
We can remove the condition
\begin{align}
\mathrm{ord}_\infty(f) \leq 1 \quad (f\in N_\infty^{>0})
\end{align}
as follows.
Assume that it does not hold.
Then for $1\leq i\leq l$ and $1\leq j\leq n_i$
if a boundary point of $C_{i j}^\circ$ is contained in
$\partial B(a_i) \simeq S^1$ we first extend
$\sigma_{i j k}^\circ$ $(1\leq k\leq N_{i j})$
to some twisted 1-chains
\begin{align}
\sigma_{i j k}^\prime =
C_{i j}^\prime\otimes s_{i j k}
\quad (s_{i j k}\in\tl{L}),
\end{align}
where $C_{i j}^\prime$ is a curve in
$\tl{Z}$ such that $C_{i j}^\circ\subset C_{i j}^\prime$
and one of the boundary points of it is in
$\partial B(a_{\infty}) \subset X^{\an}$.
Fix $1\leq i\leq l$, $1\leq j\leq n_i$ and
$1\leq k\leq N_{i j}$ and set
\begin{align}
C^\prime\coloneq C_{i j}^\prime,
\quad \sigma^\prime\coloneq \sigma_{i j k}^\prime
= C^\prime\otimes s \quad (s\in\tl{L})
\end{align}
for short.
Then at the boundary point
$C^\prime\cap\partial  B(a_{\infty})$ of $C^\prime=C_{i j}^\prime$
we have a decomposition
\begin{align}
s = \sum_{j=1}^{q}s_j \quad (s_j\in\tl{L})
\end{align}
of the section $s\in\tl{L}$ such that there exist curves
$\Gamma_j\subset  B(a_{\infty})^\circ$ $(1\leq j\leq q)$ in
$ B(a_{\infty})^\circ$ starting from it and ending at
a point in $\varpi^{-1}(\infty)\simeq S^1$ along
which the extensions $\tl{s_j} \in\tl{L}$ of
the sections $s_j \in\tl{L}$ satisfy the rapid decay condition.
Then the twisted 1-chain
\begin{align}
\sigma \coloneq \sigma^\prime+
\sum_{j=1}^{q}\Gamma_j\otimes \tl{s_j}
\end{align}
with value in the local system $\tl{L}$ satisfies
the desired condition
\begin{align}
[\sigma] \in H_1(W_{t+\varepsilon}\cup Q,Q;\tl{L}).
\end{align}
The same arguments can be applied to extend
$\sigma_{\infty j k}^0$ $(resp.\ \sigma_{i j}^0)$ to
rapid decay 1-chains
$\sigma_{\infty j k}$ $(resp.\ \sigma_{i j})$ such that
\begin{align}
[\sigma_{\infty j k}], [\sigma_{i j}]
\in H_1(W_{t+\varepsilon}\cup Q,Q;\tl{L}).
\end{align}
Considering all $t\in\RR$,
we thus obtain elements
$[\sigma_{i j k}]$ $(1\leq i\leq l,
1\leq j\leq n_i, 1\leq k\leq N_{i j})$,
$[\sigma_{\infty j k}]$ $(1\leq j\leq n_\infty,
1\leq k\leq N_{\infty j})$ and
$[\sigma_{i j}]$ $(1\leq i\leq l, 1\leq j \leq r_i)$
of the rapid decay homology group
\begin{align}
H_1^{\rd}(U^{\an};\SE_b^\ast,\nabla_b^\ast)
\simeq H_1(U^{\an}\cup,Q, Q;\tl{L}).
\end{align}

\begin{theorem}\label{thm-T10}
The elements
$[\sigma_{i j k}]$ $(1\leq i\leq l,
1\leq j\leq n_i, 1\leq k\leq N_{i j})$,
$[\sigma_{\infty j k}]$ $(1\leq j\leq n_\infty,
1\leq k\leq N_{\infty j})$ and
$[\sigma_{i j}]$ $(1\leq i\leq l, 1\leq j \leq r_i)$ that
we constructed above form a basis of
rapid decay homology group
\begin{align}
H_1^{\rd}(U^{\an};\SE_b^\ast,\nabla_b^\ast)
\simeq H_1(U^{\an}\cup,Q, Q;\tl{L}).
\end{align}
Moreover for any $p \not= 1$ we have the vanishing
\begin{align}
H_p^{\rd}(U^{\an};\SE_b^\ast,\nabla_b^\ast)
\simeq H_p(U^{\an}\cup,Q, Q;\tl{L}) \simeq 0.
\end{align}
\end{theorem}

\begin{proof}
The proof is similar to that of \cite[Theorem 5.5]{ET15} and
relies on the twisted Morse theory for the Morse functions
$-\Re f^b\colon B(a_i)^\circ\longrightarrow\RR$
$(1\leq i\leq l,f\in N_i^{>0})$,
$-\Re f_\infty^b\colon B(a_{\infty})^\circ\longrightarrow\RR$
$(f\in N_\infty^{>0})$ and
$\eta \colon K\rightarrow\RR$.
Recall that for $t\ll0$ we have the vanishing
\begin{align}
H_1(W_{t}\cup Q,Q;\tl{L}) \simeq 0.
\end{align}
Then the proof proceeds as in that of
Theorem \ref{Thm-T3} except for the point that
if for some $t_1<t_2$ we know
\begin{align}
H_1(W_{t_1}\cup Q,Q;\tl{L})
\simto H_1(W_{t_2}\cup Q,Q;\tl{L})
\end{align}
by some geometric observation we do not have to
calculate $H_1(W_{t}\cup Q,Q;\tl{L})$
for each $t\in(t_1,t_2)$.
Namely we can skip the times $t\in\RR$
such that $\partial W_t$ is tangent to the circles
$\partial B(a_i)$ ($1 \leq i \leq l$) or $\partial B(a_{\infty})$.
In this sense, the proof is much simpler than
that of Theorem \ref{Thm-T3}.
This completes the proof.
\end{proof}

\end{document}